\documentclass[reqno,12pt,a4paper]{amsart}
\usepackage{amssymb,amsmath,amstext,amsthm,amsfonts,amscd,bbold,wasysym}
\usepackage{euscript,a4wide}
\usepackage{graphicx}
\usepackage[shortlabels]{enumitem} 

\usepackage[colorlinks=true,backref]{hyperref}

\theoremstyle{plain}
\newtheorem{maintheorem}{Theorem}

\newtheorem{maincorollary}[maintheorem]{Corollary}

\newtheorem{theorem}{Theorem}[section]
\newtheorem{proposition}[theorem]{Proposition}
\newtheorem{corollary}[theorem]{Corollary}
\newtheorem{lemma}[theorem]{Lemma}
\newtheorem{claim}[theorem]{Claim}
\newtheorem{remark}[theorem]{Remark}

\theoremstyle{definition}
\newtheorem{definition}{Definition}
\newtheorem{conjecture}{Conjecture}
\newtheorem{example}{Example}

\newcommand{\RR}{{\mathbb R}}
\newcommand{\EE}{{\mathbb E}}
\newcommand{\DD}{{\mathbb D}}

\newcommand{\ZZ}{{\mathbb Z}}

\newcommand{\sS}{{\mathbb S}}

\def \TT {{\mathbb T}}

\newcommand{\fX}{{\mathfrak{X}}}

\newcommand{\cA}{{\mathcal A}}

\newcommand{\cC}{{\mathcal C}}
\newcommand{\cF}{{\mathcal F}}

\newcommand{\cO}{{\mathcal O}}

\newcommand{\cW}{{\mathcal W}}

\newcommand{\F}{\EuScript{F}}
\newcommand{\D}{\EuScript{D}}

\newcommand{\G}{\EuScript{G}}

\newcommand{\K}{\EuScript{K}}

\newcommand{\vfi}{{\varphi}}
\renewcommand{\epsilon}{\varepsilon}

\newcommand{\qand}{\quad\text{and}\quad}
\newcommand{\wh}{\widehat}
\newcommand{\wt}{\widetilde}

\DeclareMathOperator{\diag}{diag}
\DeclareMathOperator{\dom}{dom}
\DeclareMathOperator{\spec}{sp}
\DeclareMathOperator{\dist}{dist}
\DeclareMathOperator{\supp}{supp}

\DeclareMathOperator{\sing}{Sing}       

\DeclareMathOperator{\lip}{Lip}

\DeclareMathOperator{\m}{Leb}
\DeclareMathOperator{\close}{Closure}

\DeclareMathOperator{\tr}{Tr}

\def \Leb {\operatorname{Leb}}
\def \leb {\operatorname{Leb}}

  \def \vfi {\varphi}

\title[Physical measures for flows]{Physical
  measures for mostly sectionally expanding flows}

\date{\today}

\thanks{V.A. was partially supported by CNPq-Brazil (grant
  304047/2023-6). L.S. was partially supported by FAPERJ-Funda\c c\~ao
  Carlos Chagas Filho de Amparo \`a Pesquisa do Estado do Rio de
  Janeiro Projects APQ1-E-26/211.690/2021 SEI-260003/015270/2021 and
  JCNE-E-26/200.271/2023 SEI-260003/000640/2023, by Coordena\c c\~ao
  de Aperfei\c coamento de Pessoal de N\'ivel Superior CAPES — Finance
  Code 001 and PROEXT-PG project Dynamic Women - Din\^amicas,
  CNPq-Brazil (grant Projeto Universal 404943/2023-3).
  S.S. was partially supported by CNPq-Brazil Doctoral schoolarship.}

\author{Vitor Araujo, Luciana Salgado and S\'ergio Sousa}

\address[V.A.]{Instituto de Matem\'atica e Estat\'{\i}stica,
  Universidade Federal da Bahia, Av. Ademar de Barros s/n,
  40170-110 Salvador, Brazil.}
\email{vitor.araujo.im.ufba@gmail.com, vitor.d.araujo@ufba.br}
\urladdr{https:///sites.google.com/view/vitor-araujo-ime-ufba}

\address[L.S.]{Universidade Federal do Rio de Janeiro, Instituto de
   Matem\'atica\\
   Avenida Athos da Silveira Ramos 149 Cidade Universit\'aria, P.O. Box 68530, 
   21941-909 Rio de Janeiro-RJ-Brazil }
 \email{lsalgado@im.ufjr.br, lucianasalgado@ufrj.br}
 \urladdr{http://www.im.ufrj.br/~lsalgado}
 
 \address[S.S.]{Universidade Federal do Rio de Janeiro, Instituto de
   Matem\'atica\\
   Avenida Athos da Silveira Ramos 149 Cidade Universit\'aria, P.O. Box 68530, 
   21941-909 Rio de Janeiro-RJ-Brazil }
 \email{sergio.sousa@ufrj.br}



\begin{document}

\begin{abstract}
  We prove that a partially hyperbolic attracting set for a $C^2$
  vector field, having slow recurrence to equilibria, supports an
  ergodic physical/SRB measure if, and only if, the trapping region
  admits non-uniform sectional expansion on a positive Lebesgue
  measure subset. Moreover, in this case, the attracting set supports
  at most finitely many ergodic physical/SRB measures, which are also
  Gibbs states along the central-unstable direction.

  This extends to continuous time systems a similar well-known result
  obtained for diffeomorphisms, encompassing the presence of
  equilibria accumulated by regular orbits within the attracting set.
  In codimension two the same result holds, assuming only the
  trajectories on the trapping region admit a sequence of times with
  asymptotical sectional expansion, on a positive volume subset.
  

  We present several examples of application, including the existence
  of physical measures for asymptotically sectional hyperbolic
  attracting sets, and obtain physical measures in an alternative
  unified way for many known examples: Lorenz-like and Rovella
  attractors, and sectional-hyperbolic attracting sets (including the
  multidimensional Lorenz attractor).
  

\end{abstract}
 
\maketitle

\tableofcontents

\section{Introduction and statement of results}
\label{sec:introduction}

Much of the recent progress in Dynamics is a consequence of a
probabilistic approach to the understanding of complicated dynamical
systems, where one focuses on the statistical properties of ``typical
orbits'', in the sense of large volume in the ambient space. We deal
here with flows $\phi_t:M\to M$ on compact manifolds. The most basic
statistical data are the time averages
$T^{-1} \int_0^T \delta_{\phi_t (z)} \,dt$, where $\delta_w$
represents the Dirac measure at a point $w$.  Birkhoff's Ergodic
Theorem asserts that time averages admit asymptotic limits in the
weak$^*$ topology at almost every point $z$ with respect to any
invariant probability $\mu$. That is, for every continuous observable
$\psi:M\to\RR$ there exists a subset $E\subset M$ of full measure
$\mu(E)=1$ so that
$\lim_{T\nearrow\infty} \frac1T \int_0^T \psi(\phi_t (z)) \,dt =
\wt{\psi}(z)$ is well-defined for each $z\in E$.  Moreover, if the
measure is ergodic, then the time average coincides with the space
average, that is, $\wt{\psi}(z)=\int\psi\,d\mu$ for $z\in E$. However,
many invariant measures are singular with respect to volume in
general, and so the Ergodic Theorem is not enough to understand the
behavior of positive volume (Lebesgue measure) sets of orbits.

A \emph{physical measure} is an invariant probability measure for
which time averages exist and coincide with the space average, for a
set of initial conditions with positive Lebesgue measure, i.e. in the
weak$^*$ topology of convergence of probability measures we have
$$
B(\mu):=\left\{z\in M: \lim_{T\nearrow\infty} \frac1T \int_0^T
  \delta_{\phi_t (z)} \,dt = \mu \right\}
\;\text{with  } \Leb(B(\mu))>0.
$$
This set is the \emph{basin} of the measure. Sinai, Ruelle and Bowen
introduced this notion about fifty years ago, and proved that, for
uniformly hyperbolic (Axiom~A) diffeomorphisms and flows, time averages
exist for Lebesgue almost every point and coincide with one of
finitely many physical measures; see~\cite{BR75,Si72}.

The problem of existence (and finiteness) of physical measures, beyond
the Axiom A setting, remains a main goal of Dynamics. The construction
of the so called Gibbs $u$-states, by Pesin and Sinai in \cite{PS82},
was the beginning of the extension of the Sinai, Ruelle and Bowen
ideas to partially hyperbolic systems, a fruitful generalization of
the notion of uniform hyperbolicity, which more recently was shown to
encompass Lorenz-like or singular-hyperbolic flows \cite{MPP99}
and to be a consequence of robust transitivity: see the surveys
\cite{AraPac2010s,BDV2004} for much of the progress obtained so far and
the recent extensions to higher-dimensional flows
\cite{araujo_2021,LeplYa17}.

The papers of Alves, Bonatti and Viana~\cite{ABV00,BoV00}, and
Dolgopyat~\cite{Do2000} are of special interest to us here since they
prove existence and finiteness of physical measures for partially
hyperbolic diffeomorphisms, which are also $u$-Gibbs states, under the
assumption that the central direction is either ``mostly contracting''
\cite{BoV00, Do2000} or ``non-uniformly expanding'' \cite{ABV00}. For
local diffeomorphisms with a non-flat critical or singular set, the
authors in~\cite{ABV00} show that a slow recurrence condition is
sufficient to obtain an absolutely continuous invariant probability
measure when the system is non-uniformly expanding.

Here, we extend the results of~\cite{ABV00} to attracting sets for
smooth vector fields with a dominated splitting and a sequence of
times with asymptotical sectional expansion on a positive volume
subset.  That is, every transversal direction to the vector field
along the center subbundle has a positive Lyapunov exponent in a weak
sense. In this setting partial hyperbolicity is natural. We
additionally obtain finitely many physical/SRB measures for the flow,
which are also $cu$-Gibbs states, whose ergodic basins cover the set
of non-uniform sectional expanding orbits (except for a subset of
volume zero), if we also have slow recurrence to the equilibria of the
vector field.

The properties of continuous time dynamics enable us to show that weak
asymptotical sectional expansion on a positive volume subset is a
necessary and sufficient condition for existence of ergodic
physical/SRB measures for partially hyperbolic attracting sets with
central direction of codimension two.
We present several examples of application, including the existence of
physical measures for weak asymptotically sectional hyperbolic
attracting sets.

Known examples satisfying these conditions are, besides hyperbolic
(Axiom A) flows~\cite{BR75}, all singular-hyperbolic attracting sets
for $C^2$ smooth flows~\cite{AraPac2010s} (including the Lorenz
attractor), the contracting Lorenz attractor (also known as the
Rovella attractor~\cite{Ro93}) and all sectional-hyperbolic attracting
sets~\cite{araujo_2021} (including the multidimensional Lorenz
attractor~\cite{BPV97}). Our results provide an alternative unified
way to obtain physical/SRB measures for these known systems and for
many others.


\subsection{Statements of the results}
\label{sec:statements-results}

Let $M$ be a compact connected manifold with dimension $\dim M=m$,
endowed with a Riemannian metric, induced distance $d$ and volume form
$\m$. Let $\fX^r(M)$, $r\ge1$, be the set of $C^r$ vector fields on $M$
endowed with the $C^r$ topology and denote by $\phi_t$ the flow
generated by $G\in\fX^r(M)$.

\subsubsection{Preliminary definitions}
\label{sec:prelim-definit}

We say that $\sigma\in M$ with $G(\sigma)=0$ is an {\em equilibrium}
or \emph{singularity}. In what follows we denote by $\sing(G)$ the
family of all such points. We say that an equilibrium
$\sigma\in\sing(G)$ is \emph{hyperbolic} if all the eigenvalues of
$DG(\sigma)$ have non-zero real part.

An \emph{invariant set} $\Lambda$ for the flow $\phi_t$, generated by
the vector field $G$, is a subset of $M$ which satisfies
$\phi_t(\Lambda)=\Lambda$ for all $t\in\RR$. A point $p\in M$ is
\emph{periodic} for the flow $\phi_t$ generated by $G$ if $G(p)\neq 0$
and there exists $\tau>0$ so that $\phi_\tau(p)=p$; its orbit
$\cO_G(p)=\phi_{\RR}(p)=\phi_{[0,\tau]}(p)=\{\phi_tp: t\in[0,\tau]\}$
is a \emph{periodic orbit}, an invariant simple closed curve for the
flow.  An invariant set is \emph{nontrivial} if it is not a finite
collection of periodic orbits and equilibria.

Given a compact invariant set $\Lambda$ for $G\in \fX^r(M)$, we say
that $\Lambda$ is \emph{isolated} if there exists an open set
$U\supset \Lambda$ such that
$ \Lambda =\bigcap_{t\in\RR}\close{\phi_t(U)}$.  If $U$ can be chosen
so that $\close{\phi_t(U)}\subset U$ for all $t>0$, then we say that
$\Lambda$ is an \emph{attracting set} and $U$ a \emph{trapping region}
(or \emph{isolating neighborhood}) for
$\Lambda=\Lambda_G(U)=\cap_{t>0}\close{\phi_t(U)}$.


An \emph{attractor} is a transitive attracting set, that is,
an attracting set $\Lambda$ with a point $z\in\Lambda$ so that its
$\omega$-limit
$
  \omega(z):=\left\{y\in M: \exists t_n\nearrow\infty\text{  s.t.
  } \phi_{t_n}z\xrightarrow[n\to\infty]{}y \right\}
$
coincides with $\Lambda$.

\subsubsection{Partial hyperbolic attracting sets for vector fields}
\label{sec:part-hyperb-diff}

Let $\Lambda$ be a compact invariant set for $G \in \fX^r(M)$.  We say
that $\Lambda$ is {\em partially hyperbolic} if the tangent bundle
over $\Lambda$ can be written as a continuous $D\phi_t$-invariant
Whitney sum $ T_\Lambda M=E^s\oplus E^{cu}, $ where
$d_s=\dim(E^s_x)\ge1$ and $d_{cu}=\dim (E^{cu}_x)\ge2$ for $x\in\Lambda$,
and there exists a constant $\lambda >0$ such that for all
$x \in \Lambda$, $t\ge0$, we have\footnote{For some choice of the
  Riemannian metric on the manifold, see e.g. \cite{Goum07}. Changing
  the metric does not change the rate $\lambda$ but might introduce
  the multiplication by a constant.}
\begin{itemize}
  \item domination of the splitting:
$\|D\phi_t | E^s_x\| \cdot \|D\phi_{-t} | E^{cu}_{\phi_tx}\| \le e^{-\lambda t}$;
\item uniform contraction along $E^s$:
  $\|D\phi_t | E^s_x\| \le e^{-\lambda t}$.
\end{itemize}
We refer to $E^s$ as the stable bundle and to $E^{cu}$ as the
center-unstable bundle.  
\begin{lemma}{\cite[Lemma 3.2]{arsal2012a}}
  \label{le:flow-center}
  Let $\Lambda$ be a compact invariant set for $G$.
  \begin{enumerate}
  \item Given a $D\phi_t$-invariant and continuous splitting
    $T_\Lambda M = E\oplus F$ such that $E$ is uniformly contracted,
    then $G(x)\in F_x$ for all $x\in \Lambda$.
  \item Assuming that $\Lambda$ is non-trivial and has a continuous
    $D\phi_t$-invariant \emph{and dominated} splitting
    $T_\Lambda M = E\oplus F$ such that $G(x)\in F_x$ for all
    $x\in\Lambda$, then $E$ is a uniformly contracted subbundle.
  \end{enumerate}
\end{lemma}

A {\em partially hyperbolic attracting set} is a partially hyperbolic
set that is also an attracting set.

\begin{remark}\label{rmk:domsplit2parthyp}
  In the flow setting, a dominated splitting becomes partially
  hyperbolic whenever the flow direction is contained in the
  central-unstable bundle $G\in E^{cu}$, from
  Lemma~\ref{le:flow-center}. Thus, this inclusion is equivalent to
  partial hyperbolicity. Since the flow direction is invariant, then
  partial hyperbolicity is the natural setting to study invariant sets
  (which are not composed only of equilibria) for flows with a
  dominated splitting.
\end{remark}


\subsubsection{Singular/sectional-hyperbolicity}
\label{sec:singul-hyperb-asympt}

The center-unstable bundle $E^{cu}$ is \emph{volume expanding} if
there exists $K,\theta>0$ such that
$|\det(D\phi_t| E^{cu}_x)|\geq K e^{\theta t}$ for all $x\in \Lambda$,
$t\geq 0$.

We say that a compact nontrivial invariant set $\Lambda$ is a
\emph{singular hyperbolic set} if all equilibria in $\Lambda$ are
hyperbolic, and $\Lambda$ is partially hyperbolic with volume
expanding center-unstable bundle.  A singular hyperbolic set which is
also an attracting set is called a {\em singular hyperbolic attracting
  set}.

We say that $E^{cu}$ is \emph{($2$-)sectionally expanding} if there are
positive constants $K , \theta$ such that for every $x \in \Lambda$
and every $2$-dimensional linear subspace $L_x \subset F_x$ one has
$|\det(D\phi_t| L_x)|\geq K e^{\theta t}$ for all $t\geq 0$.  A
\emph{sectional-hyperbolic (attracting) set} is a partially hyperbolic
(attracting) set whose central subbundle is sectionally expanding.

\subsubsection{Asymptotical sectional-hyperbolicity}
\label{sec:asympt-singul-hyperb}

A compact invariant partially hyperbolic set $\Lambda$ of a vector
field $G$ whose equilibria are hyperbolic, is \emph{asymptotically
  sectional-hyperbolic} (ASH) if the center-unstable subbundle is
eventually asymptotically sectional expanding outside the stable
manifold of the equilibria. That is, there exists $c_*>0$ so that
  \begin{align}\label{eq:assecexp}
    \limsup_{T\nearrow\infty}\frac1T
    \log|\det(D\phi_T\mid_{F_x})|\ge c_*
  \end{align}
  for every
  $x\in\Lambda\setminus \cup\{W^s_\sigma:\sigma\in\sing_\Lambda(G)\}$
  and each $2$-dimensional linear subspace $F_x$ of $E^{cu}_x$, where
  we write $\sing_\Lambda(G)=\sing(G)\cap\Lambda$ and
  $W^s_\sigma=\{x\in M: \lim_{t\to+\infty}\phi_tx = \sigma\}$ is the
  \emph{stable manifold} of the hyperbolic equilibrium $\sigma$. It is
  well-known that $W^s_\sigma$ is a immersed submanifold of $M$; see
  e.g.~\cite{PM82}. This implies that all transverse directions to the
  vector field along the center-unstable subbundle have positive
  Lyapunov exponent; that is, if $v\in E^{cu}_x\setminus(\RR\cdot G)$,
  then
  $\chi(x,v):=\limsup_{t\to+\infty}\log\|D\phi_t(x)v\|^{1/t}\ge
  c_*>0$; see e.g. Theorem~\ref{thm:NUSEvar};
  Remark~\ref{rmk:wMASE-wPSL} and Conjecture~\ref{conj:PSL}.

  \begin{lemma}[Hyperbolic Lemma]\label{le:hyplemma}
    Every compact invariant subset $\Gamma$ without equilibria
    contained in a asymptotically sectional-hyperbolic set is
    uniformly hyperbolic.
  \end{lemma}

  We say that an invariant compact subset $\Gamma$ is
  \emph{(uniformly) hyperbolic} if $\Gamma$ is partially hyperbolic
  and the central-unstable bundle admits a continuous splitting
  $E^{cu}=(\RR\cdot G)\oplus E^u$, with $\RR\cdot G$ the
  one-dimensional invariant flow direction and $E^u$ a uniformly
  expanding subbundle. That is, we get the following dominated
  splitting $T_\Gamma M= E^s\oplus(\RR\cdot G)\oplus E^u$ into
  three-subbundles; see e.g.~\cite{fisherHasselblatt12}.
  
  \begin{proof}[Proof of Lemma~\ref{le:hyplemma}]
    See e.g.~\cite[Proposition 1.8]{MPP04} for sectional-hyperbolic
    sets; and~\cite[Theorem 2.2]{SmartinVivas20} for the
    asymptotically sectional-hyperbolic case.
  \end{proof}


\subsection{Non-uniform sectional expansion}
\label{sec:non-uniformly-expand-1}

Let us fix $G\in\fX^2(M)$ endowed with a partially hyperbolic
attracting set $\Lambda=\Lambda_G(U)$ with a trapping region $U$.
Then we can take a continuous extension
$T_UM=\wt{E^s}\oplus \wt{E^{cu}}$ of $T_\Lambda M=E^s\oplus E^{cu}$
and for small $a>0$ find center unstable and stable cones
\begin{align}
  \label{eq:cucone}
  C^{cu}_a(x)
  &=
    \{ v=v^s+v^c : v^s\in \wt{E^s}_x, v^c\in\wt{E^{cu}}_x, x\in
    U, \|v^s\|\le a\|v^c\|\}, \qand
  \\
  C^s_a(x)
  &=
    \{ v=v^s+v^c : v^s\in \wt{E^s}_x, v^c\in\wt{E^{cu}}_x, x\in
    U, \|v^c\|\le a\|v^s\|\}, \nonumber
\end{align}
which are invariant in the following sense
\begin{align}\label{eq:coneinv}
  D\phi_t(x)\cdot C^{cu}_a(x)\subset C^{cu}_a(\phi_t(x))
  \qand
  D\phi_{-t}\cdot C^{s}_a(x)\supset C^{s}_a(\phi_{-t}(x)),
\end{align}
for all $x\in\Lambda$ and $t>0$ so that $\phi_{-s}(x)\in U$ for all
$0<s\le t$; see Subsection~\ref{sec:ext-stable-bundle}. We can assume
that $\wt{E^{cu}_x}\subset C^{cu}_a(x)$ still contains the flow
direction $G(x)$ for each $x\in U$. Otherwise, since
$G(x)\in E^{cu}_x$ for $x\in\Lambda$ by Lemma~\ref{le:flow-center}(1),
then $G(x)\in C^{cu}_a(x)$ and we can set
$N^{cu}_x:=\wt{E^{cu}}\cap G(x)^\perp$ and
$\wh{E^{cu}}:=G(x)\oplus N^{cu}_x$ for all $x\in U$. Then
$\wh{E^{cu}}$ contains the flow direction by construction, is
continuous and $\wh{E^{cu}_x}\subset C^{cu}_a(x), x\in U$.  We can
also assume, without loss of generality according to~\cite{ArMel17},
that the continuous extension of the stable direction $E^s$ of the
splitting is still $D\phi_t$-invariant and
$\wt{E^s_x}\subset C^s_a(x), x\in U$.
In what follows, we keep the notation $T_U M=E^s\oplus E^{cu}$ and
write $N_x^{cu}=E^{cu}_x\cap G(x)^{\perp}, x\in U$.



In what follows we write $f:=\phi_1$ for the time-$1$ diffeomorphism
induced by the flow. We say that the attracting set $\Lambda$ is
\emph{non-uniform $2$-sectionally expanding} (NU2SE) if if there
exists $c_0>0$ so that
\begin{align}\label{eq:NU2SE}
  \Omega=\left\{
  x\in U: \limsup_{n\nearrow\infty}\frac1n\sum\nolimits_{i=0}^{n-1}
  \log
  \big\|\wedge^2 (Df\mid_{E^{cu}_{f^i x}})^{-1}\big\|
  \le - c_0
  \right\}
  \qand \m(\Omega)>0.
\end{align}
Non-uniform sectional expansion ensures the existence of 
at most finitely many physical/SRB measures whose
ergodic basins cover the (NU2SE) subset $\leb\bmod0$.

\begin{maintheorem}\label{mthm:discretefabv0}
  Let $G\in\fX^2(M)$ be a vector field with a partially hyperbolic
  attracting set $\Lambda=\Lambda_G(U)$ with no equilibria
  $\sing_\Lambda(G)=\emptyset$.  Then we have (NU2SE) on
  $\Omega\subset U$ if, and only if, there are finitely many ergodic
  hyperbolic physical/SRB measures whose basins cover $\m$-a.e. point
  of $\Omega$:
  $\m\Big(\Omega\setminus \big(B(\mu_1)\cup\dots\cup
  B(\mu_p)\big)\Big)=0$.
\end{maintheorem}

Here, hyperbolicity of an invariant probability measure means
\emph{non-uniform hyperbolicity}. More precisely, we set the forward
Lyapunov exponent of $v\in T_x M\setminus\{\vec0\}$ as
\begin{align*}
  \chi^+(x,v):=\limsup\nolimits_{t\nearrow\infty}\log\|D\phi_t(x)v\|^{1/t}.
\end{align*}
The function $\chi(x,\cdot)$ admits finitely many values only
$\chi^+_1(x)>...>\chi^+_{p(x)}(x)$ on $T_x M\setminus\{0\}$ (the
\emph{Lyapunov Spectrum} of $\mu$) and generates a filtration
$ 0\subsetneq V_{p(x)}(x) \subsetneq \cdots \subsetneq V_{1}(x)=T_xM$
with $V_i(x)=\{ v\in T_xM, \ \chi^+(x,v)\leq \chi^+_i(x)\}$. The
functions $p(x), \chi^+_i(x)$ are $\phi_t$-invariant; the vector
subspaces $V_i(x)$, $i=1,...,p(x)$ are $D\phi_t$-invariant; and each of
them depend Borel measurably on $x$.

An invariant ergodic probability measure $\mu$ supported in $\Lambda$ is
\emph{hyperbolic}, if the tangent bundle over $\Lambda$ splits into a
sum $T_z M = E^s_z\oplus (\RR\cdot G(z)) \oplus F_z$ of invariant
subspaces defined for $\mu$-a.e.  $z\in \Lambda$ and depending
measurably on the base point $z$;
$\RR\cdot G(z)$ is the flow direction (with zero Lyapunov exponent);
$\RR\cdot G(z)\oplus F_z=E^{cu}_z$ and $F_z$ is the direction with
positive Lyapunov exponents, that is, it satisfies
$\lim_{t\to+\infty}\log\big\|(D\phi_t\mid_{F_z})^{-1}\big\|^{-1/t}>0$.
Naturally, all directions along $E^s_z$ have negative Lyapunov
exponents
$\lim_{t\to+\infty}\log\big\|D\phi_t\mid_{E^s_z})\big\|^{1/t}<0$. In
terms of the Lyapunov Spectrum this means that
$\chi^+_{d_{cu}+1}(z)<0 = \chi^+_{d_{cu}}(z)<\chi^+_{d_{cu}-1}(z)$.

\subsubsection{Physical/SRB measures and $cu$-Gibbs states}
\label{sec:physicalsrb-measures}

In the uniformly hyperbolic setting, it is well known that physical
measures $\mu$, for hyperbolic attractors of $C^2$ diffeomorphisms $g$,
admit a disintegration into conditional measures along the unstable
manifolds of almost every point which are absolutely continuous with
respect to the induced Lebesgue measure on these submanifolds, see
\cite{Bo75,BR75,PS82}. By Leddrappier-Young characterization of
measures satisfying (Pesin's) Entropy Formula~\cite{LY85}, this is
equivalent to
\begin{align}
  \label{eq:Pesin}
  h_\mu(g)=\int\Sigma^+\,d\mu=\int \log|\det Dg\mid_{E^u}|\,d\mu > 0,
\end{align}
where $E^u$ is the unstable invariant subbundle over the hyperbolic
attractor, and $\Sigma^+(z)=\sum_{i\le\dim(E^u)} \chi_i^+(z)$ is the
sum of positive Lyapunov exponent (with multiplicities). In the
hyperbolic setting for diffeomorphisms, condition~\eqref{eq:Pesin}
means that $\mu$ is a $u$-Gibbs state. These measures are known as
\emph{Sinai-Ruelle-Bowen (SRB)} measures.

In our setting, existence of unstable manifolds is guaranteed by the
hyperbolicity of physical measures: the \emph{strong-unstable
  manifolds} $W^{uu}_z$ are the ``integral manifolds'' tangent to
$F_z$ defined by
$
  W^{uu}_z
  =
  \big\{ y\in M: \lim\nolimits_{ t\to -\infty}
  d\big(\phi_t(y),\phi_t(z)\big)=0
  \big\}
$
and exist for $\mu$-a.e. $z$ with respect to the measures obtained in
Theorem~\ref{mthm:discretefabv0}.

The weak-unstable manifolds $W^{cu}_z$ are the saturation
$\phi_{(-1,1)}(W^{uu}_z)$ of $W^{uu}_z$ by the flow, and are tangent
to $E^{cu}_z$, $\mu$-a.e. $z$.  The sets $W^{cu}_z$ are embedded
sub-manifolds in a neighborhood of $z$ which, in general, depend only
measurably (including its size) on the base point $z\in\Lambda$.  We
note that, since $\Lambda$ is an attracting set, then
$W^{cu}_z\subset\Lambda$ where defined\footnote{For if
  $y\in W^{uu}_z\cap U$ and $z\in\Lambda$, then
  $d(\phi_{-t} y,\phi_{-t} z)\to0$ for $t\nearrow\infty$. Thus
  $\phi_{-t} y\subset U$ for all $t\ge0$, that is,
  $y\in \cap_{t\ge0} \phi_t(U)=\Lambda$.}.

The arguments of our proofs, adapted from~\cite{ABV00}, enable us to
obtain not only ergodic hyperbolic physical invariant probability
measures, but also the condition corresponding to~\eqref{eq:Pesin} in
the flow setting
\begin{align}
  \label{eq:fPesin}
  h_\mu(\phi_1)
  =
  \int \Sigma^+\,d\mu
  =
  \int\log|\det Df\mid_{E^{cu}}|\,d\mu>0,
\end{align}
that is, the physical measures are \emph{$cu$-Gibbs states}. The
results from Leddrappier-Young~\cite{LY85} are still valid in our
setting, that is, ergodic physical measures
satisfying~\eqref{eq:fPesin}, supported on partially hyperbolic
attractors in our setting, admit a disintegration into conditional
measures along the unstable manifolds of almost every point, which are
absolutely continuous with respect to the induced Lebesgue measure on
these submanifolds -- see e.g.~\cite{AraPac2010s} for the
three-dimensional case, and~\cite{araujo_2021} (and references therein)
for the general case.

\begin{example}[Hyperbolic examples]\label{ex:Anosov}
  Anosov flows and hyperbolic attractors (attracting basic pieces of
  the spectral decomposition of Smale) for smooth vector fields admit
  a physical/SRB probability measure whose basin covers the trapping
  region except a zero volume subset \cite{KH95,fisherHasselblatt12},
  and are asymptotically sectional hyperbolic, with no equilibria.
  
  In fact, the central-unstable bundle splits
  $E^{cu}_\Lambda=\RR\cdot G\oplus E^u$ into a pair of continuous
  subbundles: the direction of the flow and an unstable bundle $E^u$
  (uniformly contracting in negative time).

  Using an adapted metric, we can assume without loss of generality that
  $T_\Lambda M=E^s\oplus \RR\cdot G\oplus E^u$ is an \emph{orthogonal}
  invariant splitting
  ; see e.g.~\cite{Goum07}. Then, for all $x\in\Lambda$ and also in
  $U$, the backward contraction on $E^u$ ensures that there exists
  $\lambda >0$ so that, for any bivector $u\wedge v$ with
  $u,v\in E^u_x$, we may assume without loss of generality that
  $\langle u,v \rangle=0$ and obtain
  $\|\wedge^2D\phi_{-t}\cdot(u\wedge v)\| \le
  \|D\phi_{-t}u\|\cdot\|D\phi_{-t}v\| \le e^{-2\lambda t}\|u\|\cdot\|v\|
  =e^{-2\lambda t}\|u\wedge v\|$.  If we instead consider a bivector
  $G(x)\wedge v$ for $v\in E^u_x$ we obtain
  \begin{align*}
    \|\wedge^2D\phi_{-t}\cdot(G(x)\wedge v)\|
    \le
    e^{-\lambda t}\|v\|\cdot \|G(\phi_{-t}x)\|/\|G(x)\| 
    \le
    K e^{-\lambda t}\|G(x)\wedge v\|
  \end{align*}
  since $\|G(x)\|$ is bounded above and also bounded away from zero on
  $M$. We thus obtain (ASH) for all $x\in M$.
  Moreover, we also obtain (NU2SE) for $f=\phi_n$, where 
  $n\ge1$ is such that $K e^{-\lambda n}<1$, that is, we
  get~\eqref{eq:NU2SE} with $\Omega=M$ and $f=\phi_n$.

  
\end{example}

\subsubsection{Recurrence control to equilibria}
\label{sec:recurr-control-equil}

To construct the physical probability measure in the presence of
equilibria for a partially hyperbolic attracting set, we need to
control the recurrence near the equilibria.

\begin{definition}[Slow recurrence, continuous and discrete]
  \label{def:SR}
  Let $U\subset M$ be a forward invariant set of $M$ for the flow of a
  $C^1$ vector field $G$, where all equilibria are hyperbolic. We say
  that $G$ has
  \begin{description}
  \item[(CSR)] \emph{continuous slow recurrence to equilibria}
if, on the positive Lebesgue measure
subset $\Omega\subset U$, for every $\epsilon>0$, we can find
$\delta>0$ so that
\begin{align}
  \label{eq:SSR}
  \limsup\nolimits_{T\nearrow\infty}\frac1T\int_0^T-\log d_\delta
  \big(\phi_t(x),\sing_\Lambda(G)\big) \, dt <\epsilon, \quad x\in\Omega,
\end{align}
where $d_\delta(x,S)$ $\delta$-{\em truncated distance} from $x\in M$
to a subset $S$, that is
\begin{align*}
  d_{\delta}(x,S)=\left\{
  \begin{array}{lll}
d(x,S) & \textrm{if $0<d(x,S)\leq\delta$;}
\\
\left(\frac{1-\delta}{\delta}\right)d(x,S)+2\delta-1 
& \textrm{if $\delta<d(x,S)<2\delta$;}
\\
1 & \textrm{if $d(x,S)\geq2\delta$.}
\end{array} \right.
\end{align*}
\item[(SR)]\emph{slow recurrence to equilibria} if, on the positive
  Lebesgue measure subset $\Omega\subset U$, for every $\epsilon>0$,
  we can find $\delta>0$ so that
\begin{align}
  \label{eq:SR}
  \limsup_{n\nearrow\infty}\frac1n\sum\nolimits_{i=0}^{n-1}-\log d_\delta
  \big(f^i(x),\sing_\Lambda(G)\big) <\epsilon, \quad x\in\Omega,
\end{align}
  \end{description}
\end{definition}

\begin{remark}\label{rmk:SRcond}
  If the function $|\log d_\delta(x,\sing_\Lambda(G))|$ is
  $\mu$-integrable, whenever $\mu$ is a physical measure, the slow
  recurrence to equilibria is a consequence of Birkhoff's Ergodic
  Theorem; see e.g. the comments
  in~\cite{alves-luzzatto-pinheiro2005}. However, either this
  integrability property, or conditions (SR) or (CSR), are hard to
  obtain; see e.g. \cite{ArSzTr} for a setting where this
  was deduced from global properties of the transformation.

  Clearly, the slow recurrence condition is vacuous (and automatically
  valid) if $\Lambda$ contains no equilibria.
\end{remark}

In the presence of equilibria, we obtain the same conclusion of
Theorem~\ref{mthm:discretefabv0} assuming non-uniform sectional
expansion together with slow recurrence.

\begin{maintheorem}\label{mthm:discretefabv}
  Let $G\in\fX^2(M)$ be a vector field with a partially hyperbolic
  attracting set $\Lambda=\Lambda_G(U)$ satisfying (SR) on
  $\Omega\subset U$, with $\m(\Omega)>0$.  Then we have (NU2SE) on
  $\Omega$ if, and only if, there are finitely many ergodic
  physical/SRB measures, which are $cu$-Gibbs states, and whose basins
  cover $\m$-a.e. point of $\Omega$:
  $\m\Big(\Omega\setminus \big(B(\mu_1)\cup\dots\cup
  B(\mu_p)\big)\Big)=0$.
\end{maintheorem}

The proof of Theorem~\ref{mthm:discretefabv} relies on carefully
constructing hyperbolic times and pre-disks\footnote{See
  Sections~\ref{sec:hyperb-times-center}
  and~\ref{sec:lebesgue-measure-at} together with
  e.g. \cite[Chap. 7]{Alves2020b} and compare \cite{ABV00} for many
  more details.} needed to apply the main technical results from
\cite{ABV00}.

We adapt the main ideas from~\cite{ABV00} to allow for the neutral
(non-expanding nor contracting) flow direction along the
center-unstable subbundle -- we stress that this invariant direction
$\RR\cdot G$ (away from singularities) inside $E^{cu}$ is \emph{not
  part of a continuous splitting of $E^{cu}$ in general}, as in the
uniformly hyperbolic case; see e.g.~\cite{fisherHasselblatt12}. The
goal is to build, for each hyperbolic time of a given trajectory, a
center-unstable disk of uniform inner radius which is uniformly
backward contracted along the sectional center-unstable directions,
and deduce a distortion bound to control the density of push-forwards
of Lebesgue measure along these disks.

\begin{example}[Singular-hyperbolic attracting
  sets]\label{ex:Lorenzlike}
  We recall that all singular-hyperbolic attracting sets, as the
  (geometric) Lorenz attractor~\cite{AraPac2010s}, for $C^2$ smooth
  flows with codimension two partially hyperbolic attracting sets,
  admit finitely many (one only if transitive) physical/SRB
  probability measures, which are $cu$-Gibbs states and whose basins
  cover $\m$-a.e. point of the trapping region of these attracting
  sets; see e.g.~\cite{araujo_2021}. These are  partially
  hyperbolic and sectional expanding attracting
  sets. But they also exhibit slow recurrence to equilibria.
  Indeed, it is even possible to obtain \emph{exponentially} slow
  recurrence; see e.g.~\cite{ArSzTr}.

\end{example}

\subsection{Continuous time versus discrete time}
\label{sec:contin-time-versus}

A property similar to ASH on a positive volume subset of the trapping
neighborhood $U$ of a partially hyperbolic attracting set $\Lambda$ is
the following.  We say that $\Lambda$ is \emph{mostly asymptotically
  sectional expanding} (MASE) if there exists $c_0>0$ such that
\begin{align}\label{eq:MASE}
  \Omega=\left\{ x\in U:\limsup_{T\nearrow\infty}
  \frac1T\log\|\wedge^2 (D\phi_T\mid_{E^{cu}_x})^{-1}\| \le - c_0\right\}
  \qand \m(\Omega)>0.
\end{align}
We can interrelate the previous notions (NU2SE), (MASE) and (SR),(CSR)
as follows.

\begin{maintheorem}[Equivalence between discrete and continuous time
  versions]\label{mthm:equivalence}
  Let $G\in\fX^2(M)$ be given admitting a partially hyperbolic
  attracting set $\Lambda=\Lambda_G(U)$. Then
  \begin{enumerate}
  \item the slow recurrence $(SR)$ condition~\eqref{eq:SR} holds for
    $x\in U$ if, and only if, continuous slow recurrence $(CSR)$
    condition~\eqref{eq:SSR} holds for $x$.
  \item if the subset
    $\Omega=\{x\in U:$~\eqref{eq:SR} holds for $x\}$ has positive
    volume, then the following pair of conditions are equivalent:
    \begin{enumerate}[(A)]
    \item there exists a hyperbolic physical/SRB invariant probability
      measure for the flow, which is a $cu$-Gibbs state with
      $\supp\mu\subset\Lambda$;
    \item there exists $T>0$ and $c_0>0$ and a positive volume subset
      $E\subset\Omega$ where we have (NU2SE) for $\phi_T$, that is
      \begin{align}\label{eq:NUSE2}
        \limsup_{n\nearrow\infty}\frac1n\sum\nolimits_{i=0}^{n-1}
        \log\big\|\wedge^2\big(D\phi_T\mid_{E^{cu}_{\phi_{iT}(x)}}\big)^{-1}
        \big\|^{1/T}
        < - c_0 <0, \quad
        x\in E.
      \end{align}
    \end{enumerate}
  \item if either condition of item (2) is met, then
    \begin{enumerate}
    \item the (MASE) condition~\eqref{eq:MASE} holds $\m$-a.e. in $E$;
      and
    \item if, additionally, $\Lambda$ is transitive, then there exists
      one ergodic physical/SRB measure such that
      $\m(B(\mu)\setminus\Omega)=0$.
    \end{enumerate}
  \end{enumerate}
\end{maintheorem}
\begin{remark}[(NU2SE) implies (MASE)]
  This means that, under the assumption of slow recurrence, (NU2SE)
  implies (MASE) whenever we have hyperbolic physical/SRB measures, by
  either reparametrizing the flow $(\phi_t)_{t\in\RR}$ to
  $(\phi_{tT})_{t\in\RR}$ or, equivalently, replacing $G$ by a
  multiple $T\cdot G$.
\end{remark}


\subsection{Non-uniformly sectional hyperbolic variants}
\label{sec:nush}

We explore several possible definitions of non-uniform hyperbolicity
for partially hyperbolic attracting set and relations between them, in
order to exhibit the equivalence of some properties, and present some
results where slow recurrence is not needed to prove the existence of
physical measures.


\begin{definition}[Linear Poincar\'e Flow]
  If $x$ is a regular point of the vector field $G$ (i.e.
  $G(x)\neq \vec0$), denote by
  $ N_x=\{v\in T_xM: \langle v , G(x)\rangle=0\} $ the orthogonal
  complement of $G(x)$ in $T_xM$.  Denote by $O_x:T_xM\to N_x$ the
  orthogonal projection of $T_x M$ onto $N_x$.  For every $t\in\RR$
  define, see Figure~\ref{fig:LPF}
  \begin{align*}
    P_x^t:N_x\to N_{\phi_t x}
    \quad\text{by}\quad
    P_x^t=O_{\phi_t x}\circ D\phi_t(x).
  \end{align*}

\begin{figure}[h]
  \centering \includegraphics[width=10cm]{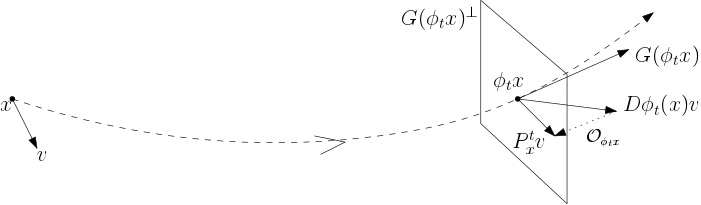}
  \caption{\label{fig:LPF}Sketch of the Linear Poincar\'e flow $P^t_x$
    of a vector $v\in T_xM$ with $x\in M\setminus\sing(G)$ and the
    orthogonal projection $O_{\phi_t x}:T_{\phi_tx}M\to N_{\phi_tx}$,
    with $N_{\phi_tx}=G(\phi_t x)^\perp$.}
\end{figure}
It is easy to see that $P=\{P_x^t:t\in\RR, G(x)\neq 0\}$ satisfies the
cocycle relation $ P^{s+t}_x=P^t_{\phi_s x}\circ P^s_x $ for every
$t,s\in\RR$.  The family $P=P_G$ is called the {\em Linear Poincar\'e
  Flow} of $G$.
\end{definition}

\begin{theorem}[Pointwise non-uniform sectional expanding variants]
  \label{thm:NUSEvar}
  Let $G$ be a $C^1$ vector field generating the flow $\phi_t:M\to M$
  on the compact connected finite dimensional manifold $M$ admitting a
  compact invariant attracting subset $\Lambda=\Lambda_G(U)$, with a
  dominated splitting $T_\Lambda M = E^{cs}\oplus E^{cu}$ on the
  trapping region $U$, continuously extended to $U$ with
  $G(x)\in E^{cu}_x, x\in U$.
  Let $c_0>0$ be given together with $x\in U$ so that its future
  trajectory $\big(\phi_t(x)\big)_{t\ge0}$ does not converge to any
  equilibrium point. Then the following condition (A) implies
  condition (B):
  \begin{enumerate}[(A)]
  \item \emph{weak non-uniformly $2$-sectional expanding
      \textbf{(wNU2SE)}}:
 \begin{align}\label{eq:sNUE}
   \liminf_{n\nearrow\infty}
   \frac1n\sum\nolimits_{i=0}^{n-1}
   \log\|\wedge^2 (Df\mid_{E^{cu}_{f^ix}})^{-1}\| \le - c_0;
 \end{align}
\item \emph{weak non-uniformly sectional expanding \textbf{(wNUSE)}}:
\begin{align}\label{eq:wNUE}
  \liminf_{n\nearrow\infty}\frac1n\sum\nolimits_{i=0}^{n-1}
  \log\|(P^1\mid_{N^{cu}_{f^i x}})^{-1}\| \le -c_0.
\end{align}
\end{enumerate}
  Moreover, the former condition (B) on the time-$1$ map $f$
  implies the following conditions
  \begin{enumerate}[(a)]
  \item \emph{weak mostly asymptotically sectional expanding
      \textbf{(wMASE)}}:
    \begin{align}\label{eq:wMASE1}
      \liminf_{T\nearrow\infty} \frac1T
      \log\|\wedge^2(D\phi_T\mid_{E^{cu}_x})^{-1}\| \le - c_0;
    \end{align}
      \item \emph{weak asymptotic sectional expansion for the
      Linear Poincar\'e Flow \textbf{(wASEP)}}:
    \begin{align}\label{eq:wASELPF}
      \liminf_{T\nearrow\infty}\frac1T
      \log\|(P^T\mid_{N^{cu}_x})^{-1}\| \le - c_0;
    \end{align}
  \end{enumerate}
  and $(a)$ implies $(b)$.
  Again, the previous conditions imply the following 
  conditions
    \begin{enumerate}[(i)]  
\item \emph{weak asymptotically sectional expanding \textbf{(wASE)}}:
      for each $2$-dimensional linear subspace $F_x$ of $E^{cu}_x$ we have
    \begin{align}\label{eq:wASE}
      \limsup_{T\nearrow\infty}\frac1T
      \log|\det(D\phi_T\mid_{F_x})|\ge c_0.
    \end{align}
      \item \emph{weak positive sectional Lyapunov exponents
      \textbf{(wPSL)}}:
    \begin{align}\label{eq:wPSL}
      \limsup_{T\nearrow\infty} \frac1T
      \log\|D\phi_T(x)v\| \ge  c_0, \forall v\in
      E^{cu}_x\setminus\RR\cdot G_x;
    \end{align}
  \end{enumerate}
  and $(i)$ implies $(ii)$.
  
  Reciprocally, if $d_{cu}=\dim E^{cu}=2$, then the latter
  condition (i) implies (A).

  In addition, the statements above are also true after exchanging
  $\liminf$ and $\limsup$. In particular, the following is implied by
  (NU2SE):
  \begin{itemize}
  \item\emph{non-uniformly sectional expanding \textbf{(NUSE)}}:
    \begin{align}\label{eq:NUE}
      \limsup_{n\nearrow\infty}\frac1n\sum\nolimits_{i=0}^{n-1}
      \log\big\|\big(P^1\mid_{N^{cu}_{f^i x}}\big)^{-1}\big\| \le - c_0;
    \end{align}
  \end{itemize}
\end{theorem}

\begin{remark}\label{rmk:wMASE-wPSL}
In shorthand, we obtain the following relations
  \begin{align*}
    [(NU2SE)\implies(NUSE)]
    \quad\&\quad
    (NU2SE)\implies(MASE)\implies(ASEP)
    \implies(ASE)
  \end{align*}
  together with the weak forms
  \begin{align*}
    [\eqref{eq:sNUE}\implies\eqref{eq:wNUE}]
    \quad\&\quad
    \eqref{eq:sNUE}\implies
    \eqref{eq:wMASE1}\implies\eqref{eq:wASELPF}
    \implies\eqref{eq:wASE}
  \end{align*}
  for trajectories not converging to any equilibrium in the future.
  Theorem~\ref{thm:NUSEvar} shows that the (wMASE) condition implies
  positive central-unstable Lyapunov exponents along directions
  transversal to the flow direction: \emph{positive sectional Lyapunov
    exponents}, i.e. the condition (wPSL). Moreover, this condition is
  a consequence of asymptotical sectional expansion (wASE) as claimed
  when presenting the (ASH) condition~\eqref{eq:assecexp}.

  It is easy to see that  all asymptotic properties in the
  statement of Theorem~\ref{thm:NUE-ASEA} do not depend on the
  particular continuous extension of $E^{cu}_\Lambda$ to $U$ chosen
  before, due to the domination of the splitting; see
  Proposition~\ref{prop:Ccu}.
\end{remark}

The assumption~\eqref{eq:NU2SE} of (NU2SE) is close to the notion of
non-uniform sectional hyperbolicity on critical elements (equilibria
and periodic orbits) defined by Arbieto-S. \cite[Definition 2.5 \&
Remark 2.6]{ArbSal2011} to obtain sectional hyperbolicity for the
non-wandering set on a $C^1$ residual subset of vector fields, among
those with non-uniform sectional hyperbolic critical elements.

More precisely, let us assume that the attracting set $\Lambda$ admits
a \emph{continuous invariant splitting} $E^{cs}\oplus E^{cu}$ for the
flow of $G$.  We say that the positive trajectory
$(\phi_t(x))_{t\ge0}$ of $x$ is \emph{weakly non-uniformly sectional
  hyperbolic} (wNUSH) if there exists $\omega>0$ so that
\begin{align}
  \liminf_{T\nearrow\infty}
  &\frac1T\int_0^T \log\| Df\mid_{E^{cs}_{\phi_tx}}\| \, dt <
    -\omega,
    \qand \label{eq:NUCsal}
  \\
  \liminf_{T\nearrow\infty}
  &\frac1T\int_0^T
    \log\| \wedge^2(Df\mid_{E^{cu}_{\phi_tx}})^{-1}\| \, dt
    < -\omega.\label{eq:NUEsal}
\end{align}
These conditions on a \emph{total probability set}\footnote{That is,
  every $x\in E$ satisfy both~\eqref{eq:NUCsal} and~\eqref{eq:NUEsal}
  with the same $c_0$ and $\mu(E)=1$ for each invariant probability
  measure $\mu$ supported in $\Lambda$.} ensure
sectional-hyperbolicity.  Moreover, a strong form of the
property~\eqref{eq:NUEsal} is enough to get (NU2SE), as follows.

\begin{theorem}\label{thm:NUE-ASEA}
  Let $G$ be a $C^1$ vector field generating the flow $\phi_t:M\to M$
  on the compact connected finite dimensional manifold $M$ admitting a
  compact invariant attracting subset $\Lambda=\Lambda_G(U)$, with a
  dominated splitting $T_\Lambda M = E^{cs}\oplus E^{cu}$ on the
  trapping region $U$, continuously extended to $U$ with
  $G(x)\in E^{cu}_x, x\in U$.
  If there exists a total probability subset $E$ of $U$ satisfying
  (wNUSH), then $\Lambda$ is sectional-hyperbolic.

  Moreover, if $\Lambda=\Lambda_G(U)$ is a partially hyperbolic
  attracting set which admits a positive volume subset of points in
  $U$ satisfying~\eqref{eq:NUEsal}, then we obtain a positive volume
  subset of points of $U$ satisfying (wNU2SE) for some $c_0>0$.
  In addition, if we replace~\eqref{eq:NUEsal} by
    \begin{align}\label{eq:NUEsal+}
      \limsup_{T\nearrow\infty}
      \frac1T\int_0^T \log\|\wedge^2(Df\mid_{E^{cu}_{\phi_tx}})^{-1}\|\,dt
      < -\omega,
    \end{align}
    then we obtain (NU2SE), that is, condition~\eqref{eq:NU2SE} for
    some $c_0>0$.
\end{theorem}

This, together with Theorem~\ref{mthm:equivalence}, enables us to
restate Theorem~\ref{mthm:discretefabv} with the
assumption~\eqref{eq:NUEsal+} in the place of (NU2SE).


\subsection{Weak non-uniform sectional expansion and physical measure}
\label{sec:weak-non-uniform}

Combining the proof of the following Theorem~\ref{mthm:fABV} and of
Theorem~\ref{mthm:discretefabv}, enable us to obtain the following
result on existence of physical/SRB measures with weak non-uniform
sectional expansion.

\begin{maincorollary}\label{mcor:wABVndcu}
  Let $G\in\fX^2(M)$ be a vector field with a partially hyperbolic
  attracting set $\Lambda=\Lambda_G(U)$ admitting no
 negative exponents along $E^{cu}_\Lambda$: for all
    $x\in\Lambda$ we have
    \begin{align}\label{eq:NNE}
      \liminf_{T\nearrow\infty} \frac1T
      \log\|D\phi_T(x)v\| \ge 0, \forall v\in
      E^{cu}_x\setminus\RR\cdot G_x;
    \end{align}
  If the (wNU2SE) condition holds, then there exists an ergodic
  physical/SRB probability measure $\mu$ supported on $\Lambda$.

  Moreover, if  condition (SR) holds on $\Omega\subset U$, with
  $\m(\Omega)>0$, then: we have (wNU2SE) on $\Omega$ if, and only if,
  there are finitely many ergodic physical/SRB measures, which are
  $cu$-Gibbs states, and whose basins cover $\m$-a.e. point of
  $\Omega$.
\end{maincorollary}

\begin{remark}\label{rmk:anydcu}
  In
  Theorems~\ref{mthm:discretefabv0},~\ref{mthm:discretefabv},~\ref{mthm:equivalence}
  and Corollary~\ref{mcor:wABVndcu}, there are no restrictions on the
  dimension of the central-unstable subbundle.
\end{remark}

\begin{example}[Sectional-hyperbolic attracting
  sets]\label{ex:multidimLorenz}
  The multidimensional Lorenz attractor~\cite{BPV97} is a
  sectional-hyperbolic attractor with a generalized Lorenz-like
  equilibrium.
  Sectional-expansion on an attracting set naturally implies sectional
  expansion along every orbit in a trapping neighborhood $U$ which, in
  turn, clearly ensures (MASE). In particular, it ensures the
  non-existence of negative Lyapunov exponents along the
  central-unstable direction: we get~\eqref{eq:NNE} with positive
  $\liminf$. Then Corollary~\ref{mcor:wABVndcu} provides an
  alternative proof of existence of a unique ergodic physical/SRB
  measure for this family of attractors, complementing the arguments
  given in~\cite{BPV97} and the proof presented in~\cite{LeplYa17} for
  the general smooth sectional-hyperbolic attractor (which was
  extended in~\cite{araujo_2021} for smooth sectional-hyperbolic
  attracting sets).
\end{example}

\subsubsection{Codimension two partially hyperbolic attractors}
\label{sec:codimens-two-partial}

In low codimension, we show that the (wMASE) condition on a positive
volume subset is enough to ensure the existence of a physical/SRB
measure.  

\begin{maintheorem}\label{mthm:fABV}
  Let a partially hyperbolic attracting set $\Lambda=\Lambda_G(U)$ for
  a vector field $G\in\fX^2(M)$ be given, with $d_{cu}=\dim
  E^{cu}=2$. Then $\Lambda$ satisfies (wASE) on a positive volume subset
  \begin{align*}
    \leb\left( \left\{ x\in U: \liminf_{T\nearrow\infty} \frac1T
    \log|\det (D\phi_T\mid_{E^{cu}_x})| >0 \right\} \right) >0
  \end{align*}
  \emph{if, and only if}, there exists a
  physical/SRB ergodic hyperbolic measure $\mu$.
  If $\Lambda$ is transitive, then $\mu$ is unique and
  $\leb(\Omega\setminus B(\mu))=0$.

  Reciprocally, without restriction on $d_{cu}$, the existence of an
  invariant ergodic hyperbolic physical/SRB measure implies that
  (wMASE) holds on a positive volume subset of $U$.
\end{maintheorem}
Hence, to obtain a physical measure, it is enough to obtain a sequence
of times with asymptotic sectional expansion, along the trajectories
on a positive volume subset.

\subsubsection{Codimension two partially hyperbolic attractors with
  slow recurrence}
\label{sec:codimens-two-hyperb-1}

Restricting to codimension two with slow recurrence, we gain
finiteness of physical measures.

\begin{maincorollary}\label{mcor:wABV}
  Let a partially hyperbolic attracting set $\Lambda=\Lambda_G(U)$ for
  a vector field $G\in\fX^2(M)$ be given, with $d_{cu}=2$ and assume
  that there exists $c_0>0$ so that (wASE) holds on a positive volume
  subset
$
  \Omega:=\left\{ x\in U :
  \liminf_{T\nearrow\infty} \frac1T
    \log|\det (D\phi_T\mid_{E^{cu}_x})| >0
  \right\}
$
and points in $\Omega$ satisfy (SR). Then, there are finitely many
ergodic physical/SRB measures $\mu_1,\ldots,\mu_p$, which are
$cu$-Gibbs states, and whose basins cover $\m$-a.e. point of $\Omega$.
\end{maincorollary}

Using the proof of the previous main results, we can deduce the
following extension of the main statement from~\cite{smviv23} on
existence of physical measure for (ASH) attractors.

\begin{maincorollary}
  \label{mcor:SRB-ASH}
  Let a $C^2$ vector field $G$ on $M$ and a trapping region $U$ be
  given containing a attracting set $\Lambda=\Lambda_G(U)$ with
  $d_{cu}=2$ so that every $x\in\Lambda$ not converging to any
  equilibrium satisfies (wASE).

  If $\Lambda$ contains only saddle-type hyperbolic equilibria, then
  there exists a physical/SRB probability measure supported on
  $\Lambda$.  If $\Lambda$ is transitive, then $\Lambda$ supports a
  unique physical/SRB probability measure whose basin covers a
  neighborhood of $\Lambda$.
\end{maincorollary}

This corollary shows that it is enough to obtain just a sequence of
times with asymptotic sectional expansion, along the trajectories not
converging to a singularity, to conclude the existence of a physical
measure.


\begin{example}[Rovella attractors]\label{ex:Rovellalike}
  Another class of partially hyperbolic and mostly asymptotically
  sectional expanding attracting sets are the Rovella
  attractors~\cite{Ro93} (also known as contracting Lorenz attractors)
  which are (weak) asymptotically singular-hyperbolic;
  see~\cite{SmartinVivas20}. These attractors admit a physical/SRB
  probability measure whose basin covers the trapping region except a
  zero volume subset and, moreover, exhibit slow recurrence to the
  equilibrium at the origin\footnote{The known proof of asymptotical
    sectional expansion for this family of attractors intertwines with
    the control of recurrence to the equilibrium point.};
  see~\cite{mtz001}.
\end{example}

\subsection{Comments}
\label{sec:comments-conjectures}

Recently, Crovisier et al.~\cite{CWYZ23} obtained physical measures for
$C^1$-generic $C^\infty$ multisingular vector fields. However this
class of results fails to take into account non-transitive attracting
singular sets as well as Rovella-like attractors, which are
encompassed by our main statements.

\subsubsection{On the proof of Theorems~\ref{mthm:discretefabv0}
  and~\ref{mthm:discretefabv}}
\label{sec:proof-theorems-mthmabv}

The reliance on discrete dynamics, through the reduction to asymptotic
properties of iterations of the time-$1$ map $f=\phi_1$, suggests the
classical alternative of reducing the flow dynamics to a global
Poincar\'e return map to a well-chosen finite collection of
cross-sections, as done in \cite{APPV,ArSzTr,ArTri21,mtz001} to obtain
physical/SRB measures for contracting Lorenz attractors and
sectional-hyperbolic attractors. This strategy however is technically
challenging:
\begin{enumerate}
\item non-uniform sectional expanding partial hyperbolic attractors
  \emph{may admit different hyperbolic non-Lorenz-like
    equilibria}, especially with $d_{cu}>2$, while
  sectional-hyperbolic or contracting Lorenz attractors admit only a
  well-controlled family of (generalized) Lorenz-like equilibria;
\item continuous slow recurrence or slow recurrence \emph{for the time-$1$
    map} do not play a role in the construction of physical measures
  for sectional-hyperbolic or contracting Lorenz attractors because,
  in this special partial hyperbolic setting, the dynamics of the
  global Poincar\'e map can be further reduced to a one-dimensional
  quotient map over the stable foliation, which demands that
  \begin{itemize}
  \item $d_{cu}=2$ in order to obtain the one-dimensional quotient;
  \item holonomies along the stable foliation exhibit some smoothness
    on the trapping region, to ensure that the quotient transformation
    is at least piecewise H\"older-$C^1$.
  \end{itemize}
\end{enumerate}
Circumventing the problems posed by the previous items in the (NUSE)
setting, especially with $d_{cu}>2$, with similar strategies to e.g.
\cite{ArTri21,mtz001} may be impossible, since higher dimensional
invariant foliations are in general only
H\"older-continuous~\cite{PSW97}. However, we can adapt the
construction from \cite{ABV00} to the vector field setting, as
presented in this text, and also show that slow recurrence is
automatic for trajectories not converging to equilibria.

\subsubsection{Idea of the proof of Theorem~\ref{mthm:fABV}}
\label{sec:idea-proof-theorem}

The proof relies on reducing mostly
asymptotical sectional expansion to a discrete version applied to the
time-$1$ map of the flow, or some other fixed time after a time
reparametrization, taking care to allow for the neutral (non-expanding
nor contracting) flow direction along the center-unstable subbundle;
and applying the following useful extension of Pesin's
Formula~\cite{Pe77} obtained by
Catsigeras-Cerminara-Enrich~\cite{CatCerEnr15}.

\begin{theorem}[Generalized Pesin's Inequality~\cite{CatCerEnr15}]
  \label{thm:GenPesin}
  For any $C^1$ diffeomorphism $f$, if $\Lambda$ is an invariant
  compact set with a dominated splitting $T_\Lambda M= E\oplus F$,
  then for Lebesgue almost every point $x$ satisfying
  $\omega(x)\subset \Lambda$, the entropy of any weak$^*$ limit
  measure $\mu$ of the sequence
  $\big(\frac1n \sum_{i=0}^{n-1} \delta_{f^i(x)}\big)_{n\ge1}$ is
  bounded from below:
  \begin{align}\label{eq:gpesin}
    h_\mu(f)\geq \int \log|\det Df\mid_{F}|\,d\mu.
  \end{align}
\end{theorem}

We use this for $x$ satisfying the (wMASE) condition~\eqref{eq:wMASE1}
to find an ergodic hyperbolic physical/SRB measure as an ergodic
component of a limit measure $\mu$ as above. It is well known from the
work of Ledrappier-Young~\cite{LY85} that for $C^2$ systems such
measures are physical measures.

\subsubsection{Conjectures and extensions} 
\label{sec:conjectures}

On the one hand, using the same techniques from~\cite{CMY2017}, we may
replace the domination condition by H\"older continuity of the
splitting over $\Lambda$, keeping the conclusions of the main theorems
and extending the thermodynamical methods from~\cite{CatCerEnr15} to
our vector field setting.

On the other hand, from Remark~\ref{rmk:domsplit2parthyp}, we can
replace the partial hyperbolic assumption on $\Lambda$ by the
assumption that $\Lambda$ be a non-trivial attracting set with a
dominated splitting $T_\Lambda M=E^s\oplus E^{cu}$, such that the
vector field $G$ is contained in $E^{cu}$, and keep the same
conclusions of the main results.  More precisely, we obtain the
following statements with small adaptations of our arguments.

\begin{theorem}\label{thm:scholium1}
  Let $\Lambda=\Lambda_G(U)$ be an attracting set for a vector field
  $G\in\fX^2(M)$ admitting an invariant splitting
  $T_\Lambda M=E^s\oplus E^{cu}$ such that
  \begin{itemize}
  \item either the splitting is H\"older-continuous and $E^s$ is
    uniformly contracted;
  \item or the splitting is dominated and the flow direction is
    contained in $E^{cu}$.
  \end{itemize}
  If either (SR) holds on $\Omega_0\subset U$, or $d_{cu}=2$, then
  $\Lambda$ satisfies (NU2SE) on a positive volume subset
  $\Omega\subset\Omega_0$ \emph{if, and only if}, there exists an
  ergodic hyperbolic physical/SRB measure $\mu$, which is a $cu$-Gibbs
  state with $\m(B(\mu)\cap\Omega)>0$.
\end{theorem}

\begin{example}[Non-dominated continuous splitting]\label{ex:nondomcont}
  We consider the suspension flow $\phi_t$, with constant roof of
  height $1$, of the diffeomorphisms $f:\TT^4\to\TT^4$ of the
  $4$-torus studied by Tahzibi~\cite{tah04}. The diffeomorphism
  admits a $Df$-invariant dominated splitting
  $TM= E^{cs}\oplus E^{cu}$ together with a hyperbolic periodic point
  with expanding direction contained in $E^{cs}$, and another
  hyperbolic periodic point with contracting direction contained in
  $E^{cu}$. Hence, the naturally inherited $D\phi_t$-invariant
  splitting of the suspension
  $T\wt{M}=E^{cs}\oplus (E^{cu}\oplus(\RR\cdot G))$ is
  \emph{continuous but cannot be dominated}, where $\RR\cdot G$ is the
  flow direction on the special manifold $\wt{M}:= M/\sim$ defined by
  the natural dynamical identification.

  Moreover, since the time-$1$ map $f=\phi_1$ satisfies $P^1=Df$, then
  $E^{cu}=N^{cu}$ and, by the constructions presented in~\cite{tah04},
  we obtain condition (NUSE). It also satisfies average asymptotic
  contraction along $E^{cs}$. In addition, since $f$ is transitive and
  admits a unique ergodic physical/SRB invariant measure, then the
  suspension of this measure is also physical/SRB with respect to the
  suspension flow $\phi_t$; see
  e.g. Subsection~\ref{sec:finitely-mamy-physic}.
\end{example}

Recently, Cao-Mi-Zou~\cite{caomizou} propose a method to construct
Pesin unstable manifolds for hyperbolic measures admitting a continuous
splitting of the tangent bundle. Example~\ref{ex:nondomcont} indicates
that the following stronger result should be attainable.

\begin{conjecture}
  \label{conj:nondomcont}
  The same conclusion of Theorem~\ref{thm:scholium1} holds assuming a
  continuous splitting $T_\Lambda M=E^{cs}\oplus E^{cu}$ with $E^{cs}$
  having only negative Lyapunov exponents Lebesgue almost everywhere.
\end{conjecture}

Naturally we expect the dimensional restriction $\dim E^{cu}=2$ on the
statement of Theorem~\ref{thm:NUSEvar} to be just a technical
shortcoming of our proof.

\begin{conjecture}
  \label{conj:fABV}
  We have (wASE)$\implies$(wNU2SE) for any value
  of $\dim E^{cu}>2$.
\end{conjecture}

Example~\ref{ex:nuenophysical}, namely
inequalities~\eqref{eq:secontr}, show that (wPSL)$\implies$(wASE) is
subtle when we are dealing with trajectories which might not be
Oseledets regular. On forward Oseledets regular points, conditions
(A)-(B); (a)-(b) and (i)-(ii) are
equivalent. From~\cite{JiangQiPing03} our main results ensure that the
ergodic basins of the physical measures are positive volume subsets of
forward Oseledets regular points. It is thus natural to conjecture
that this might be a priori true in more generality.

\begin{conjecture}
  \label{conj:PSL}
  The implications (wNUSE)$\implies$(wNU2SE); (wPSL)$\implies$(wASE)
  and (wNU2SE)$\implies$~\eqref{eq:NUEsal} hold on a full volume
  subset of points of the trapping region of a generic smooth
  partially hyperbolic attractor.
\end{conjecture}

Analogously to Theorem~\ref{thm:NUE-ASEA} if, in the setting of
Theorem~\ref{mthm:fABV}, we have (wMASE) \emph{on total
  probability}, that is, the relation~\eqref{eq:wMASE1} for all $x$ on
a total probability subset $\Omega\subset \Lambda$, then $\Lambda$
becomes a sectional-hyperbolic attracting set and the conclusion of
Theorem~\ref{mthm:fABV} still holds true by the result
from~\cite{araujo_2021} or Corollary~\ref{mcor:wABVndcu} without any
dimensional restriction.

The analogous condition for (local) diffeomorphisms was shown by
Alves, Dias, Luzzatto and Pinheiro~\cite{ADLP,Pinheiro05} to be enough
to obtain the same conclusions of the main results from~\cite{ABV00}.
Recently, a different approach to the proof of this result was
presented by Burguet-Yang~\cite{BurguetYang25} and
Theorem~\ref{mthm:fABV} can be seen as an extension of this result to
the continuous time setting. Since  (SR) was not assumed, it
is natural to present the following.

\begin{conjecture}\label{conj:wNUE}
  In the same setting of Theorem~\ref{mthm:equivalence}, either the
  (wNUSE) condition~\eqref{eq:wNUE}, or the (wMASE)
  condition~\eqref{eq:wMASE1} is enough to obtain the same conclusion
  on existence and finiteness of physical/SRB measures, without
  assuming the slow recurrence $(SR)$ condition.
\end{conjecture}



As for the number of distinct ergodic physical/SRB measures supported
in the attracting set, motivated by the recent result~\cite{araujo23}
we propose the following.

\begin{conjecture}\label{conj:numbersrb}
  In the setting of the main results, the number $s$ of
  ergodic physical/SRB measures supported in the attracting set
  satisfies $s\le 2\cdot s_L$, where $s_L$ is the number of
  generalized Lorenz-like equilibria contained in $\Lambda$.
\end{conjecture}

The natural path after obtaining a physical/SRB measure is to study
its statistical properties. Motivated by what has already been achieved
for singular-hyperbolic attracting sets
\cite{ArMel16,ArMel18,AMV15,ArTri21}; for sectional-hyperbolic
attracting sets~\cite{araujo_2021}; for contracting Lorenz
attractors \cite{mtz00}; and also in the discrete time case
\cite{Alves2020b,alves-luzzatto-pinheiro2005,AlPi10,Pinho2011}, we
propose the following.

\begin{conjecture}\label{conj:statprop}
  Given a non-uniformly sectional expanding partial hyperbolic
  attracting set $\Lambda_G(U)$ with hyperbolic singularities for a
  $C^2$ vector field $G$, then
  \begin{itemize}
  \item modulo an arbitrary small perturbation of the norm of $G$, the
    field is topologically equivalent to a $C^2$ nearby vector field
    so that each ergodic physical/SRB is exponential mixing with
    respect to smooth observables; and
  \item both the flow $\phi_t$ of $G$ and its time-$1$ map $f=\phi_1$
    satisfy the Central Limit Theorem, the Law of the Iterated
    Logarithm and the Almost Sure Invariance Principle (for a
    comprehensive list, see e.g~\cite{PhilippStout75}) with respect
    to $\mu$.
  \end{itemize}
  Moreover, motivated by~\cite{araJSP21,mtz00}, the
  physical/SRB measures supported on $\Lambda$ should be statistically
  stable and also stochastically stable.
\end{conjecture}

\begin{remark}
  \label{rmk:neutral}
  Recently, Bruin-Farias~\cite{bruin2023mixing} (see
  Example~\ref{ex:nonhypeq} in Section~\ref{sec:new-examples}), 
  polynomial speed of mixing was proved for a neutral geometrical
  Lorenz-like attractor, so the assumption of hyperbolic equilibria
  should be necessary in Conjecture~\ref{conj:statprop}.
\end{remark}

\subsection{Organization of the text}
\label{sec:organization-text}

In Section~\ref{sec:preliminary-results} we present preliminary
results that are needed as tools for the overall construction,
including the proof of Theorems~\ref{thm:NUSEvar}
and~\ref{thm:NUE-ASEA}.  In Section~\ref{sec:hyperb-times-center} we
start the proof of Theorem~\ref{mthm:discretefabv}, using the
domination of the splitting to obtain bounded distortion on $cu$-disks
at hyperbolic times. Using this tool, in
Section~\ref{sec:lebesgue-measure-at} we study the push-forward of
Lebesgue measure at hyperbolic times along $cu$-disks, obtaining the
main tool to construct physical/SRB measures which are $cu$-Gibbs
states. We then provide an overview of the construction of
physical/SRB measures, citing the relevant results from Alves, Bonatti
and Viana~\cite{ABV00}, and complete the proof of
Theorem~\ref{mthm:discretefabv}.

In Section~\ref{sec:recipr-condit} we prove
Theorem~\ref{mthm:equivalence} first.  Then, using the Generalized
Pesin's Inequality Theorem~\ref{thm:GenPesin}, we obtain
Theorem~\ref{mthm:fABV} and deduce
Corollaries~\ref{mcor:wABVndcu},~\ref{mcor:wABV}
and~\ref{mcor:SRB-ASH} in this section.

Finally, in Section~\ref{sec:new-examples}, we present mostly
asymptotically sectional expanding examples which are either
non-sectional hyperbolic or non-singular hyperbolic, with or without
hyperbolic equilibria, as well as counter-examples failing some of our
assumptions and having no physical measure.

\subsection*{Acknowledgments}

V.A. thanks the Mathematics and Statistics Institute of the Federal
University of Bahia (Brazil) for its support of basic research and
CNPq (Brazil) for partial financial support. L.S. and S.S. thank the
Mathematics Institute of Universidade Federal do Rio de Janeiro
(Brazil) for its encouraging of mathematical research and S.S.  thanks
CNPq for the Doctoral scholarship. L.S. thanks the Mathematics and
Statistics Institute of the Federal University of Bahia (Brazil) for
its hospitality together with CAPES-Finance code 001, CNPq, FAPERJ -
Aux{\'\i}lio B\'asico \`a Pesquisa (APQ1) Project E-26/211.690/2021;
and FAPERJ - Jovem Cientista do Nosso Estado (JCNE) grant
E-26/200.271/2023 for partial financial support; and also PROEXT-PG
project Dynamic Women - Din\^amicas, CNPq-Brazil (grant Projeto
Universal 404943/2023-3).  We also thank the anonymous referees for
the careful reading of this manuscript and the many valuable
suggestions given.



\section{Auxiliary results}
\label{sec:preliminary-results}

The following results will be used in our arguments.

\subsection{Partial hyperbolic attracting sets}
\label{sec:section-hyperb-attra-1}

The following properties of partial hyperbolic attracting sets will be
used as tools in our arguments. 

\subsubsection{Extension of the stable bundle and
  center-unstable cone fields}
\label{sec:ext-stable-bundle}

Let $\D^k$ denote the $k$-dimensional open unit disk and let
$\mathrm{Emb}^r(\D^k,M)$ denote the set of $C^r$ embeddings
$\psi:\D^k\to M$ endowed with the $C^r$ distance. We say
that \emph{the image of any such embedding is a $C^r$
  $k$-dimensional disk}.

\begin{proposition}{\cite[Proposition~3.2, Theorem~4.2 and
    Lemma~4.8]{ArMel17}}
  \label{prop:Ws}
  Let $\Lambda$ be a partially hyperbolic attracting set.
  \begin{enumerate}
  \item The stable bundle $E^s$ over $\Lambda$ extends to a
    continuous uniformly contracting $D\phi_t$-invariant
    bundle $E^s$ on an open positively invariant
    neighborhood $U $ of $\Lambda$.
  \item There exists a constant $\lambda\in(0,1)$, such that
    \begin{enumerate}
    \item for every point $x \in U $ there is a $C^r$
      embedded $d_s$-dimensional disk $W^s_x\subset M$, with
      $x\in W^s_x$, such that $T_xW^s_x=E^s_x$;
      $\phi_t(W^s_x)\subset W^s_{\phi_tx}$ and
      $d(\phi_tx,\phi_ty)\le \lambda^t d(x,y)$ for all
      $y\in W^s_x$, $t\ge0$ and $n\ge1$.

    \item the disks $W^s_x$ depend continuously on $x$ in
      the $C^0$ topology: there is a continuous map
      $\gamma:U \to {\rm Emb}^0(\D^{d_s},M)$ such that
      $\gamma(x)(0)=x$ and $\gamma(x)(\D^{d_s})=W^s_x$.
      Moreover, there exists $L>0$ such that
      $\lip\gamma(x)\le L$ for all $x\in U $.

    \item the family of disks $\cF^s=\{W^s_x:x\in U \}$ defines a
      topological foliation $\cW^s$ of $U $: every $x_0\in
      U $ admits a neighborhood $V\subset U $ and a
      homeomorphism $\psi:V\to\RR^{d_s}\times\RR^{d_{cu}}$
      so that $\psi(W^s_x)=\pi_s^{-1}\{\pi_s(\psi(x))\}$ where
      $\pi^s: \RR^{d_s}\times\RR^{d_{cu}} \to\RR^{d_s}$ is
      the canonical projection.
    \end{enumerate}
  \end{enumerate}
\end{proposition}

\begin{remark}
  \label{rmk:contholo}
  For any two close enough $d_{cu}$-disks $D_1,D_2$ contained in $U$
  and transverse to $\cF^s$ there exists an open subset $\hat D_1$ of
  $D_1$ so that $W^s_x\cap D_2$ is a singleton for each
  $x\in\hat D_1$. This defines the \emph{holonomy map}
  $h:\hat D_1\to D_2, \hat D_1\ni x\mapsto W^s_x\cap D_2$ and
  Proposition~\ref{prop:Ws} ensures that $h$ is continuous.
\end{remark}

The splitting $T_\Lambda M=E^s\oplus E^{cu}$ extends continuously to a
splitting $T_{U } M=E^s\oplus E^{cu}$ where $E^s$ is the invariant
uniformly contracting bundle in Proposition~\ref{prop:Ws} -- however
$E^{cu}$ is not invariant in general, but the center-unstable cone
field satisfies the following.

\begin{proposition}
  \label{prop:Ccu}
  Let $\Lambda$ be an attracting set with a dominated splitting so
  that the flow direction is contained in the center-unstable bundle
  $G\in E^{cu}$. Then, for any $a>0$, after possibly shrinking $U$, we
  can find $\kappa>0$ so that
  $ D\phi_t\cdot \cC^{cu}_a(x)\subset \cC^{cu}_{\kappa
    e^{-\lambda^t}a}(\phi_tx)$ for all $t>0$, $x\in U$.
\end{proposition}

\begin{proof} See~\cite[Proposition~3.1]{ArMel17} considering the
  choice of adapted Riemannian metric as defined in
  Subsection~\ref{sec:part-hyperb-diff}: we estimate for
  $v\in C^{cu}_a(x)$ (using only the domination of the splitting)
  \begin{align*}
    \frac{\|D\phi_t(x)\cdot v^s\|}{\|D\phi_t(x)\cdot v^c\|}
    \le
    \frac{\|D\phi_t\mid_{E^s_x}\|\cdot\|v^s\|}
    {\|(D\phi_t\mid_{E^{cu}_x})^{-1}\|^{-1}\cdot\|v^c\|}
    \le
    e^{-\lambda^t} a.
  \end{align*}
  However, since $E^{cu}$ extended to $U$ is not necessarily
  $D\phi_t$-invariant, we need to project $D\phi_t(x)\cdot v^c$ to
  $E^{cu}_{\phi_tx}$ parallel to $E^s_{\phi_t x}$ to decompose
  $D\phi_t(x)\cdot v$ into stable/center-unstable components. Because
  both $E^{cu}_{\phi_tx}$ and $D\phi_t\cdot E^{cu}_x$ are contained in
  $C^{cu}_a(\phi_tx)$, then we can find $\kappa=\kappa(a)>0$ so that
  $\kappa\|\pi^{cu}\cdot D\phi_t(x)\cdot v^c\|\ge\|D\phi_t(x)\cdot v^c\|$,
  and so
$
    \frac{\|D\phi_t(x)\cdot v^s\|}{\|\pi^{cu}\cdot D\phi_t(x)\cdot v^c\|}
    \le
    \kappa e^{-\lambda^t}a,
 $
  which completes the proof of the statement.
\end{proof}

\subsubsection{Partial hyperbolicity of Poincar\'e maps}
\label{sec:hyperb-poincare-maps}

Let $\Sigma,\Sigma'$ be a small cross-sections to $G$ contained in $U$
and let $R:\dom(R)\to\Sigma'$ be a Poincar\'e map\footnote{Note that
  $R$ needs not correspond to the first time the orbits of $\Sigma$
  encounter $\Sigma'$ nor it is defined everywhere in $\Sigma$.}
$R(y)=\phi_{t(y)}(y)$ from an open subset $\dom(R)$ of $\Sigma$ to
$\Sigma'$ (possibly $\Sigma=\Sigma'$).  The splitting
$E^s\oplus E^{cu}$ over $U$ induces a continuous splitting
$E_\Sigma^s\oplus E_\Sigma^{cu}$ of the tangent bundle $T\Sigma$ to
$\Sigma$ (analogously for $\Sigma'$) as
\begin{align}\label{eq:splitting}
E_\Sigma^s(y)=E^{s}_y\cap T_y{\Sigma}
\quad\mbox{and}\quad
E_\Sigma^{cu}(y)=E^{cu}_y\cap T_y{\Sigma}.
\end{align}
The splitting \eqref{eq:splitting} is partially hyperbolic for $R$, as
follows.

\begin{proposition}{\cite[Proposition 4.1]{ArMel18}\& \cite[Lemma
  8.25]{AraPac2010s}}
  \label{pr:secaohiperbolica}
  Let $R:\Sigma\to\Sigma'$ be a Poincar\'e map with Poincar\'e time
  $t(\cdot)$.  Then $DR\cdot E_\Sigma^s(x) = E_\Sigma^s(R(x))$ at every
  $x\in\Sigma$ and $DR\cdot E_\Sigma^{cu}(x) = E_\Sigma^{cu}(R(x))$ at
  every $x\in\Lambda\cap\Sigma$.  Moreover, for $x\in\Sigma$ we have
$
\|DR \mid_{E^s_\Sigma(x)}\| < \lambda^{t(x)}$
and
$\|DR \mid_{E^s_\Sigma(x)}\|\cdot\|\big(DR \mid_{E^{cu}_\Sigma(x)}\big)^{-1}\|
  < \lambda^{t(x)}.
$
\end{proposition}

Given a cross-section $\Sigma$, $b>0$ and $x\in \Sigma$, the unstable
cone of width $b$ at $x$ is
\begin{align}
\label{eq:cusectioncone}
C^u_b(\Sigma,x)=\{v=v^s+v^u : v^s\in E^s_\Sigma(x),\, v^u\in E^{cu}_\Sigma(x)
\mbox{ and } \|v^s\| \le b \|v^u\| \}.
\end{align}

\begin{corollary}
\label{cor:cusectioncone}
There exists $b>0$ small enough so that, for each $R:\Sigma\to\Sigma'$
as in Proposition~\ref{pr:secaohiperbolica}, we have
$DR(x)\cdot C^u_b(\Sigma,x)) \subset
C^u_{b\lambda^{t(x)}}(\Sigma',R(x))$ for all $x \in\Sigma$.
\end{corollary}

\begin{proof}
  Cf. proof of \cite[Proposition 3.1]{ArMel17} which is similar to
  the proof of Proposition~\ref{prop:Ws}.
\end{proof}


\subsection{H\"older control of the tangent bundle in the
  center-unstable direction}
\label{sec:holder-control-tange}

We recall that we have continuous extensions of the two subbundles
$E^{s}$ and $E^{cu}$ defined on an isolating neighborhood $U$ of
$\Lambda$, and the respective cone fields $C_a^s(x), C_a^u(x), x\in U$
for a small $0<a<1$ which are invariant in the sense of
\eqref{eq:coneinv}.

We may assume without loss of generality that, up
to increasing the value of $\lambda > 0$ by a small amount and reducing
the neighborhood $U$ of $\Lambda$, a ``bunched domination condition''
holds true for vectors in these cone fields: there exists
$\zeta\in(0,1)$ so that 
\begin{align*}
  \|D\phi_t\cdot u \| \cdot \|D\phi_{-t} \cdot v\|^{1+\zeta}
  \le
  e^{-\lambda^t}\cdot \|u\|\cdot \|v\|,
  \quad\text{for  } t>0, x\in U, u\in C^s_a(x) \;\&\;
  v\in C^{cu}_a(\phi_tx).
\end{align*}
A $C^1$ disk $D$ on $M$, that is, the image of a $C^1$ embedding
$\psi:B(0,1)\subset\RR^{d_{cu}}\to M$ defined on the unit ball of an
Euclidean space, is a \emph{$cu$-disk} if
$T_yD\subset C^{cu}_a(y), y\in D$.

We fix $\rho_0>0$ so that the inverse of the exponential map $\exp_x$
is defined on the $\rho_0$ neighborhood of each point $x\in U$, which
we identify with the corresponding neighborhood $V_x$ of the origin
$0$ in $T_xM$, and $x$ with $0$.

We may assume without loss of generality that $E^s_x\subset C^s_a(y)$ for all
$y\in V_x$ so that, in particular, $E^s_x\cap C^{cu}_a(x)=\{0\}$.  If
$x\in D$ then $T_yD$ is given by the graph of the linear map
$A_x(y):T_xD\to E^s_x$ for each $y\in V_x\cap D$.

We say that \emph{the tangent bundle $TD$ is $(C,\zeta)$-H\"older} if
there are constants $C, \zeta >0$ such that $\|A_x(y)\|\le C d_x(y)^\zeta$ for
$y\in D\cap V_x$, where $d_x(y)$ is the intrinsic distance from $x$ to
$y$ within $D\cap V_x$\footnote{The length of the shortest curve
  connecting $x$ to $y$ inside $D\cap V_x$}.  Given a $cu$-disk $D$ we
write $\kappa(D)$ for the least $C>0$ so that the tangent bundle of
$D$ is $(C,\zeta)$-H\"older.

We recall the notation $f=\phi_1$ for the time-$1$ map of the flow of
$G$. Then we can prove the following, since $\Lambda$ is also a
partially hyperbolic attracting set for $f$.

\begin{proposition}{\cite[Proposition 2.2 \& Corollary 2.4]{ABV00}}
  \label{pr:curvature} There exists $C_1>0$ so that
  each $cu$-disk $D\subset U$ satisfies
  \begin{enumerate}[(a)]
  \item there exists $n_0\ge 1$ such that $\kappa(f^n(D)) \le C_1$ for
    every $n\ge n_0$ such that $f^k(D) \subset U$ for all
    $0\le k \le n$; and, if $\kappa(D) \le C_1$, then $n_0= 1$;
\item
  $ J_k: f^k(D)\ni x \mapsto \log |\det \big(Df \mid_{T_x f^k(D)}\big)|$
  are $(L_1,\zeta)$-H\"older continuous for $0\le k \le n$ whenever
  $D$ and $n$ are as above, where $L_1=L_1(f,C_1)>0$ depends only on
  $C_1$ and $G$.
\end{enumerate}
\end{proposition}

\begin{remark} \label{rmk:flowdisk} For any small $\epsilon>0$, the
  family
  $\{\phi_{(-\epsilon,\epsilon)}\big[D\cap\exp_z\big(N_z^u\cap
  B(0,\rho_0)\big)\big]: z\in U\}$ of $cu$-disks is flow
  invariant. Then all of them have curvature bounded above by $C_1$.
\end{remark}

\subsection{Equivalences and implications between non-uniform
  sectional expansion versions}
\label{sec:non-uniform-section-2}

Here we prove Theorem~\ref{thm:NUSEvar} first and then
Theorem~\ref{thm:NUE-ASEA}.

\begin{proof}[Proof of Theorem~\ref{thm:NUSEvar}]
  We recall that the extension of $E^{cu}_\Lambda$ to $U$ is not
  necessarily $D\phi_t$-invariant and note that, given any
  $2$-subspace $F_x$ of $E^{cu}_x$, the map
  $t\mapsto |\det (D\phi_t\mid_{F_x})|$ is a multiplicative function
$
    |\det (D\phi_{t+s}\mid_{F_x})|
    =
    |\det (D\phi_s\mid_{D\phi_t\cdot F_x})| \cdot |\det (D\phi_t\mid_{F_x})|$,
    for all $t,s\ge0, x\in U$.

  In addition, by Proposition~\ref{prop:Ccu}, we get
  $D\phi_t\cdot E^{cu}_x\subset C^{cu}_{\kappa \lambda^t a}(\phi_tx)$
  for constants $\kappa,a>0$. Then the stable direction
  $E^s_{\phi_tx}$ is complementary to both the $D\phi_t\cdot E^{cu}_x$
  and $E^{cu}_{\phi_tx}$ directions at $\phi_tx$. Therefore there
  exists a natural isomorphism
  $\pi^s:(D\phi_t\cdot E^{cu}_x)\to E^{cu}_{\phi_tx}$ given by the
  projection parallel to $E^s_{\phi_tx}$. Hence,
  $\pi^s(D\phi_t\cdot F_x) = {\wh F}_{\phi_tx}\subset
  E^{cu}_{\phi_tx}$ and since the width $\kappa\lambda^t a$ of the
  center-unstable cone around $E^{cu}_{\phi_tx}$ is small for large
  $t>0$, then we obtain $\xi_t\to 1$ when $t\nearrow\infty$ such that
  for any fixed $s\ge0$
  \begin{align}\label{eq:approxmutl}
    \xi_t^{-1}
    |\det (D\phi_s\mid_{{\wh F}_{\phi_t x}})|
    &\le
      \frac{|\det (D\phi_{t+s}\mid_{F_x})|}{|\det (D\phi_t\mid_{F_x})|}
      \le
      \xi_t
      |\det (D\phi_s\mid_{{\wh F}_{\phi_t x}})|.
  \end{align}
  
  We first show that \textbf{(wASE)}, given by~\eqref{eq:wASE},
  implies the \textbf{(wNU2SE)} condition~\eqref{eq:sNUE}, i.e.,
  (i)$\implies$(A), when
  $d_{cu}=2$. By assumption, given any $2$-subspace
  $F_x\subset E^{cu}_x$ and sufficiently small $\epsilon>0$, there
  exists a strictly increasing sequence $t_k$ such that
  $t_k\nearrow\infty$ and $ \log|\det(D\phi_{t_k}\mid_{F_x})|
  \ge (c_0-\epsilon)t_k$. Since
  $|\det (D\phi_{[t_k]}\mid_{F_x})|=\prod_{i=0}^{[t_k]-1} \frac{|\det
    (D\phi_{i+1}\mid_{F_x})|}{|\det (D\phi_i\mid_{F_x})|}$
  we get\footnote{We write
    $[t]=\sup\{\ell\le t: \ell\in\ZZ^+\}$ for the integer part of
    $t\in\RR$.}
  from~\eqref{eq:approxmutl}
  \begin{align*}
    \left|\log|\det (D\phi_{t_k}\mid_{F_x})^{-1}|
    +\sum\nolimits_{i=0}^{[t_k]-1}\log|\det(Df\mid_{{\wh F}_{f^ix}})|\right|
    \le
    \sum\nolimits_{i=0}^{[t_k]}|\log\xi_i|.
  \end{align*}
  Hence, for all $k$ and each $2$-subspace $F_x\subset E^{cu}_x$,
  we can write
  \begin{align*}
    \sum\nolimits_{i=0}^{[t_k]-1}\log|\det(Df\mid_{F_{f^ix}})^{-1}|
    \le
    (-c_0+\epsilon) t_k + \sum\nolimits_{i=0}^{[t_k]-1}|\log\xi_i|.
  \end{align*}
  Since $\epsilon>0$ was arbitrary, we obtain
  $ \liminf\frac1n\sum\nolimits_{i=0}^{n-1}
  \log\|\wedge^2(Df\mid_{E^{cu}_{f^ix}})^{-1}\| \le -c_0, $ that is,
  (wNU2SE) if $\dim E^{cu}=2$. Analogously, we deduce that
  \textbf{(ASE)$\implies$(NU2SE)}, i.e., (i)$\implies$(A) after exchanging
  $\liminf$ with $\limsup$, and note that no assumption on
  non-convergence to equilibria was needed here. 


  We now easily show both \textbf{(NU2SE) $\implies$ (NUSE)} and
  \textbf{(wNU2SE) $\implies$ (wNUSE)}, i.e., (A)$\implies$(B), without restrictions on
  $d_{cu}$.
For a regular point $x\in U$, any $2$-subspace $F_x\subset E^{cu}_x$
containing $G(x)$ admits an orthonormal basis
$\{G(x)/\|G(x)\|,n(x)\}$\footnote{Recall that $E^{cu}_x\ni G(x)$ for
  $x\in U$ by construction; see
  Subsection~\ref{sec:non-uniformly-expand-1}.} and 
$\wh F_{f^ix}=\pi^s(Df^i_x\cdot F_x)$ admits an orthonormal basis
$\{G(f^ix)/\|G(f^ix)\|,n(f^ix)\}$ for $i\ge1$, where $n(f^ix)$ is a
unit vector on the line $\wh F_{f^ix}\cap G(f^ix)^\perp$\footnote{In
  general, we do not know if $E^{cu}$ is orientable, however the
  choice $\pm n(f^ix)$ is immaterial since we consider the modulus of
  the Jacobian in what follows}.  These choices ensure that
  \begin{align}\label{eq:exterior}
    |\det(D\phi_1\mid_{{\wh F}_{f^ix}})|
    =
    \det
    \begin{pmatrix}
      \|G(f^{i+1}x)\|/\|G(f^ix)\| & \star
      \\
      0 & \|P^1\cdot n(f^ix)\|
    \end{pmatrix}.
  \end{align}
  Let $\G_2(f^ix)$ be the compact family of all $2$-subspaces of
  $E^{cu}_{f^ix}$ which contain $G(f^ix)$.  From~\eqref{eq:exterior}
  we get
  \begin{align*}
    \sup\nolimits_{F\in\G_2(f^ix)}|\det(Df\mid_{F})^{-1}|
    =
    \big(\|G(f^{i}x)\|/\|G(f^{i+1}x)\|\big)\cdot\sup\nolimits_{n\in
    N^{cu1}_{f^ix}} \|P^1\cdot n\|^{-1},
  \end{align*}
  where $N^{cu1}_{f^ix}$ is the unit sphere on the normal bundle
  $N^{cu}_{f^ix}$ and $i\ge1$. Since this is a subset of all
  $2$-subspaces of $E^{cu}_x$, we get
  \begin{align}\label{eq:secLPF0}
    \frac{\|G(f^{i}x)\|}{\|G(f^{i+1}x)\|}
    \cdot \|(P^1\mid_{N^{cu}_{f^ix}})^{-1}\|
    \le
    \|\wedge^2(Df\mid_{E^{cu}_{f^ix}})^{-1}\|.
  \end{align}
  Let $\Gamma(x):= \log\|\wedge^2(Df\mid_{E^{cu}_{x}})^{-1}\|$ and
  $\psi(x):= \log \|(P^1\mid_{N^{cu}_{x}})^{-1}\|$ in what follows.
  Therefore
  \begin{align*}
    \sum\nolimits_{i=0}^{n-1} \Gamma(f^ix)
     +
    \sum\nolimits_{i=0}^{n-1}\log\frac{\|G(f^{i+1}x)\|}{\|G(f^ix)\|}
      &\ge
        \sum\nolimits_{i=0}^{n-1}      \psi(f^ix),
        \quad n\ge1.
  \end{align*}
Since we assume that the future trajectory of $x$
does not converge to any equilibrium, we can use the following.

\begin{lemma}
  \label{le:medG1}
  We have
  $\limsup\frac1n\sum_{i=0}^{n-1}
  \log\frac{\|G(f^{i+1}x)\|}{\|G(f^{i}x)\|} = 0 = \limsup\frac1n
  \log\|G(f^nx)\|$ when $n\nearrow\infty$; and also
  $\limsup\frac1t \log\|G(\phi_tx)\|=0$ when $t\nearrow\infty$, for
  every $x\in U$ whose future trajectory does not converge to any
  equilibrium. In particular, it holds for $\leb$-a.e $x\in U$.
\end{lemma}

Hence, dividing by $n$ and taking $\limsup$ when $n\nearrow\infty$,
Lemma~\ref{le:medG1} together with~\eqref{eq:secLPF0} ensures that
\textbf{(NU2SE)$\implies$(NUSE)}. Analogously, using $\liminf$ provides
\textbf{(wNU2SE)$\implies$(wNUSE)} since, for all $\epsilon>0$ there exists
$N>1$ so that
$\frac1n\sum_{i=0}^{n-1}
\log\frac{\|G(g^{i+1}x)\|}{\|G(g^{i}x)\|}<\epsilon$ for all $n\ge N$,
by Lemma~\ref{le:medG1}. Thus, for any $\epsilon>0$ we get
  \begin{align*}
    \liminf_{n\nearrow\infty}
    \frac1n \sum\nolimits_{i=0}^{n-1} \psi(f^i x)
    \le
      \epsilon+\liminf_{n\nearrow\infty}
    \frac1n\sum\nolimits_{i=0}^{n-1} \Gamma(f^ix)
    \le\epsilon -c_0.
  \end{align*}
  Both \textbf{(wMASE) $\implies$ (wASEP)} and \textbf{(MASE)
    $\implies$ (ASEP)}, i.e., (a)$\implies$(b), follow from the
  analogous to~\eqref{eq:exterior} and~\eqref{eq:secLPF0}: we get for
  the family of $2$-subspaces $F$ of $E^{cu}_x$ with orthonormal basis
  $\{G(x),n(x)\}$ for some $n(x)\in N^{cu1}_x$
  \begin{align}\label{eq:ext-t}
    |\det(D\phi_t\mid_F)|
    =
    \det
    \begin{pmatrix}
      \|G(\phi_tx)\|/\|G(x)\| & \star
      \\
      0 & \|P^t\cdot n(x)\|
    \end{pmatrix}
  \end{align}
  and it follows, as before, that
$
    \frac{\|G(x)\|}{\|G(\phi_tx)\|}
    \cdot
    \|(P^t\mid_{N^{cu}_x})^{-1}\|
    \le
    \|\wedge^2(D\phi_t\mid_{E^{cu}_x})^{-1}\|.
 $
Then, for each $t>0$, we get
$
   \frac1t\left(\log\|\wedge^2(D\phi_t\mid_{E^{cu}_x})^{-1}\|
   +
   \log\frac{\|G(\phi_tx)\|}{\|G(x)\|}\right)
   \ge
   \frac1t\log\|(P^t\mid_{N^{cu}_x})^{-1}\|,
$
and the claimed relations follow just like above using
Lemma~\ref{le:medG1}.

Analogously, we obtain (b)$\implies$(i), that is, \textbf{(wASEP)
  $\implies$ (wASE) } and \textbf{(ASEP) $\implies$ (ASE)}. For any
$2$-subspace $F$ of $E^{cu}_x$ containing $G(x)$ and each fixed $t>0$,
we have~\eqref{eq:ext-t}. For $F$ not containing $G(x)$, we choose a
basis $\{v_i+\omega_i G(x), i=1,2\}$ of $F$, with
$\omega_1\omega_2\neq0$; and $\{v_i^t+\omega_i^tG(\phi_tx), i=1,2\}$
of $F_t:=D\phi_t(x)F$, such that $\{ v_1^t , v_2^t\}$ is
orthonormal. Thus,
$P^{-t}\big(v_i^t+\omega_i^tG(\phi_tx )\big)=v_i, i=1,2$ and
$ |\det(D\phi_t\mid_F)^{-1}|$ is given by
\begin{align*}
  \frac{\| \big(v_1+\omega_1G(x)\big) \wedge
  \big(v_2+\omega_2G(x)\big)\|}{\|\big(v_1^t+\omega_1^tG(\phi_tx)\big)\wedge\big(v_2^t+\omega_2^tG(\phi_tx)\big)\|}
=
  \|(P^{-t}v_1^t) \wedge (P^{-t}v_2^t)\|
  \le
  \|(P^t\mid_{N^{cu}_x})^{-1}\|^2.
\end{align*}
Then
$ |\det(D\phi_t\mid_F)^{-1}| \le \max\left\{
  \frac{\|G(x)\|}{\|G(\phi_tx)\|}\|(P^t\mid_{N^{cu}_x})^{-1}\| ,
  \|(P^t\mid_{N^{cu}_x})^{-1}\|^2 \right\}=\max\{a_t,b_t\}$.
From (wASEP) and Lemma~\ref{le:medG1}, for any $\epsilon>0$ we have
$T>0$ so that  $\|G(\phi_tx)\|>e^{-\epsilon t}$ for all $t>T$, and
then
$\liminf\frac1t\log a_t$ satisfies
\begin{align*}
\liminf_{t\nearrow\infty}\frac1t\log\Big(\frac{\|G(x)\|}{\|G(\phi_tx)\|}
\|(P^t\mid_{N^{cu}_x})^{-1}\|\Big)
  \le
  \epsilon+
\liminf_{t\nearrow\infty}\frac1t\log\|(P^t\mid_{N^{cu}_x})^{-1}\|
\le -c_0+\epsilon.
\end{align*}
Since $\epsilon>0$ is arbitrary, we see that for any sequence
$t_n\nearrow\infty$ there exists $N>1$ so that for all $n>N$ we have
$|\det(D\phi_{t_n}\mid_F)^{-1}|\le a_{t_n}$ and so
\begin{align*}
\limsup_{n\nearrow\infty}\frac1n\log |\det(D\phi_t\mid_F)|
  &=
    \liminf_{n\nearrow\infty}\frac1n\log |\det(D\phi_t\mid_F)^{-1}|
    \le -c_0.
\end{align*}
Since $F$ is an arbitrary $2$-subspace of $E^{cu}_x$, we have obtained
(wASE), as claimed. If we assume (ASEP), then from
Lemma~\ref{le:medG1} we have
\begin{align*}
  \limsup_{t\nearrow\infty}\frac1t\log a_t
  &\le
 \liminf_{t\nearrow\infty}\frac1t\log\frac{\|G(\phi_tx)\|}{\|G(x)\|}
  +
  \limsup_{t\nearrow\infty}\frac1t\log\|(P^t\mid_{N^{cu}_x})^{-1}\|\le -c_0.
\end{align*}
Hence, by the same reasoning as above, we conclude (ASE)
\begin{align*}
  \liminf_{n\nearrow\infty}\frac1n\log |\det(D\phi_t\mid_F)|
  &=
    \limsup_{n\nearrow\infty}\frac1n\log |\det(D\phi_t\mid_F)^{-1}|
    \le -c_0.
\end{align*}
We can now show that (i)$\implies$(ii), that is,
\textbf{(wASE)$\implies$(wPSL)} and
\textbf{(ASE)$\implies$(PSL)}. Indeed, given any
$v\in E^{cu}_x\setminus \RR\cdot G(x)$, we let $F$ be the $2$-subspace
of $E^{cu}_x$ generated by $\{G(x),n(x)\}$, where $n(x)\in N^{cu}_x$
is the orthogonal projection of $v$. Then, assuming (wASE), given
$\epsilon>0$, there exists $t_n\nearrow\infty$ so that
\begin{align*}
  e^{\epsilon t_n}\|P^{t_n}n(x)\|
  \ge
  \frac{\|G(\phi_{t_n}x)\|}{\|G(x)\|}\|P^{t_n}n(x)\|
  =
  |\det(D\phi_{t_n}\mid_F)|
  \ge
  e^{(c_0-\epsilon)t_n}
\end{align*}
from~\eqref{eq:ext-t} and Lemma~\ref{le:medG1}. By definition of the
Linear Poincar\'e Flow, since  $v=n(x)+\omega G(x)$ with 
$\omega\in\RR$, then there exits $\eta>0$ so that
\begin{align*}
  \|D\phi_{t_n}(x)v\|^2
  =
  \|P^{t_n}n(x)+\eta G(\phi_{t_n}x)\|^2
  =
  \|P^{t_n}n(x)\|^2+\|\eta G(\phi_{t_n}x)\|^2
  \ge
  \|P^{t_n}n(x)\|^2.
\end{align*}
Altogether we obtain $\|D\phi_{t_n}(x)v\|\ge
e^{(c_0-2\epsilon)t_n}$. Since $\epsilon>0$ is arbitrary, we get
\textbf{(wPSL)}. Assuming now \textbf{(ASE)}, that is
$\liminf\frac1T\log|\det(D\phi_T\mid F)|\ge c_0$, then we obtain the
same estimates above for all times $t$ larger than some $T>0$ in the
place of $t_n$. We thus conclude that
$\|D\phi_{t}(x)v\|\ge e^{(c_0-2\epsilon)t}$ for all $t>T$. Hence, we
get $\liminf\frac1t\log\|D\phi_{t}(x)v\|\ge c_0$, i.e, (PSL).

 
Finally, for (A)$\implies$(a), that is, \textbf{(wNU2SE) $\implies$
  (wMASE)} and \textbf{(NU2SE) $\implies$ (MASE)}: these follow from
the subadditivity of the continuous map
$(t,x)\in\RR^+\times U
\mapsto\psi_t(x)=\log\|\wedge^2(D\phi_t\mid_{E^{cu}_x})^{-1}\|$
together with the relative compactness of $U$.  Indeed, for any real
$T>0$ we have
  \begin{align*}
    \frac{\log\|\wedge^2(D\phi_T\mid_{E^{cu}_x})^{-1}\|}T
    &\le
     \frac{\log\|\wedge^2(D\phi_{T-[T]}\mid_{E^{cu}_{f^{[T]}x}})^{-1}\|}T
      +
      \sum_{i=0}^{[T]-1}\frac{ \log\|\wedge^2(Df\mid_{E^{cu}_{f^ix}})^{-1}\|}T
    \\
    &\le
      \frac{[T]}T\cdot\frac1{[T]}\sum\nolimits_{i=0}^{[T]-1}
      \Gamma(f^ix) + \frac2T\log L,
  \end{align*}
  where
  $L:=\sup\{ \log\|\wedge^2(D\phi_t\mid_{E^{cu}_x})^{-1}\| : x\in U, 0\le
  t\le1\}$. So letting $T\nearrow\infty$ shows that the discrete time
  average of $\Gamma$ is larger than the asymptotic average, as
  claimed.

  This completes the proof of Theorem~\ref{thm:NUSEvar}, except for
  the proof Lemma~\ref{le:medG1}, which we present next.
\end{proof}

  \begin{proof}[Proof of Lemma~\ref{le:medG1}]
    The sum in the statement is telescopic, so it becomes
    \begin{align}\label{eq:telescop}
        \limsup\nolimits_{n\nearrow\infty}
        (\log\|G(f^nx)\|-\log\|G(x)\|)/n
        =
        \limsup\nolimits_{n\nearrow\infty}
        \log\|G(f^nx)\|^{1/n}.
    \end{align}
    Since $\|G\|$ is bounded above, we get~\eqref{eq:telescop}$\le\limsup_{t\nearrow\infty}\frac1t\log\|G(\phi_tx)\|\le0$. Hence,
    it is enough to find the same lower bound.

    Let us assume that the future trajectory of $x$ does not converge
    to an equilibrium.  We argue by contradiction: if there exists
    $\epsilon_0>0$ so that
    $ \limsup_{t\nearrow\infty} \log\|G(\phi_tx)\|^{1/t}<-\epsilon_0$,
    then we can find $T=T(x,\epsilon_0)>1$ so that for all $t>T$ we
    get $ \|G(\phi_tx)\| \le e^{-\epsilon_0 t}$. This implies that $\phi_tx$
    is converging to the finite subset $\sing_\Lambda(G)$, and so
    converges to some equilibrium of this set. This contradiction
    shows that $\limsup\frac1t\log\|G(\phi_tx)\|=0$. The same argument
    replacing a real $t\nearrow\infty$ with an integer
    $n\nearrow\infty$ completes the proof of the first statement of
    the lemma.
   
    Let $S(\Omega)\subset\sing_\Lambda(G)$ be the collection of
    equilibria accumulated by the future trajectories of points of
    $\Omega$. Then $\sigma\in S(\Omega)$ cannot be a sink (otherwise
    trajectories converging to $\sigma$ cannot have sectional
    expansion) nor a source.  Since $S(G)$ consists of finitely many
    hyperbolic fixed points of saddle-type, then every approach of the
    trajectory to any element $\sigma\in S(G)$ is followed by a
    departure from $\sigma$, except for the points of the stable
    manifold of $\sigma$. But, in this setting, the union
    $\cup_{\sigma\in S(G)}W^s_\sigma$ of the local stable manifolds
    has zero volume. Indeed, the stable manifold of a hyperbolic
    critical element in an immersed submanifold \cite{PM82}, so it
    has zero volume as a subset of the ambient manifold. Hence,
    $U \setminus\cup \{W^s_\sigma:\sigma\in\sing_\Lambda(G)\}$ has
    full volume in $U$.

    Thus, the previous convergence of $f^nx$ or
    $\phi_tx$ to $\sing_\Lambda(G)\supset S(\Omega)$ is impossible on
    a positive Lebesgue measure subset. This completes the proof of
    the lemma.
  \end{proof}


\begin{proof}[Proof of Theorem~\ref{thm:NUE-ASEA}]
  Let $\Gamma(x):=\log\| \wedge^2(Df\mid_{E^{cu}_x})^{-1}\|$ and set
  $E:=\{x\in U:$~\eqref{eq:NUEsal} holds for $x\}$. If
  $\mu(E)=1$ for some ergodic invariant measure $\mu$ for the flow, we
  obtain
  $ \mu(\Gamma) =
  \lim_{t\nearrow\infty}\frac1t\int\Gamma(\phi_tx)\,d\mu(x) = \int
  \lim_{t\nearrow\infty}\frac1t\Gamma(\phi_tx)\,d\mu(x) <-\omega $
  from Kingman's Subadditive Ergodic Theorem, since the limit exists
  and coincides with the limit inferior for $\mu$-a.e. $x$. But by
  subadditivity
  $\int\log\|\wedge^2(D\phi_t\mid_{E^{cu}_x})^{-1}\|\,d\mu(x)$ is
  bounded above by
  \begin{align*}
    &\int\log\| \wedge^2(D\phi_{t-[t]}\mid_{E^{cu}_{\phi_{[t]}x}})^{-1}\| \,d\mu(x)
   +
      \int\log\|\wedge^2(D\phi_{[t]}\mid_{E^{cu}_x})^{-1}\|\,d\mu(x)
    \\
    &\le
      \sup_{x\in U}\sup_{s\in[0,1]}\log\| \wedge^2(D\phi_s\mid_{E^{cu}_x})^{-1}\|
      +
      \sum\nolimits_{i=0}^{[t]-1}\int\Gamma(f^ix)\,d\mu(x)
      = C + [t]\mu(\Gamma)
  \end{align*}
  where $f:=\phi_1$, $C$ is a constant depending on the flow,
  and we used that $\mu$ is $f$-invariant. Therefore, again by the
  Subadditive Ergodic Theorem 
  \begin{align*}
    \int
    \lim_{t\nearrow\infty}\frac1t
    \log\|\wedge^2(D\phi_t\mid_{E^{cu}_x})^{-1}\|\,d\mu(x)
    &=
    \lim_{t\nearrow\infty}\frac1t
    \int\log\|\wedge^2(D\phi_t\mid_{E^{cu}_x})^{-1}\|\,d\mu(x)
    \\
    &\le
    \limsup_{t\nearrow\infty}\frac{[t]}t\cdot\mu(\Gamma)
    <  -\omega.
  \end{align*}
  Since we assume that the above holds for any invariant ergodic
  probability measure $\mu$ for the flow, we are in the conditions of
  the following result.
  
  \begin{proposition}{\cite[Corollary 4.2]{ArbSal2011}}
    \label{pr:arbsal}
    Let $\{t\mapsto f_t:\Lambda\to \RR\}_{t\in \RR}$ be a continuous
    family of continuous function which is subadditive and suppose
    that, for every invariant probability measure $\mu$ for the flow,
    the limit $ \wt{f}(x):=\lim_{t\nearrow\infty}\frac1t f_t(x)$
    defined for $\mu$-a.e. $x$ satisfies $\int \wt{f}\, d\mu<0.$ Then
    there exist constants $T>0$ and $\lambda>0$ such that for every
    $x\in \Lambda$ and every $t>T$ we have
    $f_t(x)\leq -\lambda \cdot t$.
\end{proposition}

If we set $f_t(x):=\log\|\wedge^2(D\phi_t\mid_{E^{cu}_x})^{-1}\|$ we
obtain $T_u,\lambda_u>0$ such that 
\begin{align*}
  \|\wedge^2(D\phi_t\mid_{E^{cu}_x})^{-1}\|\le e^{-\lambda_u t}
  \quad\text{for
  $t\ge T_u$ and all $x\in \Lambda$}.
\end{align*}
Analogously, if we instead set $f_t(x)=\log\|D\phi_t\mid_{E^{cs}_x}\|$
and reapply the same reasoning, we conclude the existence of
$T_s,\lambda_s>0$ so that
$\|D\phi_t\mid_{E^{cs}_x}\|\le e^{-\lambda_s t}$ for all $x\in\Lambda$
and $t\ge T_s$.

This ensures that the compact set $\Lambda$ with the invariant
continuous splitting $E^{cs}\oplus E^{cu}$ is a sectional-hyperbolic
set, and concludes the proof of the first statement of the theorem.

For the second statement, we write
$\psi_\delta(x):=\int_0^\delta\Gamma(\phi_tx) \, dt$ where
$\delta=1/m>0$ for some $m>1$ so large that for $0<\zeta<\omega/4$ we get
$
\exp\big(\Gamma(\phi_\delta x)-\Gamma(x)\big)
\le 1+\zeta$ and so
$|\Gamma(\phi_\delta x)- \Gamma(x)|\le \log(1+\zeta) < \omega/4 $ for
all $x\in U$.  Then $\psi_\delta(x)=\Gamma(x)+\xi(x)$ where
$\xi:U\to(-\zeta,\zeta)$ is a continuous function.

  \begin{lemma}
    \label{le:contdiscrete}
    Let us consider a point $x\in U$ and a sequence
    $T_n\nearrow\infty$ satisfying 
  \begin{align}\label{eq:Tn}
    \int_0^{T_n}\Gamma(\phi_tx) \, dt < -\omega T_n
  \end{align}
  for all large $n>1$. Then there exists $0\le \wh{j}<m$ so that for
  $n$ large enough, with $\ell=\ell_n=[T_n/\delta]/m$ and
  $g:=\phi_\delta$, we have
  $ \sum\nolimits_{i=0}^{\ell-1}\Gamma(f^ig^{\wh{j}}x) < -\ell
  \omega/4m.  $
  \end{lemma}
  
  \begin{proof}
    We rewrite the integral in~\eqref{eq:Tn} as follows
    \begin{align}
      \label{eq:intGamma}
      \int_0^{T_n}\Gamma(\phi_tx)\,dt
      =
         \int_{[T_n]}^{T_n-\delta[T_n/\delta]}\Gamma(\phi_tx) \, dt
      +\sum\nolimits_{i=0}^{[T_n/\delta]-1}\psi_\delta(g_\delta^ix).
    \end{align}
    Since the integrand $\Gamma$ is uniformly bounded by a constant
    $C$ and $T_n-\delta[T_n/\delta]<\delta$,
    can rewrite~\eqref{eq:Tn} as
$
      \sum\nolimits_{i=0}^{[T_n/\delta]-1}\psi_\delta(g^ix)
      <
      -\omega T_n + C\delta
      <
      (-\omega + C\delta/T_n)T_n
      < - [T_n]\omega/2
 $
    for all $n$ large enough. We group iterates that are time-$1$
    apart rewriting the summation as follows: we let
    $[T_n/\delta]=\ell m$ for some $\ell=\ell_n>1$ and 
    \begin{align*}
      \sum\nolimits_{j=0}^{m-1}
      \sum\nolimits_{i=0}^{\ell-1}\psi_\delta(f^ig^jx)
      <
      - [T_n/\delta]\delta \omega/2
      =
      - \ell \omega/2.
    \end{align*}
    So, one of the inner sums is negative.  More precisely, there
    exists $\wh{j}=\wh{j}(x)\in\{0,\dots,m-1\}$ so that
$
      \sum_{i=0}^{\ell-1}\psi_\delta(f^ig^{\wh{j}}x)
      <
      -\ell \omega/2m.
 $
    By the choice of $\delta$ we obtain
    \begin{align*}
      \sum\nolimits_{i=0}^{\ell-1}
      \big(\Gamma(f^ig^{\wh{j}}x)\pm\xi(f^ig^{\wh{j}}x)\big)
     <
      -\ell \omega/2m
      \implies
      \sum\nolimits_{i=0}^{\ell-1}\Gamma(f^ig^{\wh{j}}x)
      <
      -\ell \omega/2m+\zeta
      <
      -\ell \omega/4m
    \end{align*}
    for $n$ large enough, with $\ell=\ell_n=[T_n/\delta]/m$, as in the
    statement of the lemma.
  \end{proof}  
  We define
$
    E_j
    :=
    \left\{x\in U:
    \liminf\nolimits_{\ell\nearrow\infty}
    \sum\nolimits_{i=0}^{\ell-1}\Gamma(f^ig^jx)^{1/\ell} < -\omega/4m
    \right\}
    $
for each $0\le j < m$.
  If we assume that $\leb(E)>0$, then from Lemma~\ref{le:contdiscrete}
  we have that $E\subset\cup_{j=0}^{m-1}E_j$ and so there exists
  $\wh{j}\in\{0,\dots,m-1\}$ such that $\leb(E_{\wh{j}})>0$. By
  smoothness of $f$, the set $\Omega=g^{-\wh{j}}E_{\wh{j}}$ satisfies
  $\leb(\Omega)>0$ and every $x\in \Omega$ is such that
  \begin{align*}
    \liminf_{n\nearrow\infty} \frac1n\sum\nolimits_{i=0}^{n-1}
    \Gamma(f^ix)
    \le
    -\omega/4m.
  \end{align*}
This is (wNU2SE) with $c_0=\omega/4m$, completing the proof of the
first part of the second statement of Theorem~\ref{thm:NUE-ASEA}.

Exchanging $\liminf$ by $\limsup$ in $E$ and $E_j$ enables us to
follow the same reasoning to show that~\eqref{eq:NUEsal+}, on a
positive volume subset of $U$, implies (NU2SE) on a positive volume
subset of $U$ with a slightly smaller rate $c_0=\omega/m$.  This
completes the proof of Theorem~\ref{thm:NUE-ASEA}.
\end{proof}


  





\section{Hyperbolic times and center-unstable disks}
\label{sec:hyperb-times-center}

In this section we start the proof of Theorem~\ref{mthm:discretefabv}.
We only use the domination of the splitting and hyperbolic times along
the sectional center-unstable direction. We assume that (NUSE) hold
for a positive Lebesgue measure subset $\Omega$ as in~\eqref{eq:NUE};
see Theorem~\ref{thm:NUSEvar}.

\subsection{Hyperbolic (Pliss) times for the Linear Poincar\'e Flow}
\label{sec:hyperb-pliss-times}

The following is a very useful tool 
which enables us to use \emph{hyperbolic times}.

\begin{lemma}{\cite[Lemma 3.1]{ABV00}}
  \label{le:pliss} Let $A\ge c_2 > c_1$ be real numbers and
  $\zeta={(c_2-c_1)}/{(A-c_1)}$. Given real numbers $a_1,\ldots,a_N$
  satisfying
  \begin{align*}
 \sum\nolimits_{j=1}^N a_j \ge c_2 N
\qand a_j\le A \;\;\mbox{for all}\;\; 1\le j\le N,
  \end{align*}
  there are $\ell>\zeta N$ and $1<n_1<\ldots<n_\ell\le N$ such that
  \begin{align*}
    \sum\nolimits_{j=n+1}^{n_i} a_j \ge c_1\cdot(n_i-n) \;\;\mbox{for
    each}\;\; 0\le n < n_i, \; i=1,\ldots,\ell.
  \end{align*}
\end{lemma}


Let us fix $x\in\Omega$ satisfying~\eqref{eq:NUE}.  Since
$(P^1\mid_{N^{cu}_x})^{-1}=\cO_x\cdot Df^{-1}\mid_{P^1(N^{cu}_x)}$,
then $\|(P^1\mid_{N^{cu}_x})^{-1}\|\le\|Df^{-1}\|\le e^L$ with
$L=\sup_{x\in U}\|DG(x)\|$, which is finite because $U$ is relatively
compact.  We also have
$\|(P^1\mid_{N^{cu}_x})^{-1}\|\ge\|P^1\mid_{N^{cu}_x}\|^{-1}\ge
e^{-L}$, and thus
$A_0=\sup_{x\in U}\big|\log\|(P^1\mid_{N^{cu}_x})^{-1}\|\big|\le e^L$.
Then we apply Lemma~\ref{le:pliss} to
$a_i=-\log\|(P^1\mid_{N^{cu}_{f^{i-i}x}})^{-1}\|$ for $i=1,\dots, N$
so that $\sum_{i=1}^Na_i\ge c_0 N /2$ -- this inequality holds for all
large enough $N=N(x)>1$.

We obtain $\ell>\zeta_1 N$ with
$\zeta_1=(c_0/2-c_0/4)/(A_0-c_0/4)=c_0/(4A_0-c_0)>0$ and times
$1<n_1<\ldots<n_\ell\le N$ such that
\begin{align}\label{eq:hyptimex}
  \prod\nolimits_{j=n}^{n_i-1}\|(P^1\mid_{N^{cu}_{f^jx}})^{-1}\|
  \le
  e^{-c_0(n_i-n)/4}, \quad 0\le n<n_i, \quad i=1,\ldots,\ell.
\end{align}
We say that $n_i$ is a \emph{hyperbolic time} for $x$ if
$\sing_\Lambda(G)=\emptyset$. In the presence of equilibria, we need
to control the visits of the future orbit of $x$ near these fixed
points where the Linear Poincar\'e flow is not defined. For that, we
reapply Lemma~\ref{le:pliss} to~\eqref{eq:SR}, as follows.

For $\epsilon_0\in(0,\zeta_1c_0/32)$ we take $\delta_0>0$ satisfying 
\eqref{eq:SR} for all $x\in\Omega$, so that for some
$N(x)>1$ we have
\begin{align*}
  n\ge N(x) \implies
  \sum\nolimits_{i=0}^{n-1}\log d_{\delta_0}
  \big(f^i(x),\sing_\Lambda(G)\big) \ge -2\epsilon_0\cdot n.
\end{align*}
Since the summands are non-positive, we can take $A=0$,
$c_2=-2\epsilon_0$ and $c_1=-c_0/16$ to obtain
$\zeta_2=(c_2-c_1)/(A-c_1) = 1-c_2/c_1>1-\zeta_1$. Hence
$\zeta_1+\zeta_2>1$ and for $\zeta=\zeta_1+\zeta_2-1>0$ we
have $\ell\ge\zeta N$ and times $1<n_1<\ldots<n_\ell\le N$
simultaneously satisfying~\eqref{eq:hyptimex} and the conclusion of
Pliss' Lemma for the last summation. We have proved the following.
\begin{proposition}
  \label{pr:hyptimes}
  For each sufficiently small $\epsilon_0>0$, we can find a small
  enough $\delta_0>0$ such that there are $\theta,\epsilon_0$ and for
  $x\in\Omega$ we can find $N=N(x)\in\ZZ^+$ so that for any given
  integer $T\ge N$, there exists $\ell\ge\theta T$ and times
  $1<n_1<\cdots<n_\ell\le T$ satisfying~\eqref{eq:hyptimex} and
  \begin{align}\label{eq:SRtimex}
    d_{\delta_0}\big(f^jx,\sing_\Lambda(G)\big)
    >
    e^{- c_0(n_i-j)/16}, \quad 0\le j<n_i, \quad i=1,\dots,\ell.
  \end{align}
\end{proposition}

The times $n_i$ satisfying the conclusion of
Proposition~\ref{pr:hyptimes} will be referred to as \emph{hyperbolic
  times} for $x$ when $\sing_\Lambda(G)\neq\emptyset$.

\begin{remark}[uniform positive asymptotical density]
  \label{rmk:upad}
  This shows that for every $x\in\Omega$ satisfying~\eqref{eq:NUE}
  there exists a subset $HT(x)\subset\ZZ^+$ formed of hyperbolic times
  so that $\limsup_{n\nearrow\infty}\#(HT(x)\cap[1,n])/n \ge\theta>0$.
\end{remark}

\subsection{Estimates for nearby points and roughness of hyperbolic
  times}
\label{sec:hyperb-times-along}

We observe that the map $x\in U\mapsto E^{cu}_x$ is
H\"older-continuous, by the domination of the splitting, see
e.g. \cite[Subsection 4.2]{ArMel17}. In addition, both
$x\mapsto Df(x)$ and $x\in U^*\mapsto \cO_x$ are Lipschitz, where
$U^*=U\setminus\sing(G)$, because $G$ is of class $C^2$ and the unit
vector field $\hat G:=G/\|G\|$ defined in $U^*$ has derivative
$\hat G(x)'= \cO_x\circ DG(x)\cdot\hat G_x$ whose norm is uniformly
bounded from above. Hence
$\Psi:U^*\times U^*\to\RR, (x,y)\mapsto \log
\frac{\|(P^1\mid_{N_y^u})^{-1}\|}{\|(P^1\mid_{N_x^u})^{-1}\|}$ is
H\"older-continuous and $\Psi(x,x)=0$ for all $x\in U^*$. Therefore,
there exists a constant $C_2>0$ and an exponent $\omega\in(0,1)$ so
that $\Psi(x,y)\le C_2\cdot d(x,y)^\omega$.

We recall that $\rho_0>0$ is such that
$\big(\exp_x\mid_{B(0,\rho_0)}\big)^{-1}$ is well-defined at every
$x\in U$.  

For $z\in M$ with $G(z)\neq\vec0$ we
define the cone
\begin{align*}
  C^\perp_a(z)=\{v+\lambda G(z): v\in G(z)^\perp, \; \lambda\in\RR \;\&\; \|v\|\le
  a\|\lambda G(z)\|\}.
\end{align*}
We let $a>0$ be small enough so that for $y\in U^*$
\begin{align*}
  \|\cO_y\cdot Df^{-1}(f(y))v\|
  \le
  e^{c_0/16}
  \big\|\big(P^1\mid_{N^{cu}_y}\big)^{-1}\big\|\cdot\|v\|,
  \quad
  v\in C^{cu}_a(f(y))\cap C^\perp_a(f(y)).
\end{align*}
We choose $0<\rho_1\le\min\{\delta_0,\rho_0,1\}$
such that $C_2\rho_1^\omega<c_0/16$ and, for each
$x,y\in U^*$ with both $d(x,y)<\rho_1$ and
$d(x,y)<d(x,\sing_\Lambda(G))/2$, then 
together with the H\"older condition on $\Psi$ we get
for $v\in C^{cu}_a(f(y))\cap C^\perp_a(f(y))$
\begin{align}
  \label{eq:unifDfaway}
  \|\cO_y\cdot Df^{-1}(f(y))v\|
  \le
  e^{c_0/8}\cdot
  \big\| \big(P^1\mid_{N^{cu}_x}\big)^{-1}\big\|\cdot\|v\|.
\end{align}

We recall that $\delta_0$ is small enough satisfying
Proposition~\ref{pr:hyptimes}.

\begin{lemma}\label{le:boundsigma}
  There exists $b_*>0$ so that for each $x\in U^*$, if
  $d(x,\sing_\Lambda(G))<\delta_0$, then 
  \begin{align}\label{eq:lowerbdd}
    b_*L\le \frac{\|G(x)\|}{d(x,\sigma)}\le L
    \qand
    2\cdot d(y,x)<d(x,\sing_\Lambda(G))
    \implies
    \|G(x)\|
    \ge
    b_* \|G(x)\|.
  \end{align}
\end{lemma}

\begin{proof}
  Since all equilibria in $\sing_\Lambda(G)$ are hyperbolic, there
  are at most finitely many and those accumulated by the orbit of
  $x\in\Omega$ are of saddle type. Thus, there exists $b>0$ so that
  $\|(DG(\sigma))^{-1}\|<1/b$ for all
  $\sigma\in\omega_G(x)\cap\sing_\Lambda(G)$.

  Moreover, we have $\|DG(y)-DG(\sigma)\|\le \kappa_0 d(y,\sigma)^\beta$
  for all $y\in B(\sigma,2\delta_0)$ and some constants $\kappa_0>0$
  and $0<\beta\le1$ since $G$ is of class $C^{1+}$ --- here and in the
  following estimates, we identify $B(\sigma,\rho_0)$ with the
  $\rho_0$-ball on $T_\sigma M$. Hence, if 
  $2\kappa_0\delta_0^\beta<b$, then by the Mean Value Inequality
  \begin{align*}
    \|G(x)-G(\sigma)-DG(\sigma)\cdot(x-\sigma)\|
    \le
    \kappa d(x,\sigma)^\beta\|x-\sigma\|
    \le
    \kappa_0\delta_0^\beta d(x,\sigma)
  \end{align*}
  and so
  $ \|G(x)\| \ge \|DG(\sigma)\cdot(x-\sigma)\|-\kappa_0\delta_0^\beta d(x,\sigma)
  \ge (b -\kappa_0\delta_0^\beta) d(x,\sigma)$. On the other hand, by the
  smoothness of $G$ we have that
  $\|G(x)\|=\|G(x)-G(\sigma)\|\le L \cdot d(x,\sigma)$.  Finally, if
  $2d(y,x)<d(x,\sigma)$, then
  \begin{align*}
    \|G(y)\| \ge
    (b -\kappa_0\delta_0^\beta) d(y,\sigma)
    \ge
    \frac{b -\kappa_0\delta_0^\beta}2 d(x,\sigma)
    \ge
    \frac{b -\kappa_0\delta_0^\beta}{L} \|G(x)\|
  \end{align*}
  which completes the proof after setting
  $b_*=(b-\kappa_0\delta_0^\beta)/L$.
\end{proof}

Now we show that hyperbolic times are rough along a trajectory, in the
following sense.

\begin{proposition}\label{pr:roughyptime}
  There exists $s_0>0$ small so that, if $n>1$ is a hyperbolic time
  for $x\in U^*$, then $n$ is also a hyperbolic time for $\phi_sx$, for
  all $|s|<s_0$, but with a contracting rate of $e^{-c_0/8}$.
\end{proposition}

\begin{proof}
  First choose $s_0>0$ small enough so that $d(z,\phi_sz)<\rho_1$ for
  all $z\in U$ and $|s|<s_0$, and use~\eqref{eq:unifDfaway} to obtain
  \begin{align*}
    \prod\nolimits_{i=n-k}^{n-1}\|\big(P^1\mid_{N^{cu}_{\phi_sf^ix}}\big)^{-1}\|
    &=
      \prod\nolimits_{i=n-k}^{n-1}\left(\frac{\|\big(P^1\mid_{N^{cu}_{\phi_sf^ix}}\big)^{-1}\|}{\|\big(P^1\mid_{N^{cu}_{f^ix}}
      \big)^{-1}\|} \cdot \|\big(P^1\mid_{N^{cu}_{f^ix}}\big)^{-1}\|\right)
    \\
    &\le
      e^{kc_0/8}\cdot e^{-kc_0/4}=e^{-kc_0/8};
      \quad k=1,\dots,n; |s|<s_0.
  \end{align*}
  Then, if $d(f^ix,\sing_\Lambda(G))<\delta$, we note that
  $d(\phi_sf^ix,\sing_\Lambda(G))$ is bounded from below by
  \begin{align*}
      d(f^ix,\sing_\Lambda(G))-d(\phi_sf^ix,f^ix)
      \ge
      d(f^ix,\sing_\Lambda(G))-|s|\sup\nolimits_{|t|<s}\|G({\phi_tf^ix})\|.
  \end{align*}
  Lemma~\ref{le:boundsigma} provides $\|G({\phi_tf^ix})\|\ge
  b_*\|G(x)\|$ whenever $2
  d(\phi_tf^ix,f^ix)<d(f^ix,\sing_\Lambda(G))$, which holds for
  $|t|<d(f^ix,\sing_\Lambda(G))/(2\|G(x)\|)\le (2b_*L)^{-1}$. In this
  case, we obtain
  \begin{align*}
   d(\phi_sf^ix,\sing_\Lambda(G))
    &\ge
      d(f^ix,\sing_\Lambda(G))-|s|\cdot b_*\|G(x)\|
      \ge
      (1-b_*L\cdot|s|)\cdot d(f^ix,\sing_\Lambda(G)).
  \end{align*}
  Hence, choosing $s_0\in(0,(2b_*L)^{-1})$ small enough so that
  $d(z,\phi_sz)<\rho_1$ for all $z\in U$ we can assume without loss of
  generality that
  \begin{align*}
    d_\delta(\phi_sf^ix,\sing_\Lambda(G))
    \ge
    e^{-L}d_\delta(f^ix,\sing_\Lambda(G))
    \ge
    e^{-L}e^{-(n-i)c_0/16}; \quad i=0,\dots, n; |s|<s_0.
  \end{align*}
  The above conclusions show that, modulo a small change of rates, $n$
  is still a hyperbolic time for $\phi_sx$ for each $|s|<s_0$.
\end{proof}

\subsection{Distortion bounds at hyperbolic times along the sectional
  center-unstable direction}
\label{sec:distort-at-hyperb}


We fix a $cu$-disk $D\subset U$ so that $x\in D$ admits $n>1$ as a
hyperbolic time with a choice of $\epsilon_0,\delta_0>0$ satisfying
Proposition~\ref{pr:hyptimes}.

We set $\Sigma_z=\exp_z(B(z,\rho_1)\cap G(z)^\perp)$ as a
cross-section to $G$ through $z\in U^*$;
$\Sigma_z^i=\exp_z(B(z,(\rho_1/2)e^{-ic_0/16-L})\cap G(z)^\perp)$ a
scaled cross-section; $D_n=f^n(D)$ and $D_n^\perp(z)=D_n\cap \Sigma_z$
for each $z\in D_n$, that is, a section of $D_n$ through $z$ in the
direction orthogonal to the vector field.

We then consider the Poincar\'e first hitting maps
$R_i:\dom(R_i)\subset\Sigma_{f^i(x)}^{n-i}
\to\Sigma_{f^{i+1}(x)}^{n-i-1}$ and note that
$DR_i(f^i(x))=P^1_{f^ix}:N_{f^ix}\to N_{f^{i+1}x}$, for
$i=0,\dots,n-1$; see Figure~\ref{fig:Poincare}.

We note that \emph{these maps are well-defined even if $U$ contains
  equilibria}, by the distance bound provided by the
condition~\eqref{eq:SRtimex}.

\begin{lemma}[Local sectional continuity]\label{le:unifDf}
  For each $i=0,\dots,n-1$, let
  $D_i^\perp=f^i(D)\cap\Sigma_{f^{i+1}x}^{n-i-1}$.  Then the
  Poincar\'e maps satisfy
  $ \| DR_i(R_iy)^{-1}\mid_{T_{R_iy}D_i^\perp}\| \le e^{c_0/8}\cdot
  \|(P^1\mid_{N_{f^ix}^u)}^{-1}\|$, for $ y\in\dom(R_i)$.
\end{lemma}

\begin{proof}
  We have $y\in\Sigma_{f^i x}^{n-i}$ by  definition of $R_i$
  and so $2\cdot d(y,f^ix) \le \rho_1 e^{-(n-i)c_0/16-L}$.
  Moreover, 
    $d_{\delta_0}(f^ix,\sing_\Lambda(G))>e^{-(n-i)c_0/16}$
    with $\delta_0<\rho_1\le1$. Hence
    $2\cdot d(y,f^ix) \le d_{\delta_0}(f^ix,\sing_\Lambda(G))$ and
    thus
    $2\cdot
    d(y,\sing_\Lambda(G))>d_{\delta_0}(f^ix,\sing_\Lambda(G))$. At
    this point, we divide the argument into two cases.
  \begin{description}
  \item[Away from equilibria] If
    $d\big(f^ix,\sing_\Lambda(G)\big)\ge\delta_0\ge\rho_1$, then 
    $\Sigma_{f^ix}^{n-i}$ is away from $\sing(G)$ and the Poincar\'e
    time from $\dom(\Sigma_{f^ix}^{n-i})$ to
    $\Sigma_{f^{i+1}x}^{n-i-1}$ is between $1-\xi$ and $1+\xi$ for
    some uniform small $\xi>0$ depending on $\rho_1$.
  \end{description}
This ensures that
  $R_iy=(\phi_s\circ f)y$ with $s=s(y)$ such that $|s-1|<\xi$ and so
  for $y\in\dom(R_i)$ and $v\in T_{R_iy}D_{i+1}^\perp$ 
  \begin{align}
    \|DR_i(R_iy)^{-1}v\|
    &= \nonumber
    \|\cO_y\cdot [D(\phi_s\circ f)(R_iy)]^{-1}v\| 
    \\
    &=\label{eq:sections}
    \|\cO_y\cdot Df^{-1}(\phi_{-s}R_iy)\cdot D(\phi_s)^{-1}(R_iy)v\|
    \\
    &=\nonumber
    \|\cO_y\cdot Df^{-1}(fy)\cdot D(\phi_s)^{-1}(R_iy)v\|.
  \end{align}
  The time $s=s(y)$ can be seen as the Poincar\'e first visit time
  from the cross-section $S=f(\Sigma_{f^ix}^{n-i})$ to
  $\Sigma_{f^{i+1}x}^{n-i-1}$, and so $D(\phi_s)^{-1}(R_iy)v\in
  T_{fy}(S\cap f^{i+1}(D))\subset C^{cu}_a(fy)\cap C^\perp_a(fy)$ by
  the proximity between $R_iy, fy$ and $f^{i+1}x$. Then the statement
  of the lemma follows from~\eqref{eq:unifDfaway}.

  \begin{description}
  \item[Close to equilibria] Otherwise,
    $d\big(f^ix,\sing_\Lambda(G)\big)<\delta_0$ and
    $\Sigma_{f^ix}^{n-i}$ is close to an equilibrium
    $\sigma\in\sing_\Lambda(G)$.
  \end{description}
  We show that we can repeat the above argument by obtaining a flow
  box from $\dom(R_i)$ to $\Sigma_{f^{i+1}x}^{n-i-1}$ with flight time
  bounded from above.
  
  Reducing $\delta_0$ if necessary, we may assume, without loss of
  generality, that the flow on $B(\sigma,2\delta_0)$ is topologically
  conjugated to the flow of $\dot X=DG(\sigma)\cdot X$, \emph{because
    $\sigma$ is hyperbolic}. That is, there exists a bi-H\"older
  homeomorphism $h:B(\sigma,2\delta_0)\to \RR^m$ so that
  $(h\circ \phi_t)(z)=e^{t\cdot DG(\sigma)}h(z)$ for $t>0$ such that
  $\phi_{[0,t]}z\subset B(\sigma,2\delta_0)$; see
  e.g.~\cite{BarVal07}. We arrange so that
  $\RR^m=\RR^{u}\times\RR^{s}$ is the decomposition into stable and
  unstable subspaces of $DG(\sigma)$, which decomposes in block form as
  $\diag\{A_u,A_s\}$. We identify $f^ix$ with
  $(v,w)\in\RR^{u}\times\RR^{s}$ and then $f^{i+1}x$ becomes
  $(v_1,w_1)=(e^{A_u}v,e^{A_s}w)$.

\begin{figure}[h]
  \centering \includegraphics[width=14cm]{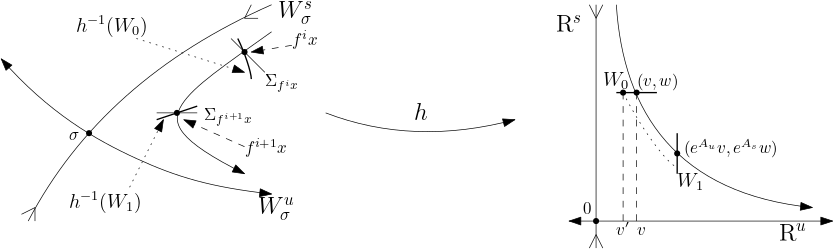}
  \caption{\label{fig:Grobman} The estimation of the flight time in
    the linearized setting.}
\end{figure}
  
For a point in $\Sigma_{f^ix}^{n-i}$ to arrive at
$\Sigma_{f^{i+1}x}^{n-i-1}$, the flow time is close to the time $\tau$
it takes a point $(v',w)\in W_0=\RR^u\times\{w\}$, with $\|v\|-\|v'\|$
small but positive, to arrive at the segment
$W_1=\{e^{A_u}v\}\times\RR^s$; see Figure~\ref{fig:Grobman}.  For each
$y\in\Sigma_{f^ix}^{n-i}$ we have $h(y)=(v',w)$ and for some
$\kappa>0$ and $\beta\in(0,1)$ we get
$ \|h(y)-h(f^ix)\| \le \kappa d(y,f^ix)^\beta \le
\frac{\kappa}{2^\beta} d(f^ix,\sigma)^\beta$ and
$ d(f^ix,\sigma) \le \kappa d(h(f^ix),h(\sigma))^\beta = \kappa
\|h(f^ix)\|^\beta$, then
\begin{align*}
  \|h(y)\|
  &\ge
    \|h(f^ix)\|-\frac{\kappa}{2^\beta} d(f^ix,\sigma)^\beta
    =
    \|h(f^ix)\|
    \left(
    1-\frac{\kappa}{2^\beta}\frac{d(f^ix,\sigma)^\beta}{\|h(f^ix)\|}
    \right)
  \\
  &\ge
    \|h(f^ix)\|
    \left(
    1-\frac{\kappa^{1+\beta}}{2^\beta} d(f^ix,\sigma)^{\beta+\beta^{-1}}
    \right)
  \\
  &\ge
    \|h(f^ix)\|
    \left(
    1-\frac{\kappa^{1+\beta}}{2^\beta} \delta_0^{\beta+\beta^{-1}}
    \right)
    \ge
    \frac12 \|h(f^ix)\|,
\end{align*}
if $\delta_0>0$ is small enough, depending only on $G$.  Using the
previous bound, we estimate
$ \|e^{\tau A_u} v'\| =\|v_1\| \ge e^{\tau\|A_u^{-1}\|^{-1}}\|v'\|$
since $\|v\|\le 2\|v'\|$ and $e^{\tau A}\cdot h(y)=(v_1,w')$ for some
$w'$, and so we obtain
\begin{align*}
    \tau
    \le
    \|A_u^{-1}\|\log\frac{\|v_1\|}{\|v'\|}
    =
    \|A_u^{-1}\|
    \log\left(\frac{\|v_1\|}{\|v\|}\cdot\frac{\|v\|}{\|v'\|}\right)
    \le
    \|A_u^{-1}\|
    \log\left(2\frac{\|v_1\|}{\|v\|}\right)
    =
    \|A_u^{-1}\|
    \log(2e^{\|A_u\|})
\end{align*}
because $v_1=e^{A_u}v'$ satisfies
$\|v_1\|=\|e^{A_u}v'\|\le e^{\|A_u\|}\|v'\|$.  This shows that $\tau$
is bounded depending only on $DG(\sigma)$.

Going back to the original coordinates, the cross-sections
$h^{-1}(W_i), i=0,1$ touch $\Sigma_{f^ix}^{n-1}$ and
$\Sigma_{f^{i+1}x}^{n-i-1}$ at $f^ix, f^{i+1}x$, respectively; see
Figure~\ref{fig:Grobman}.  The vector field in between these sections
has norm uniformly bounded away from zero and close to $G({f^ix})$, by
the estimate~\eqref{eq:lowerbdd}. Hence, the flight time is also
bounded above depending only on $G$ in a neighborhood of $\sigma$.
  
We have recovered a flow box with bounded flight time from $\dom(R_i)$
to $\Sigma_{f^{i+1}x}^{n-i-1}$. We can thus finish repeating the
argument as before, using~\eqref{eq:sections}
and~\eqref{eq:unifDfaway}.
\end{proof}

\begin{figure}[h]
  \centering \includegraphics[width=14cm]{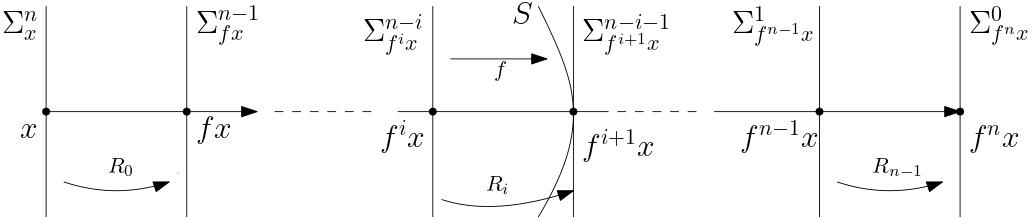}
  \caption{\label{fig:Poincare} The Poincar\'e first hitting time maps
    $R_i$.}
\end{figure}

We write $g_k=R_{n-1}\circ\dots\circ R_{n-k}, 1\le k\le n$ in what
follows, and $\dist_D(x,y)$ for the distance between two
points $x,y$ in the disk $D$, measured along $D$ as the least length
of smooth curves from $x$ to $y$ within $D$.

\begin{lemma}[Local sectional backward contraction]
\label{le:contraction}
Given any $cu$-disk $D\subset U$ tangent to the centre-unstable cone
field, $x\in D$ and $n \ge 1$ a hyperbolic time for $x$
\begin{align*}
  \dist_{D_{n-k}^\perp(x)}  (f^{n-k}(x), g_k(y)  )
  \le
  e^{-k c_0/8}\cdot \dist_{D_n^\perp(x)}(  f^n(x), g_n(y)),
  \quad k=1,\dots,n;
\end{align*}
for each $y\in D_0^\perp(x)$ satisfying
$\dist_{D_n^\perp(x)}(f^nx,g_ny)\le\rho_1$. Moreover, there exists
$\tau_*>0$ such that, defining $t_k$ as the least positive real so
that $g_k(y)=\phi_{t_{n-k}}(y)$, then $|t_{n-k}-(n-k)|\le \tau_*$.
\end{lemma}
        
\begin{proof}
  We follow the proof of \cite[Lemma 2.7]{ABV00}.  Let $\gamma_0$ be a
  curve of minimal length in $D_n^\perp(x)$ from $f^nx$ to $g_ny$ and
  let $ \gamma_k=(g_k)^{-1}(\gamma_0), k=1,\dots,n$. Arguing by
  induction, let us assume that for some $k=1,\dots,n$ we have the
  following bound for the length:
  $\ell(\gamma_k)\le e^{-kc_0/8}\ell(\gamma_0)$. The
  choice of $\rho_1$ in~\eqref{eq:unifDfaway} together with the
  definition of hyperbolic times and Lemma~\ref{le:unifDf} ensure
  that 
  \begin{align*}
    \big\|D(R_{n-1}\circ\dots\circ R_{n-k-1})^{-1}\gamma_0'(z)\big\|
    &\le
    e^{(k+1)c_0/8}\|\gamma_0'(z)\|
    \prod_{j=n-k}^n \|(P^1\mid_{N_{f^jx}^u})^{-1}\|
    \\
    &\le
    e^{-(k+1)c_0/8}\|\gamma_0'(z)\|
  \end{align*}
  where $\gamma_0'(z)$ denotes the tangent vector to $\gamma_0$ at
  $z$. Therefore
  \begin{align*}
    \ell(\gamma_{k+1})
    \le
    e^{-(k+1)c_0/8}\ell(\gamma_0)
    =
    e^{-(k+1)c_0/8}\dist_{D_n^\perp(x)}(  f^nx, g_ny)
    \le
    \rho_1e^{-(k+1)c_0/8}
    \le\rho_1,
  \end{align*}
  which shows that the maps are well-defined on their domains and
  completes the induction. Finally, as a standard consequence of
  Gronwall's Inequality, if $y\in D_0^\perp(x)$ is such that
  $\dist_{D_n^\perp(x)}(f^nx,g_ny)\le\rho_1$ and $k=1,\dots,n$, then
  \begin{align*}
    d(f^{n-k+1}x,fg_ky)
    \le
    e^L d(f^{n-k}x,g_ky)
    \le
    e^L \dist_{D^\perp_{n-k}(x)}(f^{n-k}x,g_ky)
    \le
    \rho_1e^Le^{-kc_0/8}.
  \end{align*}
  Hence, the time $\tau_k$ it takes for $f (g_k y)$ to arrive at
  $\phi_{t_{n-k+1}}y = g_{k-1}y \in D_{n-k+1}^\perp(x)$ is bounded from
  above by the above distance divided by the speed of flow. Thus,
  since for $z$ in a $\rho_1e^{-(n-k+1)c_0/8}$-neighborhood of
  $f^{n-k+1}x$, we have either $\|G(z)\|\ge\gamma_0$, or
  $\|G(z)\|\ge b_*\|G({f^{n-k+1}x})\|$, we get from
  Lemma~\ref{le:boundsigma} in the worst case
  \begin{align*}
    |\tau_k|
    \le
    \frac{\rho_1e^{L-kc_0/8}}{b_*\|G({f^{n-k+1}x})\|}
    \le
    \frac{\rho_1e^{L-kc_0/8}}{b_*^2 L d(f^{n-k+1}x,\sigma)}
    \le
    \frac{\rho_1e^{L-kc_0/8}}{b_*^2 L e^{-(k-1)c_0/16}}
    =
    \frac{\rho_1e^{L+c_0/16}}{b_*^2 L}e^{-(k-1)c_0/16}.
  \end{align*}
  This shows that $t_{n-k+1}-t_{n-k}=1+\tau_k$. Since $t_0=0$, then
  \begin{align*}
    |t_{n-k}-(n-k)|
    =
    \big|\sum_{i=k}^{n-1}(t_{n-i+1} -t_{n-i})-(n-k)|
    \le
    \sum_{i=k}^{n-1} |\tau_i|
    \le
    \frac{\rho_1e^{L+c_0/16}}{b_*^2 L}\sum_{i=k}^{n-1}e^{-(k-1)c_0/16}.
  \end{align*}
  which completes the proof after setting
  $\tau_* = \frac{\rho_1e^{L+c_0/16}}{b_*^2 L}(1-e^{-c_0/16})^{-1}$.
\end{proof}

\begin{proposition}[Sectional bounded distortion]
\label{pr:distortion}
There exists $C_2>1$ so that, given a $cu$-disk $D$ tangent to the
centre-unstable cone field with $\kappa(D) \le C_1$, and given
$x\in D$ and $n \ge 1$ a hyperbolic time for $x$, then
\begin{align*}
  \frac{1}{C_2}
  \le
  \frac{|\det Dg_n \mid_{T_y D_0^\perp(x)}|}
  {|\det Df^n \mid_{T_x D_0^\perp(x)}|}
  \le 
  C_2
\end{align*}
in the notation of Lemma~\ref{le:contraction}, for every
$y\in D_0^\perp(x)$ such that $\dist_{D_n^\perp(x)}(g_ny,f^nx)\le \rho_1$.
        \end{proposition}

\begin{proof}
  Follow~\cite[Proposition 2.8]{ABV00} we write
  $J_i(y)=|\det DR_i\mid_{T_{g_i y} D_i^\perp(x)}|$ and so by the Chain Rule
  \begin{align*}
    \log
    \frac{|\det Dg_n \mid_{T_y D_0^\perp(x)}|}
    {|\det Df^n \mid_{T_x D_0^\perp(x)}|}
    &=
    \sum\nolimits_{i=0}^{n-1}\big( \log J_i(y)-\log J_i(x) \big).
  \end{align*}
  We recall that $R_i=\hat R_i\circ f$, where
  $\hat R_i: S=f(\Sigma_{f^ix}^{n-i})\to \Sigma_{f^{i+1}x}^{n-i-1}$ is
  the Poincar\'e first visit map; see the proof of
  Lemma~\ref{le:contraction} and Figure~\ref{fig:Poincare}. Since
  $\Sigma_{f^ix}^{n-i}$ is a restriction of a $cu$-disk with curvature
  bounded by $C_1$, and contained in the $cu$-disk
  $W_i=\phi_{(-\epsilon,\epsilon)}\Sigma_{f^ix}^{n-i}$ with
  $\kappa(W_i)\le C_1$ by Remark~\ref{rmk:flowdisk}, then $S$ has
  bounded curvature by Proposition~\ref{pr:curvature}. By
  construction, $S$ is also tangent to $\Sigma_{f^{i+1}x}^{n-i-1}$ at
  $f^{i+1}x$. Hence, we can see $S$ as a graph of a
  $(L_1,\zeta)$-H\"older continuous smooth map
  $h:\Sigma_{f^{i+1}x}^{n-i-1}\to\RR\cdot G({f^{i+1}x})$ and $\hat R_i$
  as the projection from this graph to its domain.  Moreover,
  $f:W_i\to \phi_{(-\epsilon,\epsilon)}S$ is a $C^2$ diffeomorphism
  from a flat submanifold to a manifold whose curvature is bounded by
  $C_1$. Thus
  \begin{align*}
  \log J_i(y) = \log|\det D\hat R_i\mid_{T_{fg_iy}S}| + \log |\det
  Df\mid_{T_{g_iy} D_0^\perp(x)}|
  \end{align*}
  and both summands are restrictions of $(L_1,\zeta)$-H\"older
  continuous maps.

  Therefore, the sum is bounded above by
  $ \sum_{i=0}^{n-1}2 L_1\big( e^{-i c_0/8}\rho_1\big)^\zeta \le
  \frac{2L_1\rho_1^\zeta}{1-e^{-\zeta c_0/8}}.  $ The proof is
  complete after setting
  $C_2=\exp(2L_1\rho_1^\zeta/(1-e^{-\zeta c_0/8}))$.
\end{proof}


\section{Lebesgue measure and hyperbolic times}
\label{sec:lebesgue-measure-at}

We extend the construction of backward contraction to a full
neighborhood of points in a $cu$-disk at hyperbolic times in
Subsection~\ref{sec:distort-bounds-centr}.
This provides the tools needed to construct the physical/SRB measure,
outlined in Subsection~\ref{sec:constr-physic-probab}, leading to the
proof of Theorem~\ref{mthm:discretefabv} in
Subsection~\ref{sec:finitely-mamy-physic}.

\subsection{Distortion bounds and central-unstable disks
  at hyperbolic times}
\label{sec:distort-bounds-centr}

In what follows, we set
$\dist_D(x,\partial D)= \inf_{y\in\partial D}\dist_D(x,y)$ for the
distance from a point $x\in D$ to the boundary of $D$. We assume
without loss of generality that $U$ contains a $\rho_1$-neighborhood
of $\Lambda$.

Let $z\in U^*$, $N_z ^u:=E^{cu}_X\cap G(z)^\perp$ and
$W=\exp_z\big(N_z^u\cap B(0,\rho_0)\big)$ be such that the $cu$-disk
$D=\phi_{(-\rho,\rho)}W$ for some $\rho=\rho(z)>0$ satisfies
$\ell(\phi_{(-\rho,\rho)}z)>2\rho_1$ and $\Leb_D(\Omega)>0$, where
$\Leb_D$ is the volume measure induced in the embedded disk $D$ by the
Riemannian metric of $M$. Remark~\ref{rmk:flowdisk} ensures that
$\kappa(D)\le C_1$.

We note that this disk is an union of segments of trajectories of the
flow -- we say that this is a \emph{$cu$-disk}. Moreover, since there
exists $\gamma_*>0$ such that $\|G(x)\|\le\gamma_*$ for all $x\in U$,
we necessarily have that
$2\rho_1<\ell\big(\phi_{(-\rho,\rho)}z\big)\le2\rho\gamma_*$ and so
$\rho>\rho_1\gamma_*^{-1}$.  We set
\begin{align*}
  A=A(D,\rho_1)
  =
  \{ x\in D\cap \Omega : \dist_D(x,\partial D)\ge\rho_1 \}
\end{align*}
so that $\leb_D(A)>0$, reducing $\rho_1$ if necessary. 

Let $\gamma_0>0$ be such that $\|G(x)\|\ge\gamma_0$ for all $x\in U^*$
with $d(x,\sing_\Lambda(G))>\delta_0$.

Next result states robust local sectional backward contraction and
bounded distortion, together with the consequence for the push-forward
of Lebesgue measure along $cu$-disks.

\begin{proposition}\label{pr:preballs}
  Let $x\in A$ and $n>1$ be a hyperbolic time for $x$. Then there
  exists an open neighborhood $V_n(x)$ of $x$ in $D$, a $\delta_0$-ball
  $W_n(x)$ inside $f^n(D)$ centered at $f^nx$ such that
  \begin{enumerate}
  \item $f^n\mid_{V_n(x)} : V_n(x)\to W_n(x)$ is a diffeomorphism, and
  \item there exists $s_0>0$ such that
    $\phi_{(-s_0,s_0)}x\subset V_n(x)$ and $f^n(\phi_{(-s_0,s_0)}x)$
    has length at least $\rho_1$, and for all $-s_0<s<s_0$
    \begin{enumerate}
    \item $n$ is a hyperbolic time for $\phi_sx$,
      $D_0^\perp(\phi_sx)\subset V_n(x)$, and
    \item the translated Poincar\'e maps
      $R_i^s:\dom(R_i^s)\subset\Sigma_{\phi_s f^i(x)}^{n-i}
      \to\Sigma_{\phi_sf^{i+1}(x)}^{n-i-1}, i=0,\dots,n-1$ composed
      to form $g_n^s=R_{n-1}^s\circ\cdots\circ R_0^s$ satisfy:
      \begin{enumerate}
      \item
        $\big(g_n^s\mid_{D_0^\perp(\phi_sx)}\big)^{-1}:
        D_n^\perp(\phi_sf^n) \to D_0^\perp(\phi_sx)$ is a
        $e^{-nc_0/8}$-contraction, and
      \item for every $y\in D_0^\perp(\phi_sx)$ such that
        $\dist_{D_n^\perp(\phi_sx)}(g_n^sy,\phi_sf^nx)\le \rho_1$ we
        get
        \begin{align*}
          C_2^{-1}
          \le
          \frac{|\det Dg_n^s \mid_{T_y D_0^\perp}(\phi_sx)|}
          {|\det Df^n \mid_{T_{\phi_s x} D_0^\perp}(\psi_sx)|}
          \le 
          C_2.
        \end{align*}
      \end{enumerate}
    \end{enumerate}
  \item there exists $C_3>0$ so that
    $f_*^n\big(\leb\mid_{V_n(x)}\big)\le C_3
    \cdot\big(\leb\mid_{W_n(x)}\big)$.
  \end{enumerate}
\end{proposition}

\begin{proof}
  Fixing $x\in A$ and $n$ a hyperbolic time for $x$, then $n$ is also
  a hyperbolic time for $\phi_s(x)\in D$ for $-s_0<s<s_0$ with $s_0$
  given by Proposition~\ref{pr:roughyptime}.  Moreover,
  $D_0^\perp(\phi_sx)\subset D$ and we obtain item
  (2a). 

  Hence,
  $d(\phi_sf^nx,\sing_\Lambda(G))>\delta_0$, which implies that
  $\|G({\phi_sf^nx})\|\ge\gamma_0$ for all
  $|s|<s_0$.  Thus $\ell(\phi_{(-s_0,s_0)}f^nx)\ge
  2s_0\gamma_0$.

  In addition, since $\|G(y)-G(x)\|\le L d(x,y)$, if
  $d(x,\sing_\Lambda(G))<\delta_0$, then
  $s_0b_*\|G(x)\| \le \ell(\phi_{(-s_0,s_0)}x) \le s_0 \cdot 2L
  d(x,\sing_\Lambda(G)) < \frac12 d_{\delta_0}(x,\sing_\Lambda(G))$,
  by Lemma~\ref{le:boundsigma}.
      
  To obtain item (2b), we apply Lemma~\ref{le:contraction} together
  with Proposition~\ref{pr:distortion} to each $\phi_sx$ with
  $|s|<s_0$, and we also get that
  $\dist_{D_n}(f^nx,\partial D^n)>\min\{\rho_1,2s_0\gamma_0\}=\rho_1$
  (reducing $\rho_1$ if needed).

  To obtain item (1), we consider the open set
  $W_n(x)=\cup_{|s|<\rho_x}D_n^\perp(\phi_sf^nx)$ together with
  $V_n(x)=f^{-n}W_x(x)$, and note that $W_n(x)$ contains a
  $\rho_1$-ball around $f^nx$ by construction.
\begin{figure}[h]
  \centering \includegraphics[width=9cm]{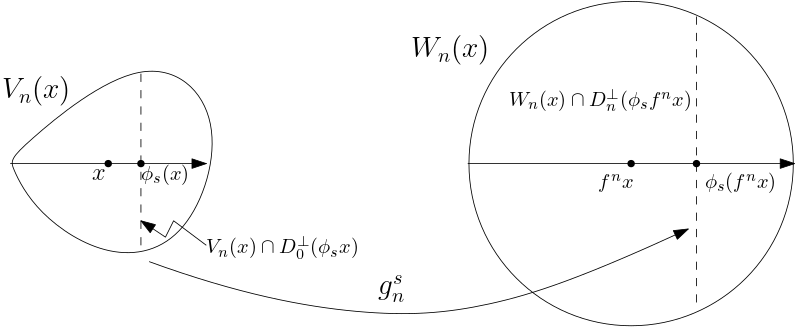}
  \caption{\label{fig:Fubini} Sectioning $V_n(x)$ and $W_n(x)$ through
    normal cross-sections to the trajectory of $x$, to then apply
    Fubini's Theorem.}
\end{figure}
  Finally, since
  $\{D_n^\perp(\phi_sf^nx)=g_n^s\big(D_0^\perp(\phi_sx)\cap
  V_n(x)\big):|s|<s_0\}$ is a measurable partition of $W_n(x)$, we use
  Lemma~\ref{le:contraction} together with
  Proposition~\ref{pr:distortion} as follows; see Figure~\ref{fig:Fubini}.
  We write
    \begin{align*}
      v(\phi_sx)
      =
      \leb^\perp\mid_{V_n(x)\cap D_0^\perp(\phi_sx)}
      \;\; \& \;\;
      w(\phi_sx)
      =
      \leb^\perp\mid_{W_n(x)\cap D_n^\perp(\psi_sx)},
      \;\;
      |s|<s_0;
    \end{align*}
    for the normalized volume measures induced on
    $D_0^\perp(\psi_sx)\cap V_n(x)$ and
    $W_n(x)\cap D_n^\perp(\psi_sx)\big)$, respectively.  Since
    $\big(\leb\mid_{V_x(x)}\big)(E)=\int_{-s_0}^{s_0} v(\phi_sx)(E)\,
    \frac{ds}{\|G({\phi_sx})\|}$ for any measurable subset $E$, then we
    can apply Fubini's Theorem.

    We have that $g_n^s(y)=\phi_{\tau(s,y)}\circ f^n$ for each
    $y\in D_0^\perp(\psi_sx)$ with $|\tau(s,y)|\le\tau_*$, so
    $f^n= h\circ g_n^s$ and $h:V_n(x)\to W_n(x)$ is a diffeomorphism
    with bounded derivatives.  Thus, item (2b)(ii) ensures that
    $(g_n^s)_*v(\phi_sx)\le C_2 w(\phi_sx)$, and there exists a
    constant $K>0$ so that
    \begin{align*}
      f_*^n\big(\leb\mid_{V_n(x)}\big)(E)
      &=
        f_*^n\left(
        \int_{-s_0}^{s_0} v(\phi_sx)(E)\, \frac{ds}{\|G({\phi_sx})\|} 
        \right)
        =
        \int_{-s_0}^{s_0} [(h\circ g_n^s)_*v(\phi_sx)](E)\,
        \frac{ds}{\|G({\phi_sx})\|}
      \\
      &\le
        C_2\int_{-s_0}^{s_0} [(h)_*w(\phi_sx)](E)\,
        \frac{ds}{\|G({\phi_sx})\|}
        \le
        C_2 K\int_{-s_0}^{s_0} w(\phi_sx)(E)\, \frac{ds}{\|G({\phi_sx})\|}
      \\
       &\le
         C_2K
         \int_{\phi_{(-s_0,s_0)}f^nx}
         \big(\leb\mid_{W_n(x)\cap D_n^\perp(z)}\big)(E) \, dz
      \\
      &=
        C_2K \big(\leb\mid_{W_n(x)}\big)(E).
    \end{align*}
    This completes the proof after setting $C_3=K C_2$.
  \end{proof}


\subsection{Construction of a physical probability measure}
\label{sec:constr-physic-probab}

We now have all the basic tools needed to follow the construction
presented in \cite[Sections 3 \& 4]{ABV00} to obtain a physical/SRB
probability measure for the flow. We present a step by step overview
in what follows.

For each $n>1$ we set
$
  H_n
  =
  \{x\in A(D,\rho_1) : n \text{\ is a hyperbolic time for\  } x\}.
$
From Proposition~\ref{pr:preballs}, if $x\in H_n$, then $f^nx$ is
$\rho_1$-away from the boundary of $f^nD$. For $\delta>0$, we denote
$\Delta_n(x,\delta)$ the $\delta$-neighborhood of $f^nx$ inside
$f^n(D)$. If $\leb_D$ is the probability measure
$\leb_D(E)=\leb(D\cap E)/\leb(D)$ for every Borel subset $E\subset D$,
obtained by normalizing the Riemannian induced volume measure on $D$,
then
$(f^n_* \Leb_D) \mid_{\Delta_n(x,\delta_1)}\le C_3
\leb_{f^n(D)}\mid_{\Delta_n(x,\delta_1)}$, again from
Proposition~\ref{pr:preballs}.

The following is a geometrical consequence of the finite
dimensionality and  bounded curvature of
$cu$-disks.

\begin{proposition}
\label{pr:maximal}
There exists $\tau>0$ so that for  $n>T_2$
there exists a finite subset $\wh{H}_n$ of $H_n$ such that the balls
$\Delta_n(z,\rho_1/4)$
in $f^n(D)$ centered at $z\in f^n(\wh{H}_n)$ are
pairwise disjoint, and their union $\Delta_n$ satisfies
$
(f_*^n \Leb_D)(\Delta_n \cap \Omega)
\ge
(f_*^n \Leb_D)(\Delta_n \cap f^n(H_n))
\ge
\tau \Leb_D(H_n).
$
\end{proposition}

\begin{proof}
  See \cite[Proposition 3.3 \& Lemma 3.4]{ABV00}.
\end{proof}

Let $\D_n=\{\Delta_n(z,\rho_1/4): z\in f^n(\wh{H}_n)\} $ be the
collection of balls that form $\Delta_n$. We note that all these balls
are $\delta_0$-away from $\sing_\Lambda(G)$, and we define
\begin{align}
  \label{eq:nun}
  \mu_n:=\frac{1}{n} \sum\nolimits_{j=0}^{n-1} f_*^j (\Leb_D);\quad
  \nu_n:=\frac{1}{n} \sum\nolimits_{j=0}^{n-1} (f_*^j \Leb_D) \mid_{\Delta_j}
  \qand
  \eta_n:=\mu_n-\nu_n.
\end{align}

\begin{proposition}
\label{pr:mass}
There is $\alpha>0$ so that both $\nu_n(\Omega)\ge \alpha$ and
$\nu_n\big(\cup_{i=0}^{n-1} f^i(D\cap \Omega)\big)\ge \alpha$ for
all sufficiently large $n>T_2$.
\end{proposition}

\begin{proof}
  Just follow \cite[Proposition 3.5]{ABV00}, using
  Proposition~\ref{pr:maximal} together with the positive asymptotic
  frequency of hyperbolic times for each $x\in\Omega$, given by
  Proposition~\ref{pr:hyptimes}, and Remark~\ref{rmk:upad}.
\end{proof}

We consider weak$^*$ accumulation points $\mu$ and $\nu$ of
$(\mu_n), (\nu_n)$ respectively, along some subsequence
$(n_k)_k$. It is standard that $\mu$ is a $f$-invariant probability
measure and that $\wt{\mu}=\int_0^1(\phi_t)_*\mu\,dt$ is a flow invariant
probability measure; see e.g. \cite{ViOl16}. In
addition, $\nu(U) \ge \limsup_k\nu_k(U)\ge \alpha>0$.

We claim that $\nu$ has a property of absolute continuity along
certain disks contained in its support. We define the collection of
these disks in what follows.

Note that $\nu_n$ is supported on the union $\cup_{j=0}^{n-1}\Delta_j$
of disks with uniform size and $\delta_0$-away from
$\sing_\Lambda(G)$. Then $\supp \nu$ is contained in
$ \Delta_\infty=\cap_{n=1}^{\infty} \close\big(\cup_{j\ge n}
\Delta_j\big), $ the family of accumulation points of such disks. That
is, for $y\in\Delta_\infty$ there are $(j_i)_i\to\infty$, disks
$\wt{\Delta}_i=\Delta_{j_i}(x_i,\delta_1/4)\subset \Delta_{j_i}$, and
points $y_i\in \wt{\Delta}_i$ so that $y_i\to y$ when
$i\nearrow\infty$.

We may assume, without loss of generality, taking subsequences if
necessary, that $x_i$ converges to some point $x$. By uniform size and
bounded curvature, we can use the Ascoli-Arzela Theorem to conclude
that $\wt{\Delta}_i$ converge to a $cu$-disk $\wt{\Delta}(x)$ with
radius $\rho_1/4$ centered at $x$. Then
$y\in\close \wt{\Delta}(x)\subset \Delta_\infty$.

\begin{lemma}\label{le:limitdisk}
  Every $y\in\wt{\Delta}(x)$ is such that $N_y^u$ is uniformly
  expanding: $\|\big(P^k\mid_{N_y^u})^{-1}\|\le e^{- k c_0/8}$ for all
  $k\ge1$. The disk $\wt{\Delta}(x)$ is contained in $\Lambda$ and also
  in the unique\footnote{The center-unstable manifold might depend on
    the radius, but it is uniquely defined given the radius.}
  center-unstable manifold $W^{cu}_x(\rho_1)$ tangent to $E^{cu}_x$
  containing a $\rho_1$-ball around $x$.
\end{lemma}

\begin{proof}
  Let $j_i\nearrow\infty$, $x_i\to x$, and
  $\wt{\Delta}_i\to \wt{\Delta}(x)$ be as in the construction
  described previously.  We have that $\wt{\Delta}_i$ is contained in
  the $j_i$th-iterate of $D$, which is a $cu$-disk.  The domination of
  the splitting on $U$ ensures that
  $\angle\big(\wt{\Delta}_i,E^{cu}\big)\to0$ as $i\nearrow\infty$,
  uniformly on $\Lambda$; this is a consequence of
  Proposition~\ref{prop:Ccu}.

  By Proposition~\ref{pr:preballs}, $f^{-k}$ is a
  $e^{-kc_0/8}$-contraction on
  $\wt{\Delta}_i\cap \exp_{\phi_s f^{j_i}x}\big(
  N_{\phi_sf^{j_i}x}^u\big)$ for every large $i$ and any given fixed
  $k\ge1$ and $|s|\le s_0$.  Passing to the limit in $i$, we get that
  $f^{-k}$ is a $e^{-kc_0/8}$-contraction on
  $\wt{\Delta}(x) \cap \exp_{\phi_sx}\big( N_{\phi_sx}^u\big) $, and
  $\wt{\Delta}(x)$ is a $cu$-disk in $\Lambda$ by continuity of the
  splitting on $U$.

  We have shown that the subspace $E^{cu}$ is uniformly sectionally
  expanding for $Df$ on $\wt{\Delta}(x)$.  Since $Df\mid_{E^s}$ is
  uniformly contracted, we are in the setting of \cite[Section 3 of
  Chapter 7]{BarPes2007} with respect to $f^{-1}$. Then there exists a
  unique \emph{center-unstable} manifold $W^{cu}_x(\rho_1)$ tangent to
  $E^{cu}_x$ containing a $\rho_1$-ball around $x$.
\end{proof}

\subsubsection{Absolute continuity}
\label{sec:absolute-continuity}

The same arguments in \cite[Section 4.1]{ABV00} imply the following
result, where we write \emph{cylinder} for any diffeomorphic image of
$\D^{d_{cu}} \times \D^{d_s}$ into $U$.

\begin{proposition}
\label{pr:cylinder}
There exist $C_3>1$ and a cylinder $\cC\subset M$, with a family
$\K_\infty$ of disjoint disks, contained in $\cC \cap \Delta_\infty$,
which are graphs over $\D^{d_{cu}}$, such that
\begin{enumerate}
\item the union $K_\infty$ of all disks in $\K_\infty$ has positive
  $\nu$-measure;
\item the restriction $\nu\mid_{K_\infty}$ has absolutely continuous
  conditional measures $\nu_\gamma$ with respect the induced volume
  $\m_\gamma$ along the disks $\gamma\in\K_\infty$, whose density is
  bounded: $C_3^{-1}\le d\nu_\gamma/d\m_\gamma \le C_3$.
\end{enumerate}
\end{proposition}

\begin{proof}
  See \cite[Proposition 4.1 \& Lemma 4.4]{ABV00}, whose proof uses the
  properties of $\Delta_\infty$ already obtained.
\end{proof}

\begin{remark}[Uniformly sized cylinder]\label{rmk:unifcylinder}
  The cylinder $\cC$ is based on a center-unstable manifold
  $W^{cu}_x(\rho_1)$ with uniform inner radius obtained in
  Lemma~\ref{le:limitdisk} and contained in the ergodic basin of
  $\mu_*$.  Since the size of the stable leaves is also uniform by
  partial hyperbolicity, then the volume of the ergodic basin is
  uniformly bounded from below away from zero, with a bound depending
  only on $\rho_1$, which depends on the map $f$.
\end{remark}

\subsubsection{Ergodicity and ergodic basin}
\label{sec:ergodic-ergodic-basi}

Following \cite[Section 4.2]{ABV00} we obtain the next result.

\begin{lemma}
\label{le:ergodic}
The $f$-invariant probability measure $\mu=\nu+\eta$ has an ergodic
component $\mu_*$ whose Lyapunov exponents are all non-zero, except
along the direction of the vector field, and whose conditional
measures along local unstable manifolds are absolutely continuous with
respect to Lebesgue measure.  Moreover, $\supp \mu_*\subset\Lambda$
and $\Leb_D(B(\mu_*)\cap H)>0$.
\end{lemma}

\begin{proof}
  This is \cite[Lemma 4.5]{ABV00}, whose proof uses the properties of
  $\Delta_\infty$ already obtained in the previous arguments.
\end{proof}

\subsection{Finitely many physical/SRB measures for the flow}
\label{sec:finitely-mamy-physic}

The following completes the proof of the existence statement of
Theorem~\ref{mthm:discretefabv}. The reciprocal is contained in the
proof of item (2) of Theorem~\ref{mthm:equivalence} in
Subsection~\ref{sec:discrete-nuse-versus}.

\begin{corollary} \label{cor:terminal} There exist finitely many
  ergodic hyperbolic physical/SRB invariant probability measures
  $\eta_1,\dots,\eta_k$ for $f$ and $\mu_1,\dots,\mu_k$ for the flow
  $\phi_t$ of $G$, supported on $\Lambda$, such that
  $\leb\big(\Omega\setminus\cup_{i=1}^k B(\mu_i)\big)=0$ and
  $\m\big(\Omega\cap ( B(\eta_i)\bigtriangleup B(\mu_i))\big)=0,
  \forall i$.
\end{corollary}

\begin{proof}
  The existence of finitely many ergodic hyperbolic physical/SRB
  measures $\eta_1,\dots,\eta_k$ with respect to $f$ supported in
  $\Lambda$ and satisfying
  $\leb\big(\Omega\setminus\cup_{i=1}^k B(\eta_i)\big)=0$ follows by
  \cite[Corollary 4.6]{ABV00} using the properties already obtained.

  
  We are left to obtain the $G$-invariant ergodic physical
  probability measures. The probability measures
  $\mu_i=\int_0^1(\phi_t)_*\eta_i\,dt$ are $\phi_t$-invariant for
  every $t>0$ and $\mu_i$ are ergodic for the flow, $i=1,\dots, k$.

  Moreover, if $\vfi:M\to\RR$ is continuous and $x\in B(\eta_i)$,
  then $\psi=\int_0^1\vfi\circ\phi_s\,ds$ is also continuous, and
  since $\phi_t$ and $f$ commute
  \begin{align*}
    \int\vfi\,d\mu_i
    &=
    \int\psi\,d\eta_i
    =
    \lim_{n\to+\infty}\frac1n\sum\nolimits_{j=0}^{n-1}\psi(f^jx)
    \\
    &=
    \lim_{n\to+\infty}\frac1n\int_0^n\vfi(\phi_sx)\,ds
    =
    \lim_{T\nearrow\infty, T\in\RR}\frac1T\int_0^T\vfi(\phi_sx)\,ds,
  \end{align*}
  where the last equality follows from boundedness of $\vfi$. This
  shows that $B(\eta_i)\subset B(\mu_i)$ and so $\mu_i$ becomes a
  physical measure and also
  $\m\big(\Omega\cap ( B(\eta_i)\bigtriangleup B(\mu_i))\big)=0$ for
  $i=1,\dots,k$.

  Hyperbolicity of $\mu_i$ follows from partial hyperbolicity together
  with the implication (NU2SE)$\implies$(MASE) from
  Theorem~\ref{thm:NUSEvar}, so that we have (MASE) for each
  $x\in B(\mu_i)$. From the results of~\cite{JiangQiPing03} the
  Lyapunov spectrum of $x\in B(\mu_i)$ and of $\mu_i$ coincide ---
  thus, $\mu_i$ has positive Lyapunov exponents along all the
  directions in $E^{cu}_z$, except $G(z)$, for
  $\mu_i$-a.e. $z\in\Lambda$.


  By smoothness of the flow, the absolute continuity of conditional
  measures of $\eta_i$ along unstable manifolds implies absolute
  continuity of conditional measures of $\mu_i$ along weak-unstable
  (or center-unstable) manifolds, so that each $\mu_i$ is also a
  $cu$-Gibbs state. That is, each $\mu_i$ is an ergodic hyperbolic
  physical/SRB measure for the flow, completing the proof.
\end{proof}


\section{Discrete versus continuous: the main theorems and
  corollaries}
\label{sec:recipr-condit}

To easily deduce Theorem~\ref{mthm:equivalence} from
Theorem~\ref{mthm:discretefabv}, we recall some general properties of
Gibbs $cu$-states first. Then we prove
Theorem~\ref{mthm:equivalence}. Afterwards, we obtain
Theorem~\ref{mthm:fABV} and
Corollaries~\ref{mcor:wABVndcu},~\ref{mcor:wABV}~\&~\ref{mcor:SRB-ASH}.

\subsection{Properties of Gibbs $cu$-states}
\label{sec:properties-gibbs-cu}

We collect some useful results here.

\begin{theorem}
  \label{thm:appv-attracting}
  Let $\Lambda=\Lambda_G(U)$ be a partially hyperbolic attracting set
  for a $C^2$ vector field $G$ which is NU2SE on $\Omega\subset U$ with
  $\m(\Omega)>0$. Then there exists $k\in\ZZ^+,\varrho>0$ and $k$
  distinct $f$-invariant ergodic probability measures
  $\mu_1,\dots,\mu_k$ satisfying
  \begin{enumerate}
  \item the family $\EE$ of all $G$-invariant physical probability
    measures $\mu$ such that $\m(\Omega\cap B(\mu))>0$ is the convex
    hull
    $ \EE=\{\sum_{i=1}^k t_i \mu_i : \sum_i t_i=1 , t_i\ge0,
    i=1,\dots,k\}.$ The same holds replacing $G$-invariance by
    $\phi_t$-invariance, for some fixed value of $t>0$.
  \item for a $G$-invariant (or $f$-invariant) \emph{hyperbolic}
    probability measure $\mu$ supported in $\Lambda$, with
    $\mu(\Omega)>0$, the following are equivalent
    \begin{enumerate}
    \item
      the Entropy Formula:
      $h_\mu(f)=\int\log|\det Df\mid_{E^{cu}}|\,d\mu$;
    \item $\mu$ is a $cu$-Gibbs state, that is, admits an absolutely
      continuous disintegration along center-unstable manifolds;
    \item $\mu$ is a physical measure, i.e., its basin
      $B(\mu)$ has positive Lebesgue measure.
    \end{enumerate}
  \item the basin $B(\mu)$ of a physical measure $\mu$ supported in
    $\Lambda$, with $\m(B(\mu)\cap\Omega)>0$, admits an open subset
    $V$, containing a $\varrho$-ball, which intersects $\Lambda$ and is
    contained in the ergodic basin except a zero volume subset, that
    is, $\m(V\setminus B(\mu))=0$ and $V\cap\Lambda\neq\emptyset$.
  \item if $\Lambda$ is transitive then $k=1$, i.e., there exists only
    one physical probability measure $\mu=\mu_1$, which is also a
    Gibbs-$cu$-state, such that $m(B(\mu)\setminus\Omega)=0$.
  \end{enumerate}
\end{theorem}

\begin{proof}
  For item (1), Theorem~\ref{mthm:discretefabv}
  (cf. Corollary~\ref{cor:terminal}) ensures the existence of finitely
  many ergodic hyperbolic physical/SRB measures $\mu_1,\dots,\mu_k$
  such that the union of their ergodic basins covers $\Omega$ Lebesgue
  almost everywhere:
    $  \m\left(\Omega\setminus\big(\cup_{i=1}^k
      B(\mu_i)\big)\right)=0.$
    We note that if there are no equilibria, then we can take
    $\Omega=U$. The measures considered can either be invariant for
    the flow, or $f$-invariant, or even $\phi_t$-invariant for any
    fixed $t>0$. 
    
    Since $\m(B(\mu)\cap\Omega)>0$, it follows that
    $B(\mu)\cap \Omega=\Omega\cap\big(\sum_{i=1}^k B(\mu)\cap B(\mu_i)\big)$
    Lebesgue modulo zero. By definition of ergodic basin, for each
    continuous observable $\vfi:U\to\RR$ we get
  \begin{align*}
    \int\vfi\,d\mu
    &=
      \frac1{\m(\Omega\cap B(\mu))}\int_{\Omega\cap B(\mu)} \int\vfi\,
      d\left(\lim_{n\to+\infty}
      \frac1n\sum\nolimits_{j=0}^{n-1}\delta_{f^jx}\right)
      \,d\leb(x)
    \\
    &=
      \sum\nolimits_{i=1}^k \frac{\m(B(\mu)\cap B(\mu_i)\cap
      \Omega)}{\m(B(\mu)\cap\Omega)}\int \vfi\,d\mu_i,
  \end{align*}
  where the limit is in the weak$^*$ topology of the probability
  measures of the ambient space $M$.  Hence, we deduce
  $\mu=\sum_{i=1}^k \frac{\leb(B(\mu)\cap B(\mu_i)\cap
    \Omega)}{\leb(B(\mu)\cap \Omega)}\mu_i$ and $\mu$ as a convex linear
  combination of the ergodic physical/SRB measures provided by
  Theorem~\ref{mthm:discretefabv}.
  
  For item (2), since $G$ is contained is $E^{cu}$ and has zero
  Lyapunov exponent, then domination of the splitting
  $E^s\oplus E^{cu}$ ensures that all Lyapunov exponents along $E^s$
  are strictly negative and so
  $\int\log|\det Df\mid_{E^{cu}}|\,d\mu=\int \Sigma^+\,d\mu$ by Oseledets'
  Multiplicative Ergodic Theorem. This holds either for $G$-invariant
  of $f$-invariant probability measures, or even $\phi_t$-invariant
  for a fixed value of $t>0$.

  Then, assumption (2a) means $h_\mu(f)=\int\Sigma^+\,d\mu>0$. In
  particular, $\mu$ is non-atomic\footnote{For otherwise by Ergodic
    Decomposition, Jacob's Theorem \cite[Chpt. 9, Sec. 6]{ViOl16} and
    Margulis-Ruelle Inequality \cite[Theorem 9.4.4]{ViOl16} we would
    obtain the Entropy Formula for each ergodic component $\nu$ of
    $\mu$. In particular, if $\nu$ is supported on a critical element
    of $\Lambda$, either an equilibrium or a periodic orbit, then
    $0=h_\nu(f)=\int\log|\det f\mid_{E^{cu}}|\,d\nu=\int\Sigma^+\,d\nu$,
    contradicting the hyperbolicity assumption on $\mu$.}  and this
  becomes the necessary and sufficient condition for absolutely
  continuous disintegration along unstable manifolds $W^{uu}_x$ for
  $\mu$-a.e. $x$, by the characterization of measures satisfying the
  Entropy Formula \cite{LY85} for $C^2$ smooth dynamics. This means,
  more precisely, that for $\mu$-a.e. $x\in\Lambda$ there exists
  $\rho=\rho(x)>0$ so that
  \begin{align*}
    \Pi_x=\{W^{uu}_y: y\in B(x,\rho)\, \& \,  W^{uu}_y \text{ crosses
    } B(x,\rho) \}
  \end{align*}
  and\footnote{We say that $W^{uu}_y$ \emph{crosses} $B(x,\rho)$ if
    the connected component $\gamma_y$ of $W^{uu}_y\cap B(x,\rho)$
    containing $y$ projects into the corresponding connected component
    $\gamma_x$ of $W^{uu}_y\cap B(x,\rho)$ containing $x$, thought
    the stable holonomy map $\pi^s$ in a one-to-one way, i.e.,
    $\pi^s\mid_{\gamma_y}: \gamma_y\to\gamma_x$ is injective.} the
  normalized restriction $\tilde\mu=\mu\mid_{\cup\Pi_x}$ disintegrates
  along the leaves of $\Pi_x$ as $\tilde\mu=\int
  \mu_y\,d\hat\mu(y)$. Here $\mu_y$ is a probability measure supported
  on $\gamma_y$ equivalent to the restriction $\m_y$ of Lebesgue
  measure on this submanifold and $\hat\mu=\pi_*\tilde\mu$, where
  $\pi:\cup\Pi_x\to\Pi_x$ is the quotient map associating a point of
  $\cup\Pi_x$ to the corresponding leaf of $\Pi_x$.
 
  Each manifold $W^{uu}_x$ is contained in $\Lambda$ with dimension
  $\dim E^{cu}-1$ and tangent bundle in $E^{cu}$. The center-unstable
  (or weak-unstable) manifolds $W^{cu}_x=\phi_{(-1,1)}(W^u_x)$ are
  tangent to the center-unstable bundle $E^{cu}$ at each point and
  also the disjoint union of strong-unstable leaves transported by the
  flow. By smoothness of the flow, the disintegration of $\mu$ along
  the center-unstable leaves is also absolutely continuous. 

  Indeed, for small enough $\rho$, $W^{uu}_y\in\Pi_x$ if, and only if,
  $W^{cu}_y$ crosses $B(x,\rho)$. Considering
  $\Pi_x^c=\{ W^{cu}_y : y\in B(x,\rho) \,\&\, W^{cu}_y$ crosses
  $B(x,\rho)\}$, then $\cup\Pi_x=\cup\Pi_x^c$ and
  $\tilde\mu=\int \nu_z\,d\hat\nu(z)$, with
  $\hat\nu=\hat\pi_*\tilde\mu$ where $\hat\pi:\cup\Pi_x^c\to\Pi_x^c$
  is the corresponding quotient map, and
  $\nu_z=\int \mu_y \, d(\pi_*\nu_z)$ is equivalent to $\m$ induced on
  the connected component $\gamma_z^{cu}$ of $W^{cu}_z\cap B(x,\rho)$
  containing $z$. This is the property stated in item (2b).
  
  Assuming condition (2b), the Ergodic Theorem provides a full
  $\mu$-measure subset $B$ of Birkhoff generic points for $\mu$ which
  is also a full $\tilde\mu$-measure subset. Hence, $B$ has full
  $\nu_z$-measure for $\hat\nu$-a.e. $z$.  If we fix a center unstable
  disk $\gamma_z^{cu}$ for a $\hat\nu$-generic $z$, then $\nu_z(B)=1$
  and $B$ is also a full $\m_z$-measure subset of
  $\gamma_z^{cu}$. Since the stable foliation is defined at all points
  of $\Lambda$, tangent to the stable bundle $E^s$ which makes an
  angle with the center-unstable bundle uniformly bounded away from
  zero, then the subset
  $V=\{ W^s_w(\epsilon): w\in\gamma_z^{cu}\}$ is an open neighborhood
  of $z$ for small enough $\epsilon>0$, where $W^s_w(\epsilon)$ is the
  connected component of $W^s_w\cap B(w,\epsilon)$ containing
  $w$. Moreover, the stable foliation is absolutely
  continuous~\cite[Section 6]{ArMel18}, and so the subset
  $W=\{ W^s_w(\epsilon): w\in B\cap\gamma_z^{cu}\}$ has full
  $\m$-measure in $V$: 
  $\m(W\setminus V)=0$. In addition, each
  $y\in W$ is such that $d(\phi_ty,\phi_tw)\to0$ when
  $t\nearrow\infty$ for some $w\in B\cap\gamma_z^{cu}$. Hence, for any
  given continuous observable $\vfi:U\to\RR$ we obtain
  \begin{align}\label{eq:contcontraction}
    \lim_{T\nearrow\infty}\frac1T\int_0^T\vfi(\phi_ty)\,dt
    =
    \lim_{T\nearrow\infty}\frac1T\int_0^T\vfi(\phi_tw)\,dt
    =
    \int\vfi\,d\mu
  \end{align}
  and thus $W\subset B(\mu)$ with $\m(W)>0$ and $\mu$ becomes a
  physical measure, as stated in item (2c). For an $f$-invariant
  measure, we replace~\eqref{eq:contcontraction} by
  $\lim n^{-1}\sum_{i=0}^{n-1}\vfi(f^ix)$ an argue in the same way.

  We note that we immediately obtain item (3) from the previous
  construction, since $\m(V\setminus
  W)=0$, once we show that (2c) implies (2a). Indeed, the uniformed
  sized $\varrho$-ball inside $V$ is given by the cylinder
  $\cC$ obtained in the proof of Theorem~\ref{mthm:discretefabv}; see
  Remark~\ref{rmk:unifcylinder}.

  Moreover, from item (3) we easily obtain item (4). Indeed, if there
  are two physical measures $\mu_1,\mu_2$, then from item (3) there
  exist open subsets $V_i$ such that $\m(V_i\setminus B(\mu_i))=0$ and
  $V_i\cap\Lambda\neq\emptyset, i=1,2$. Transitivity ensures that there
  exist $x_1\in V_1$ and $t>0$ so that $\phi_tx_1\in
  V_2$. Smoothness of the flow ensures that $\phi_{-t}B(\mu_2)$ has
  positive volume in $V_1$, thus by flow invariance of the ergodic
  basin we find $y\in B(\mu_1)\cap B(\mu_2)$, which implies that
  $\mu_1=\mu_2$.

  We are left with showing that condition (2c) implies (2a). But this
  is an easy consequence of item (1), since a physical probability
  measure $\mu$ is a linear convex combination of the finitely many
  ergodic physical/SRB measures provided by
  Theorem~\ref{mthm:discretefabv} which are $cu$-Gibbs states, that
  is, satisfy (2a). The proof is complete.
\end{proof}

\subsection{Proof of the main theorems}
\label{sec:proof-corollaries}

We are now ready to deduce Theorem~\ref{mthm:equivalence} first,
assuming the statement of Theorem~\ref{mthm:discretefabv}. Afterwards,
we prove Theorem~\ref{mthm:fABV} and then
Corollaries~\ref{mcor:wABVndcu},~\ref{mcor:wABV}~and~\ref{mcor:SRB-ASH}.


\subsubsection{Discrete (NU2SE) versus continuous (MASE)}
\label{sec:discrete-nuse-versus}

\begin{proof}[Proof of Theorem~\ref{mthm:equivalence}]
  For item (1), to obtain slow recurrence (SR) from continuous slow
  recurrence (CSR), we note that, since $G$ is $L$-Lipschitz, where
  $L=\sup_{x\in U}\|DG(x)\|$, we have for $\vfi(x)=d(x,\sing(G))$
\begin{align}\label{eq:ODE}
  \left|\frac{d}{dt}\vfi(t)\right|
  \le
  \|G(\phi_tx)\|
  =
  \|G(\phi_tx)-G(\sigma)\|
  \le
  L\cdot d(x,\sigma)
  =
  L\cdot \vfi(t),
\end{align}
whenever $x$ is near $\sigma\in\sing(G)$. Hence,
$e^{-Ls}\le \vfi(\phi_sx)/\vfi(x) \le e^{Ls}$ for $|s|$ small enough
so that $\vfi(\phi_sx)=d(\phi_sx,\sigma)$.
Therefore, setting
$\vfi_\delta(x)=d_\delta(x,\sing_\Lambda(G))$, given $\delta>0$
we can take $s>0$ so that $Ls<-\log\delta^{1/2}$ and if
$d(x,\sigma)<\delta$, then for $0\le t\le s$
\begin{align}\label{eq:perturb}
  -\log \vfi_\delta(\phi_tx)
  =
  -\log \vfi(\phi_tx)
  \ge
  -\log \vfi(x)-Lt
  \ge
  -(1/2)\log\vfi_\delta(x).
\end{align}
Thus, from~\eqref{eq:SSR}: for any $\epsilon>0$ we can find $\delta$
and $k\ge2, k\in\ZZ^+$ so that $L/k<-\log\delta$ and for all
$x\in\Omega$ there exists $N=N(x)>1$ so that for each $n\in\ZZ^+$
satisfying $n\ge N$ we have
\begin{align*}
  \epsilon n
  \ge
  \int_0^{n} \vfi_\delta(\phi_sx)\,ds
  &=
    \sum\nolimits_{i=0}^{nk-1}
    \int_0^{1/k} \vfi_\delta(\phi_{s+i/k}(x))\,ds
  \ge
    \frac12\sum\nolimits_{i=0}^{nk-1}
    \vfi_\delta(\phi^i_{1/k}(x)).
\end{align*}
Setting $g:=\phi_{1/k}$, this ensures that for $m\in\ZZ^+$, if
$n=[m/k]+1$, then
  \begin{align*}
    \sum\nolimits_{i=0}^{m-1}-\log\vfi_\delta(g^ix)
    \le
    \sum\nolimits_{i=0}^{nk-1}-\log\vfi_\delta(g^ix)
    \le
    2
    \int_0^{n}-\log\vfi_\delta(\phi_sx)\,ds
    <2n\epsilon
  \end{align*}
  if $x\in\Omega$ and $m>k\cdot N(x)$. Thus we obtain the next time
  reparametrization of~\eqref{eq:SR}:
\begin{align}\label{eq:slowapproxdiscrete}
  (1/m)\sum\nolimits_{i=0}^{m-1}-\log\vfi_\delta(g^ix)
  <
  2\big(1/k+1/m\big)\epsilon.
  \end{align}

  Noting that from~\eqref{eq:ODE} we may likewise deduce the reverse
  inequality to~\eqref{eq:perturb}, then a similar argument to the
  previous one enables us to reciprocally obtain continuous slow
  recurrence~\eqref{eq:sNUE} from slow recurrence~\eqref{eq:SR}. This
  completes the proof of the  statement of item (1).

  For item (2): Theorem~\ref{mthm:discretefabv} shows that $(B)$
  implies $(A)$. For the reciprocal, let $\mu$ be a hyperbolic
  physical/SRB measure for the flow, which is also a $cu$-Gibbs state
  supported on the partial hyperbolic attracting set
  $\Lambda=\Lambda_G(U)$.

  To obtain the (NU2SE) condition~\eqref{eq:NU2SE}, we note that by
  hyperbolicity of $\mu$, all central-unstable directions transversal
  to the vector field have positive Lyapunov exponents, thus
  $\chi(x):=\lim_{t\to+\infty}
  \log\big\|\wedge^2(D\phi_t\mid_{E^{cu}_x})^{-1}\big\|^{1/t} \le
  -c_0<0$ for $\mu$-a.e. $x$ and some constant $c_0>0$, by Oseledets'
  Theorem. By Fatou's Lemma for bounded sequences of functions
  we get
  $ -c_0 > \int\chi\,d\mu \ge
  \limsup_{T\nearrow\infty}
  \int\log\|\wedge^2(D\phi_T\mid_{E^{cu}_x})^{-1}\|^{1/T}\,d\mu(x).$ 
  Thus, we find $T>0$ so that
  $\int \log\|\wedge^2(D\phi_t\mid_{E^{cu}})^{-1}\|^{1/t} \,d\mu\le
  -c_0/2$ for all $t\ge T$. Hence, there exists an ergodic component
  $\nu$ of $\mu$ so that
  $\int \log\|\wedge^2(D\phi_t\mid_{E^{cu}})^{-1}\|^{1/t} \,d\nu\le
  -c_0/2$.

  Moreover, $\nu$ is also a physical/SRB measure for the flow; see
  Theorem~\ref{thm:appv-attracting}. Since $\nu$ is flow invariant, we
  can assume without loss of generality that $\nu$ is ergodic with
  respect to $g=\phi_T$, using the flow invariance and ergodicity
  of $\nu$.

\begin{remark}[Ergodicity for continuous time ensures dicrete time
  ergodicity]
  \label{rmk:contergodic}
  The ergodicity of a flow invariant probability measure $\mu$ implies
  that, for a co-countable set of times $t_*\in\RR$, we have that
  $\mu$ is $\phi_{t_*}$-ergodic; see e.g.~\cite{PuSh71}. That is, if a
  measurable set $A$ is $\phi_{t_*}$-invariant $\phi_{-t_*}(A)=A$ for
  this fixed value of $t_*$, then $\mu(A)\cdot\mu(M\setminus A)=0$.
  
  This property is not true for transformations, i.e., if $\mu$ is
  $g$-ergodic, then not necessarily $\mu$ is $g^k$-ergodic for some
  $k>1$. 
\end{remark}

Hence, because
$U\ni x\mapsto \Gamma(x):=\log\|(Dg\mid_{E^{cu}_x})^{-1}\big\|$
is continuous and $\nu$ is also physical/SRB with respect to $g$, we
get 
$
    \lim_{n\to\infty}\frac1n\sum\nolimits_{i=0}^{n-1}
    \Gamma(g^ix)
    =
    \int \Gamma\,d\nu
    \le
    -c_0/2,
$
  for all $x\in B_g(\mu)$,
  where $B_g(\mu)$ is the ergodic basin of $\mu$ as a $g$-invariant
  probability measure.

  From Theorem~\ref{thm:appv-attracting} this shows
  that~\eqref{eq:NU2SE} holds on the positive Lebesgue measure subset
  $B_g(\nu)$. We also have $B_g(\nu)\subset B(\nu)$ and so
  $\m(B_g(\nu)\cap\Omega)>0$, completing the proof that $(A)$ implies
  $(B)$ with $E=B_g(\nu)$. This finishes the proof of item (2).

  Item (3b) is a straightforward consequence of item (4) of
  Theorem~\ref{thm:appv-attracting}.

  For item (3a), we assume from now on that we are in the setting of
  the statement of Theorem~\ref{mthm:discretefabv} and its conclusion,
  i.e., we have both properties $(A)$ and $(B)$ and take $\mu$ an
  ergodic hyperbolic physical/SRB measure for the flow which is also
  $f$-ergodic (reparametrizing the flow if needed) and
  satisfying $\m(E\cap B(\mu))>0$; recall
  Remark~\ref{rmk:contergodic}.
  Since, from Theorem~\ref{thm:NUSEvar}, all trajectories in $U$ 
  satisfying (NU2SE) automatically satisfy (MASE), then we obtain a
  full volume measure subset of $E$ satisfying (MASE).

Since $E=\cup_{i=1}^k(E\cap B(\mu_i))$ except perhaps a
subset of zero Lebesgue measure, we can repeat this argument for each
$\mu_i$ to cover $\m$-a.e. point of $E$. This completes the proof
of item (3a) and of Theorem~\ref{mthm:equivalence}.
\end{proof}


\subsubsection{Weak asymptotic sectional expansion: proof of the main
  corollaries}
\label{sec:weak-asympt-sectio}
  
\begin{proof}[Proof of Theorem~\ref{mthm:fABV}]
We recall that $(wASE)\implies(wNU2SE)$ from
Theorem~\ref{thm:NUSEvar} if $\dim E^{cu}=d_{cu}=2$. Hence we have
(wNU2SE) on a positive volume subset $\Omega$ of $U$.

To obtain the physical/SRB measure, we set
$\Gamma(x):=\log\|\wedge^2(Df\mid_{E^{cu}_x})^{-1}\|$ 
and $J^{cu}(x):=\log|\det
(Df\mid_{E^{cu}_x})|$ in what follows. We note that for
$x\in\Omega$ we have a subsequence of the integers
$n_k$ so that
$\lim\frac1{n_k}\sum_{i=0}^{n_k-1}\Gamma(f^i(x))<0$ and the sequence
$\big(\frac1{n_k}\sum_{i=0}^{n_k-1}\delta_{f^ix}\big)_{k\ge1}$
weak$^*$ accumulates on a $f$-invariant probability measure
$\mu$ which satisfies~\eqref{eq:gpesin} (after
Theorem~\ref{thm:GenPesin}) and, for some subsequence $m_i$ of $n_k$
\begin{align}
  \label{eq:weakstar}
  \mu(\Gamma)
  =
  \lim_{i\to+\infty}\frac1{m_i}\sum\nolimits_{j=0}^{m_i-1}\Gamma(f^j(x))
  =
  \lim_{k\to+\infty}\frac1{n_k}\sum\nolimits_{i=0}^{n_k-1}\Gamma(f^i(x))<0.
\end{align}
Since $\mu$ is not necessarily ergodic, we consider its ergodic
decomposition to write
\begin{align}
  0
  &\le
  \int \big(h_{\mu_y}(f)-\mu_y(J^{cu})\big)\,d\mu(y)\nonumber
  \\
  &=\label{eq:disint}
  \int_{y\in\Upsilon}\big(h_{\mu_y}(f)-\mu_y(J^{cu})\big)\,d\mu(y)
  +
  \int_{y\in\supp\mu\setminus\Upsilon}\big(h_{\mu_y}(f)-\mu_y(J^{cu})\big)\,d\mu(y)
\end{align}
where $\Upsilon:=\{ y : \mu_y(\Gamma)<0\}$;
$\mu(\Gamma)<0$ by~\eqref{eq:weakstar}; and
$\mu_y:=\lim_{n\to\infty}\frac1n\sum_{i=0}^{n-1}\delta_{f^i(y)}$ (in
the weak$^*$ topology) is a
$f$-invariant ergodic probability measure so that
$y\in\supp\mu_y$ for $\mu$-a.e
$y$; see e.g.~\cite[Chapter II.6]{Man87}.


\begin{claim}\label{cl:SRB}
  There exists $y\in\Upsilon$ so that $h_{\mu_y}(f)=\mu_y(J^{cu})$.
\end{claim}

Assuming the claim, since for $y\in\Upsilon$ all the Lyapunov exponents
along $E^{cu}$ are either positive or zero along the flow direction,
and negative along $E^s$ by partial hyperbolicity, then we have that
$\nu:=\mu_y$ is a hyperbolic ergodic measure; and $\nu(J^{cu})$
coincides with the sum $\Sigma^+$ of all the positive Lyapunov
exponents of $\nu$. Hence we arrive at Pesin's Formula
$h_{\nu}(f)=\Sigma^+$. From the characterization of measures
satisfying this formula~\cite{LY85} we obtain that $\nu$ is $SRB$; and
ergodic hyperbolic SRB measures for $C^2$ diffeomorphisms are
physical; see e.g.~\cite{BDV2004}. We have obtained an ergodic
physical/SRB measure for the time-$1$ map $f$ of the flow of $G$.  The
same argument of the proof of Corollary~\ref{cor:terminal} ensures
that $\eta:=\int_0^1(\phi_t)_*\nu\,dt$ is flow invariant, ergodic and
also a hyperbolic $cu$-Gibbs state, and a physical/SRB measure for the
flow of $G$.

\begin{remark}\label{rmk:nhbhdbasin}
  There exists a neighborhood $W$ of $x$ where $\leb$-a.e. points are
  in $B(\eta)$. Indeed, by construction, we have
  $B(\eta)\subset\Omega$. Then, on $\Omega_0:=B(\eta)$ we have (NU2SE)
  and from the proof of item (3) of Theorem~\ref{thm:appv-attracting},
  we obtain the existence of an open subset $V$ such that
  $V\cap\Lambda\neq\emptyset$ and $\leb(V\setminus B(\eta))=0$. Hence
  $\nu(V)>0$ and also $\mu(V)>0$. We can take a non-negative
  continuous bump function $\vfi:M\to[0,1]$ supported on $V$ so that
  $ 0<\mu(\vfi)= \lim\frac1{n_k}\sum_{i=0}^{n_k-1}\vfi(f^i(x))$,
  ensuring that $f^k(x)\in V$ for some $k>1$. Hence, we may take
  $W=f^{-k}(V)$.

  If $\Lambda$ is an attractor, we conclude that $\eta$ is the unique
  physical/SRB measure in the same way as in the proof of item (4) of
  Theorem~\ref{thm:appv-attracting}. This proves the existence and
  uniqueness statements of Theorem~\ref{mthm:fABV}.
\end{remark}

It remains to prove the claim.

\begin{proof}[Proof of Claim~\ref{cl:SRB}]
  Arguing by contradiction, since $\mu(\Upsilon)>0$ and for $x\in\Upsilon$
  we have
  $h_{\mu_x}(f)-\mu_x(J^{cu})=h_{\mu_x}(f)-\mu_x(\Sigma^+)\le0$ by
  Ruelle's Inequality, we assume that $h_{\mu_x}(f)-\mu_x(J^{cu})<0$
  for all $x\in\Upsilon$. Then the left hand side of~\eqref{eq:disint}
  is negative, thus $\mu(\supp(\mu)\setminus\Upsilon)>0$ and there must
  be $z\in\supp(\mu)\setminus\Upsilon$ so that
  $h_{\mu_z}(f)-\mu_z(J^{cu})>0$.

  Since
  $\mu_z(\Sigma^+) \ge h_{\mu_z}(f) \ge \mu_z(J^{cu})=
  \mu_z\big(\sum_{i\le d_{cu}}\chi^+_i\big)$, then such $z$ admits a
  negative sectional Lyapunov exponent on $E^{cu}$.  \emph{Because we
    are assuming that $d_{cu}=\dim(E^{cu})=2$}, we thus obtain an
  ergodic measure $\nu=\int_0^1 (\phi_t)_* \mu_z \, dt$ for the flow
  of $G$, whose Lyapunov exponents are all negative with the exception
  of the flow direction. Hence, this measure is supported on a
  hyperbolic periodic attracting orbit $\cO$ --- a sink; see
  e.g.~\cite{Ara2020}. In particular, $\mu_z$ is an atom of $\mu$, so
  $\mu(B_\epsilon(\cO))>0$ for all $\epsilon>0$. Since $\mu$ is a
  finite measure, we may find $\epsilon_0>0$ arbitrarily small such
  that the $\epsilon_0$-neighborhood of the periodic orbit has null
  boundary: $\mu(\partial B_{\epsilon_0}(\cO))=0$. By Portmanteau
  Theorem on weak$^*$ convergence, this ensures that
  $ 0<\mu(B_{\epsilon_0}(\cO)) =
  \lim\frac1{n_k}\sum_{i=0}^{n_k-1}1_{B_{\epsilon_0}(\cO)}(f^i(x)).$
  and so $f^k(x)\in B_{\epsilon_0}(\cO)$ for some $k>1$. But this
  implies that $x$ belongs to the basin of a sink, contradicting
  $\mu(\Gamma)<0$. This contradiction proves the claim.
\end{proof}

\begin{remark}[Other setting where the claim holds]\label{rmk:claim+}
  The same proof of the claim follows, with arbitrary $d_{cu}\ge2$, if
  we know that no negative Lyapunov exponents exist along $E^{cu}$ on
  the attracting set. For instance, under condition~\eqref{eq:NNE} of
  Corollary~\ref{mcor:wABVndcu}, or for sectional hyperbolic
  attractors as in Example~\ref{ex:multidimLorenz}; see also
  Example~\ref{ex:modifLorenzMultidim}.
\end{remark}

For the reciprocal statement in Theorem~\ref{mthm:fABV}, let $\mu$ be
an hyperbolic ergodic physical/SRB measure for the partial hyperbolic
attracting set $\Lambda=\Lambda_G(U)$ with $\m(B(\mu)\cap\Omega)>0$,
and let us deduce mostly asymptotic sectional expansion (MASE), without any
restriction on $\dim E^{cu}$.

We start by noting that $G$-invariance, ergodicity and hyperbolicity
of $\mu$, together with Kingman's Subadditive Ergodic Theorem, ensure
that there exists $c_0>0$ so that
\begin{align*}
  \inf_{t>0}\int\log\|\wedge^2(D\phi_t\mid_{E^{cu}})^{-1}\|^{1/t} \,d\mu
  =
  \lim_{t\nearrow\infty}\log\|\wedge^2(D\phi_t\mid_{E^{cu}_x})^{-1}\|^{1/t}<-c_0,
  \quad \mu-\text{a.e. $x$}.
\end{align*}

For a given small $\epsilon>0$ let us fix $g=\phi_T$ with
$\int\log\|\wedge^2(Dg\mid_{E^{cu}})^{-1}\|^{1/T}
\,d\mu\le-c_0+\epsilon$ for some $T>0$ so that $\mu$ is $g$-ergodic,
from Remark~\ref{rmk:contergodic}. Note that $g$ is a partially
hyperbolic diffeomorphism with respect to the same splitting
$E^s\oplus E^{cu}$ over $\Lambda$.  Since $\mu$ is a hyperbolic
$cu$-Gibbs state even for the dynamics of $g$, then $\mu$ is a
physical measure for $g$ also; see the proof that (2b) implies (2c) in
Theorem~\ref{thm:appv-attracting}.
 
Using that $\mu$ is $g$-ergodic and physical, together with the
subadditive property of the continuous function
$(x,t)\mapsto\log\|\wedge^2(D\phi_t\mid_{E^{cu}_x})^{-1}\|$, we obtain
for $\m$-a.e. $x\in B_g(\mu)$, since $g^nx=\phi_{nT}(x)$ for $n\ge0$
\begin{align*}
  -c_0+\epsilon
  &\ge
  \int \frac{\log\|\wedge^2 (Dg\mid_{E^{cu}})^{-1}\|}{T} \,d\mu
    =
      \lim_{n\to\infty}\frac1n\sum\nolimits_{i=0}^{n-1}
      \log\|\wedge^2(D\phi_T\mid_{E^{cu}_{g^ix}})^{-1}\|^{1/T}
  \\
  &\ge
    \limsup\nolimits_{n\to\infty}
    \log\|\wedge^2(D\phi_{nT}\mid_{E^{cu}_x})^{-1}\|^{1/nT}
  \ge
    \liminf\nolimits_{t\to\infty}
    \log\|\wedge^2 (D\phi_t\mid_{E^{cu}_x})^{-1}\|^{1/t}.
\end{align*}
Since $\epsilon>0$ was arbitrary, we conclude that we have (wMASE) as
in~\eqref{eq:wMASE1} on the positive volume subset
$B_g(\mu)\subset B(\mu)$, which satisfies $\m(B_g(\mu)\cap\Omega)>0$,
completing the proof of Theorem~\ref{mthm:fABV}, since
(wMASE)$\implies$(wASE) on a full Lebesgue measure subset, from
Theorem~\ref{thm:NUSEvar}.
\end{proof}

\begin{proof}[Proof of Corollary~\ref{mcor:wABVndcu}]
  We start by assuming (wNU2SE) which, in the proof of
  Theorem~\ref{mthm:fABV} was deduced as the first step. Thus, from
  the proof of Theorem~\ref{mthm:fABV}, we obtain a probability measure
  $\mu$ supported in $\Lambda$ so that $h_{\mu}(f)=\mu(J^{cu})$. Since
  we additionally assume~\eqref{eq:NNE}, which ensures the
  non-existence of negative Lyapunov exponents along the
  central-unstable direction on $\Lambda$, we can use
  Claim~\ref{cl:SRB} by Remark~\ref{rmk:claim+}. This enables us to
  obtain a physical/SRB measure $\mu$ supported in $\Lambda$. This
  proves the first statement of Corollary~\ref{mcor:wABVndcu}.

  Moreover, such measure $\mu$ is such that its basin $B(\mu)$ is
  contained in $\Omega$, from Remark~\ref{rmk:nhbhdbasin}.

  If we additionally assume that points in $\Omega$ satisfy the slow
  recurrence condition (SR), then all points of $B(\mu)$ satisfy both
  (SR) and (NU2SE) conditions, the latter one with a rate $c_0$.
  Reapplying the construction of $cu$-disks using hyperbolic times
  from the proof of Theorem~\ref{mthm:discretefabv}, we get that
  ergodic basin $B(\mu)$ of $\mu$ has volume bounded from below away
  from zero by a uniform constant, depending only on the map $f$; see
  Remark~\ref{rmk:unifcylinder}.

  If $\m(\Omega\setminus B(\mu))=0$, then the proof is
  complete. Otherwise, we set $\mu_1=\mu$ and reapply the proof of
  Theorem~\ref{mthm:fABV} starting with $x\in\Omega\setminus B(\mu)$,
  to obtain a new ergodic physical/SRB measure $\mu_2$ with
  $B(\mu_2)\subset\Omega$. Proceeding by induction, we obtain a family
  of distinct ergodic physical/SRB probability measures
  $\mu_1,\mu_2,\ldots$ whose basins are contained in $\Omega$ with
  positive volume bounded away from zero. Since $\Omega\subset U$ with
  finite volume ($U$ is a relatively compact neighborhood of $\Lambda$
  in a finite dimensional manifold), after finitely many steps we get
  finitely distinct ergodic physical/SRB measures whose basis cover
  $\Omega$ except perhaps a zero volume subset. This completes the
  proof of the existence part of Corollary~\ref{mcor:wABVndcu}. The
  reciprocal follows just as in the proof of Theorem~\ref{mthm:fABV}.
\end{proof}

\begin{proof}[Proof of Corollary~\ref{mcor:wABV}]
  As in the proof of Theorem~\ref{mthm:fABV}, we use
  (wASE)$\implies$(wNU2SE) with rate $c_0$ from
  Theorem~\ref{thm:NUSEvar}; and the ergodic physical/SRB measure
  $\mu$ obtained is such that its basin $B(\mu)$ is contained in
  $\Omega$, from Remark~\ref{rmk:nhbhdbasin}.

  The rest of the proof of finiteness follows just like in the proof
  of Corollary~\ref{mcor:wABVndcu}, since we are assuming also that
  (SR) condition holds. 
\end{proof}

\begin{proof}[Proof of Corollary~\ref{mcor:SRB-ASH}]
  We show that a partially hyperbolic attracting set satisfying
  (wASE), when $\sing_\Lambda(G)$ contains only saddle-type hyperbolic
  equilibria, is in the setting of Theorem~\ref{mthm:fABV}.

  We start by observing that if $\Lambda$ is transitive, then
  $\sing_\Lambda(G)$ cannot contain sinks or sources.  We are left to
  find a positive Lebesgue measure subset $\Omega\subset U$ whose
  trajectories are weak asymptotically sectional expanding (wASE).

  In order to find $\Omega$, we note that, from
  Theorem~\ref{thm:NUSEvar} (see e.g. the proof of
  Theorem~\ref{mthm:fABV}), any point of
  $\Lambda':=\Lambda\setminus \cup\{W^s_\sigma:x\in\sing_\Lambda(G)\}$
  satisfies (NU2SE).

  Therefore, just like in the proof of Theorem~\ref{mthm:fABV}, for
  $x\in\Lambda'$ any $f$-invariant probability measure $\mu$ obtained
  as a weak$^*$ accumulation point of
  $\big(\frac1{n_k}\sum_{i=0}^{n_k-1}\delta_{f^ix}\big)_{k\ge1}$
  satisfies $\mu(\Gamma)\le-c_u<0$, where
  $\Gamma(x)=\log\big\|\wedge^2(Df\mid_{E^{cu}_x})^{-1}\big\|$.

  Thus, there exists an ergodic component $\nu$ of $\mu$ such that
  $\nu(\Gamma)\le -c_u$. This ensures that
  $\eta=\int_0^1(\phi_t)_*\nu\,dt$ is a hyperbolic measure for the
  flow. Hence, there exists a Pesin unstable-manifold
  $W^u_z(\epsilon)$ contained in $\Lambda$ (since $\Lambda$ is an
  attracting subset) and a corresponding center-unstable manifold
  $W^{cu}_z(\epsilon)=\cup_{|t|<\epsilon}\phi_t\big(W^u_z(\epsilon)\big)$,
  which is an embedded $cu$-disk inside $\Lambda$.
  
  Moreover, any compact connected part of the stable manifold
  $W^s_\sigma$ of each saddle-type hyperbolic equilibrium
  $\sigma\in\sing_\Lambda(G)$ is a compact submanifold transversal to
  $\Delta:=W^{cu}_z(\epsilon)$.  Therefore, since $\sing_\Lambda(G)$
  is a finite set, the intersection
  $\Delta\cap\big( \cup\{W^s_\sigma:\sigma\in\sing_\Lambda(G)\} \big)$
  has zero measure with respect to the induced volume measure
  $\leb_\Delta$ on $\Delta$.

  Thus, $\Delta':=\Delta\cap\Lambda'$ has $\leb_\Delta$-full
  measure. Since each $w\in\Delta'$ has a local stable manifold, we
  have that $y\in W^s_w$ also satisfies (wNU2SE)
  since $d(\phi_t y,\phi_tx)\to0$ as $t\nearrow\infty$.

  Therefore, we are left to check that the subset
  $\Omega:=\cup\{W^s_w: w\in\Delta'\}$ has positive volume to complete
  the conditions of Theorem~\ref{mthm:fABV} and conclude the existence
  of an ergodic physical/SRB measure supported on $\Lambda$.

  The flow is H\"older-$C^1$ (it is in fact of class $C^2$ as the
  vector field $G$) and so, from standard results, the stable leaves
  form an absolutely continuous lamination; see e.g.~\cite[Section
  6]{ArMel18} \& \cite[Appendix B.7]{fisherHasselblatt12} and
  references therein. This ensures that $\leb(\Omega)$ is positive as
  a consequence of the positivity of $\leb_\Delta(\Delta')$. This
  completes the proof of Corollary~\ref{mcor:SRB-ASH}.
\end{proof}


\section{New examples of application}
\label{sec:new-examples}

Here, we present mostly asymptotically sectional expanding examples
which are either non-sectional hyperbolic or non-singular hyperbolic,
with or without hyperbolic equilibria; as well as counter-examples
failing some of our assumptions and having no physical measure.

\subsection{Mostly asymptotically sectional expanding and
  non-sectional hyperbolic}
\label{sec:mostly-asympt-sectio}

\begin{example}[Mostly asymptotically sectional expanding,
  singular-hyperbolic and not sectional-hyperbolic, with no
  equilibria]
  \label{ex:singhypnosecnoeq} We consider the hyperbolic (Anosov)
  automorphism $f_0$ of the $3$-torus $\TT=(\sS^1)^3$ induced by the
  linear map defined by
\begin{align*}
  A=
  \begin{pmatrix}
    2 & 1 & -1 \\
    1 & 1 & 0 \\
    -1 & 0 & 2
  \end{pmatrix}
  \; \text{with} \;
            \spec(A)=\{\lambda_3\approx 0.198062<1<\lambda_2\approx 1.55496<
            \lambda_1\approx3.24698\}.
\end{align*}
Let $p$ be the fixed point at the class of the origin
$(0,0,0)\in\RR^3$ and a small neighborhood $V$ of $p$ with a choice of
basis $\{v_1,v_2,v_3\}$ where $v_i$ is a unit eigenvector
corresponding to the eigenvalue $\lambda_i, i=1,2,3$. In $V$ the map
$f_0$ has the expression
$(x,y,z)\mapsto (\lambda_1x,\lambda_2y,\lambda_3z)$.  We consider the
one-parameter family of maps of the real line
\begin{align*}
  f_\mu(x)=\psi(x)\lambda_2x+(1-\psi(x))((1-\mu)\lambda_2x+\mu\cdot
  h(x)),
  \quad 0<\mu<1;
\end{align*}
where $h(x)=(1-b) x (1+x^2(x^2-a^2))$ for $0<b<a<\xi<1$ (with small
$b$) has $3$ fixed points at $0,\pm\xi$ satisfying $h'(0)=1-b<1$ and
$h'(\pm\xi)=(1-b)(1\pm\xi(4\xi^2-2a^2))>1$ ; and $\psi:\RR\to[0,1]$ is
a $C^\infty$ bump function so that
$\supp\psi\subset\RR\setminus[-1+b,1-b]$ and $\psi(x)=1$ whenever
$|x|\ge1+b$; see the left hand side of Figure~\ref{fig:bump}.
Moreover, we can also assume that
\begin{align}\label{eq:domf1}
  \lambda_3<f'_\mu(x)<\lambda_1, \quad x\in\RR
  \qand 0<\mu<1.
\end{align}
In addition, since $f'_\mu(0)=\mu h'(0)=\mu(1-b)$ is
the minimum of $f'_\mu(x)$, then
\begin{align}\label{eq:volexp}
  \lambda_1+f'_\mu(x)>1, \quad
  x\in\RR.
\end{align}
We replace the second coordinate map $y\mapsto \lambda_2y$ by the
one-parameter family $y\mapsto \epsilon_1\cdot f_\mu(y/\epsilon_1)$
for $\epsilon_1>0$ small enough so that the ball of radius
$\epsilon_1(1+b)$ around $p$ is contained in $V$, and the properties
stated above are preserved at corresponding points after scaling.
For $\mu=0$ we have the original map $f_0$. 
\begin{figure}[htpb]
  \centering
  \includegraphics[width=3cm]{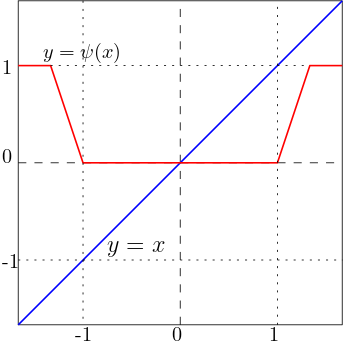}
  \;
  \includegraphics[width=3.5cm]{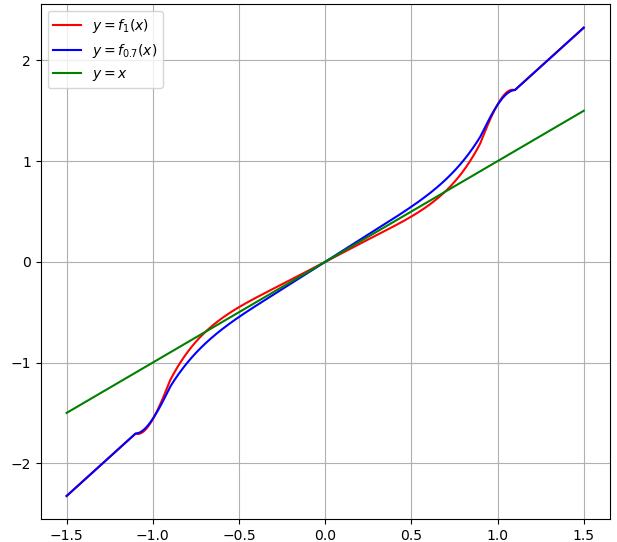}
  \;
  \includegraphics[height=2.5cm]{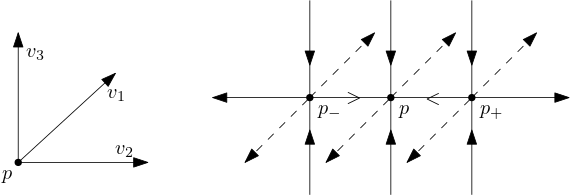}
  \caption{\label{fig:bump} From left to right: the graph of a
    continuous bump function together with $y=x$; graphs of the
    identity, $y=f_1(x)$ and $y=f_{0.7}(x)$; depiction of the eigenvalue
    directions at the origin of $T_pM$; and
    the dynamical behavior of the (un)stable directions in a
    neighborhood of $p$ for $f_1$.}
  \end{figure}

For $\mu=1$ we have a map
$f_1$ coinciding with $f_0$ outside of $V$ and having inside $V$ three
fixed hyperbolic saddle points: $p$ with index $2$; and $p_\pm$ with
index $1$, symmetrically placed with respect to $p$ along the line
segment $[-\epsilon_1(1+b),\epsilon_1(1+b)]v_2$ inside $V$; see the
right hand side of
Figure~\ref{fig:bump}.


We note that $f_1$ has a partially hyperbolic splitting
$E^s\oplus E^{cu}$ defined on all of $\TT$ which is volume hyperbolic:
\begin{itemize}
\item $E^s$ coincides with the stable bundle of $f_0$ spanned
  everywhere by $v_3$ and is uniformly contracted
  $\|Df_1\mid_{E^s}\|=\lambda_3$;
\item $E^{cu}$ coincides with the unstable bundle of $f_0$ spanned
  everywhere by $\{v_1,v_2\}$, domination of the splitting is a
  consequence of~\eqref{eq:domf1}; and
\item $\det (Df_1\mid_{E^{cu}})>1$ as a consequence of~\eqref{eq:volexp}.
\end{itemize}
The invariant bundle $E^{cu}$ further decomposes into the continuous
splitting $E^c\oplus E^u$, where $E^u$ is spanned everywhere by $v_1$
and uniformly expanded: $\|Df\mid_{E^u}\|=\lambda_1$; and $E^c$ is
spanned everywhere by $v_2$ and dominated by $E^u$. We claim that
$f_1$ is non-uniformly expanding, that is
\begin{align}\label{eq:NUEdiff}
  \limsup\nolimits_{n\nearrow\infty}\sum\nolimits_{i=0}^{n-1}\log\|(Df_1\mid_{E^{cu}_{f_1^ix}})^{-1}\|^{1/n} < 0, \quad \m-\text{a.e.} x\in\TT.
\end{align}
But $f=f_1$ is a perturbation of the Anosov automorphism $f_0$ on the
$3$-torus around the fixed point satisfying all the conditions stated
in~\cite[Appendix]{ABV00}, namely:
\begin{enumerate}
\item $f$ admits invariant cone fields $C^{cu}$ and $C^{cs}$, with
  small width containing, respectively, the unstable and stable bundle
  of the Anosov diffeomorphism $f_0$;
\item there are $0<\sigma_2<1<\sigma_1$ and $\delta_0>0$ so that for
  disks $D^{cu}$ and $D^{cs}$ through $x$ tangent, respectively, to
  the centre-unstable cone field $C^{cu}$ and to centre-stable cone
  field $C^{cs}$, we have
  \begin{enumerate}
  \item
    $\min\{|\det(Df\mid_{T_xD^{cu}})|,|\det(Df\mid_{T_xD^{cs}})^{-1}|\}
    >\sigma_1$, for $x\in M$;
  \item
    $\max\{\|(Df \mid_{T_x D^{cu}})^{-1}\|, \|(Df \mid_{T_x D^{cs}})\|\} <
    \sigma_2 $, for $x\in M\setminus V$;
  \item
    $\max\{ \|(Df \mid_{T_xD^{cu}})^{-1}\|, \|(Df \mid_{T_xD^{cs}})\|\} <
    (1+\delta_0)$, for $x\in V$.
\end{enumerate}
\end{enumerate}
Then it follows that $f_1$ satisfies~\eqref{eq:NUEdiff}; see
e.g.~\cite[Appendix]{ABV00} or~\cite[Section 7.6]{Alves2020b}.

We now consider the suspension flow $G$ on a $4$-dimensional manifold
$\wt{\TT}=\TT\times[0,1]/\sim$ of the diffeomorphism $f_1$, where the
equivalence relation is given by $(x,1)\sim(f_1(x),0)$ for $x\in\TT$;
see e.g. \cite[Proposition 3.7]{PM82}. Since $f_1$ can be taken of
class $C^r$ for every $r\ge1$, the same holds for $G$. Moreover,
$\wt{\TT}$ is also a parallelizable manifold as $\TT$ is; thus we can
consider $v_1,v_2,v_3$ as globally defined vector fields transverse to
$G$.

We observe that the flow becomes singular-hyperbolic but not
sectional-hyperbolic: the splitting
$T\wt{\TT}=F^s\oplus (F^c\oplus \RR\cdot G \oplus F^{u})$ where
$F^s, F^c, F^u$ are respectively spanned by $v_1,v_2,v_3$ everywhere
on $\wt{\TT}$ is such that $F^{cu}=F^c\oplus \RR\cdot G \oplus F^{u}$
is volume expanding, since the action of the flow $\phi_t$ of $G$
along $G$ is a translation. However, at the point $p_0=p\times\{0\}$
we have $\det (D\phi_1\mid_{F^c\oplus\RR\cdot G})=1-b<1$, contradicting
sectional-expansion, since $p_0$ belongs to a periodic orbit of $G$
with period $1$.
We claim that this flow satisfies the (MASE) condition.

Indeed, we note that since each submanifold
$\Sigma_s=\TT\times\{s\}, 0\le s<1$ is a global cross-section for the
flow $\phi_t$ with constant return time equal to $1$, then
$\phi_1\mid_{\Sigma_s}:\Sigma_s\to \Sigma_s$ is the Poincar\'e First
Return Map to $\Sigma_s$, and such return maps all coincide with $f_1$
by construction of the suspension flow as a translation on the last
coordinate. In addition, we get $P^1=D(\phi_1\mid_{\Sigma_s})$. Hence,
since $f_1$ is a partially hyperbolic non-uniformly expanding
diffeomorphism, we obtain~\eqref{eq:NUE} for $\m$-a.e. $x\in\Sigma_s$
for each $0\le s<1$. Thus by Fubini's Theorem, we get~\eqref{eq:NUE}
for $\m$-a.e. point of $\wt{\TT}$, because $\{\Sigma_s: 0\le s<1\}$ is
a smooth foliation of $\wt{\TT}$.

We have (NUSE) without singularities, so that (SR) holds automatically
by vacuity. Moreover, since the flow direction makes an angle bounded
away from zero from the global cross-sections $\Sigma_s$, we also
obtain (NU2SE).  We can then ensure the existence of finitely many
ergodic physical/SRB measures for this class of systems from the proof
of Theorem~\ref{mthm:discretefabv} (which just assumes (NUSE)). From
this we obtain (MASE) from Theorem~\ref{mthm:equivalence}.
\end{example}

\begin{example}[Mostly asymptotically sectional expanding, with
  equilibria and not sectionally
  hyperbolic]\label{ex:modifLorenzMultidim}
  We adapt the construction of the multidimensional Lorenz
  attractor, first presented by Bonatti, Pumari\~no and Viana
  in~\cite{BPV97}, to obtain an example of a mostly asymptotically
  sectional expanding attracting set with an equilibrium.
  
  We consider a ``solenoid'' constructed over a uniformly expanding
  map $g:\TT\to\TT$ of the $k$-dimensional torus $\TT$, for some
  $k\ge2$. That is, let $\DD$ be the unit disk on $\RR^2$ and consider
  a smooth embedding $F_0:N\circlearrowleft$ of $N=\TT\times\DD$ into
  itself, which preserves and contracts the foliation
  $\F^s=\big\{\{z\}\times\DD: z\in\TT\big\}$. The natural projection
  $\pi:N\to\TT$ on the first factor conjugates $F_0$ to $g$:
  $\pi\circ F_0=g\circ\pi$. We assume that the initial expanding map $g$
  has simple spectrum $\{\lambda_1>\lambda_2>\dots>\lambda_k\}$ and
  that $F_0$ admits two distinct fixed points $p$ and $q$.

  We have that $DF_0(q):T_q N\to T_q N$ is hyperbolic with a
  $2$-dimensional contracting invariant subspace, and a complementary
  $k$ dimensional expanding invariant subspace. Let
  $\{v_1,\dots,v_k,u_1,u_2\}$ be a basis of $TN=\RR^k\times\RR^2$
  formed by unit vectors so that $v_i$ is an eigenvector corresponding
  to $\lambda_i, i=1,\dots,k$. We choose coordinates on a neighborhood
  $V$ of $q$ in $N$ so that $F_0\mid_V$ has the expression
  $(x,y_1,\dots,y_k)\mapsto (Ax, \lambda_1y_1,\dots,\lambda_ky_k)$
  with $x=x_1u_1+x_2u_2$ and $A$ a linear contraction on $\RR^2$.

  We perform the same perturbation as in
  Example~\ref{ex:singhypnosecnoeq} replacing the weakest expanding
  coordinate map $y_k\mapsto\lambda_k y_k$ by
  $y_k\mapsto \epsilon_1 f_\mu(y_k/\epsilon_1)$ around the fixed point
  $q$, obtaining a new base map $F:N\to N$. We choose $\mu\in(0,1)$ so
  that the $DF(q)$ has eigenvalue $1$ along the $y_k$ direction and
  keeps the expanding/contracting eigenvalues along the other
  directions.
  
  We note that $F$ is a partially hyperbolic map with an invariant
  splitting $E^s\oplus E^c\oplus E^u$, where $E^s=\{0\}\times\RR^2$,
  $E^c=\RR\times v_k$ and $E^u$ is everywhere spanned by
  $v_1,\dots,v_{k-1}$, with
  $\|\wedge^2\big( DF\mid_{E^{cu}_z}\big)^{-1}\|<1$ for
  $z\in N\setminus\{q\}$ with respect to the standard product metric
  in $N$ where $E^{cu}=E^c\oplus E^u$ is central-unstable subbundle;
  and $\|\wedge^2\big( DF\mid_{E^{cu}_q}\big)^{-1}\|=1$. That is, we
  have area expansion along any two-dimensional subspace of the
  central-unstable subbundle away from $q$.  We also have an
  attracting subset $\Lambda_0=\cap_{n\ge0}F^n(N)$ with $N$ as
  topological basin of attraction.

  We further consider the constant vector field $X=(0,1)$ on
  $M=N\times[0,1]$ and modify this field on the cylinder
  $C=U\times\DD\times[0,1]$ around the periodic orbit of the point
  $p=(z,0)\in N\times\{0\}$, where $U$ is a neighborhood of $z$ in
  $\TT$ such that $V\cap (U\times\DD)=\emptyset$, in such a way as to
  create a hyperbolic (generalized Lorenz-like) equilibrium $\sigma$
  of saddle-type with $k$ expanding and $3$ contracting eigenvalues,
  as depicted in Figure~\ref{fig:sing-attractor-4d}. The eigenspace of
  one of the contracting eigenvalues lies along the direction of $X$,
  the other two-dimensional contracting directions still lie on the
  direction of $\DD$, and the remaining expanding eigenspaces are
  transversal to the $X$ direction.
  
\begin{figure}[htbp]
\includegraphics[width=5cm]{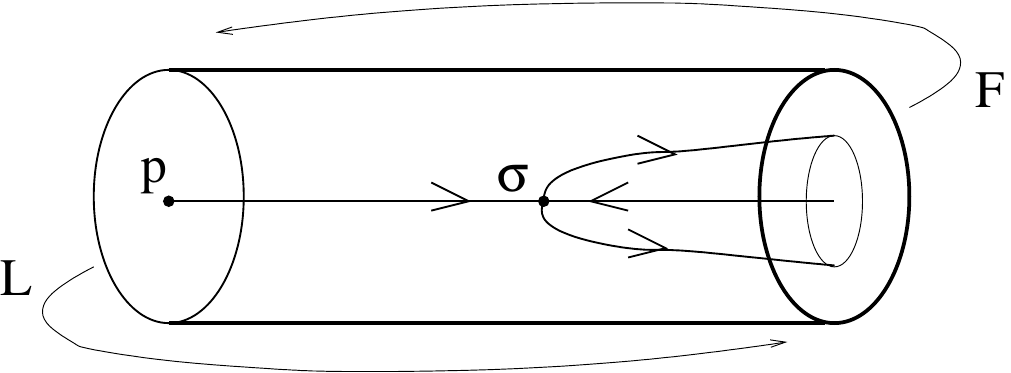}
\caption{A sketch of the modification of the vector field leading to
  the multidimensional Lorenz attractor.}
\label{fig:sing-attractor-4d}
\end{figure}

This vector field $Y$ defines a transition map from
$\Sigma_\epsilon=N\times\{\epsilon\}$ to
$\Sigma_{1-\epsilon}=N\times\{1-\epsilon\}$ for some fixed small
$\epsilon>0$, which is the identity in the first coordinate when
restricted to $\Sigma_\epsilon\setminus (U\times\DD\times\{\epsilon\})$.

We assume that the standard inner product satisfies
$\langle Y, X\rangle>0$ on $\Sigma_\epsilon\cup\Sigma_{1-\epsilon}$
and take a $C^\infty$ bump function $\psi:[0,1]\circlearrowleft$ so
that $\psi\mid_{[\epsilon/2,1-\epsilon/2]}\equiv0$ and
$\psi\mid_{[0,\epsilon/3]\cup[1-\epsilon/3,1]}\equiv1$.  We define the
vector field
$G(x,t)=\psi(t)\cdot X+(1-\psi(t))\cdot Y(x,t), (x,t)\in M$ which
generates a smooth transition map $L$ from
$\Sigma_0^*=(N\setminus\{p\})\times\{0\}$ to $\Sigma_1=N\times\{1\}$.
Together with the identification $(x,0)\sim(F_0(x),1), x\in N$ we
obtain a smooth parallelizable manifold $\wt{M}=M/\sim$ where $G$
induces a $C^\infty$ vector field which we denote by the same letter.

We may assume that the splitting $E^s\oplus E^{cu}$ is still preserved
by $F\circ L$: this is clear outside of the cylinder $C$,
inside $C$ this is obtained by the choice of $Y$ and, moreover, in $C$
the $E^{cu}$ bundle in uniformly expanded.

We may now induce invariant bundles for the flow $\phi_t$ of $G$ on
$\wt{M}$ by parallel transport: $F^{cu}=\RR\cdot G\oplus E^{cu}$ and
$F^s=E^s$ and consider the maximal invariant subset
$\Lambda=\cap_{t>0}\phi_t(\wt{M})$ which is a attracting set with
basin $\wt{M}$. Since $q\in N$ becomes a periodic point with period
$1$ for $G$ and $p\in W^s_\sigma$, we still have area expansion along
$F^{cu}_z$ for $z\neq q$ and non-sectional-expansion along $F^{cu}_q$
for $\phi_t$. But, considering the cone fields $C^{cs}$ and $C^{cu}$
of small width around $F^s, F^{cu}$ respectively, we obtain the
sufficient conditions (1-2) presented in the previous
Example~\ref{ex:singhypnosecnoeq} for non-uniform sectional expansion
of $f=\phi_1$. Condition (NUSE) is obtained again as in
Example~\ref{ex:singhypnosecnoeq}, as well as condition (NU2SE).

Moreover, by construction, there are no asymptotic contracting
directions along the $F^{cu}$ subbundle: we have
condition~\eqref{eq:NNE} from Corollary~\ref{mcor:wABVndcu}. Hence
there exists an ergodic physical/SRB probability measure.  We can also
obtain transitivity, ensuring the uniqueness of this measure and that
$\wt{M}=B(\mu), \leb\bmod0$.
\end{example}

\subsection{Mostly asymptotically sectional expanding and not
  singular-hyperbolic}
\label{sec:mostly-asymtot-secti}

\begin{example}[Geometric Lorenz-like attractor with non-hyperbolic
  periodic orbit]\label{ex:notsinghyp}
  We start with a one-dimensional Lorenz-like transformation with two
  expanding fixed repellers at the boundary of the interval, which is
  an adaptation of the ``intermittent'' Manneville map~\cite{ManPom80}
  into a local homeomorphism of the circle; . We consider $I=[-1,1]$
  and the map $f:I\to I$, $ x\mapsto
    \begin{cases}
      2\sqrt{x}-1 & \text{if $x\ge0$}, \\
      1- 2\sqrt{|x|} & \text{otherwise}.
    \end{cases}
    $; see the left hand side of Figure~\ref{fig:L1drep}.

  Then we perform the geometric Lorenz construction in such a way to
  obtain this map as the quotient over the stable leaves of the
  Poincar\'e first return map to the global cross-section of a vector
  field $G_0$; see the right hand side of Figure~\ref{fig:L1drep}.
  \begin{figure}[htpb]
    \centering \includegraphics[width=3.5cm]{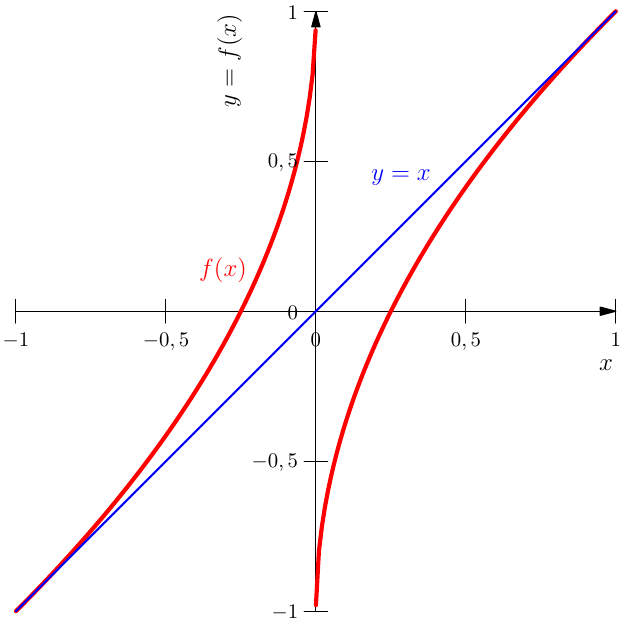}
    \qquad
    \includegraphics[width=4cm]{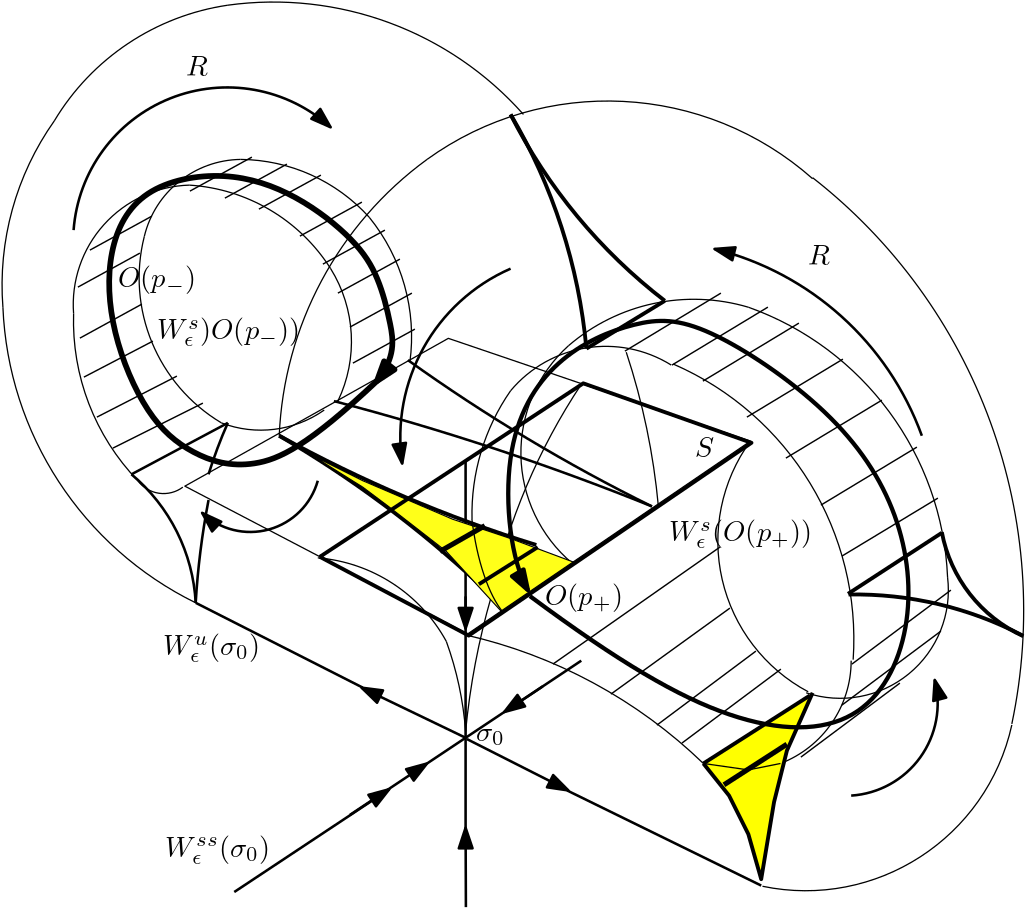}
    \caption{\label{fig:L1drep}Lorenz one-dimensional transformation
      with repelling fixed points at the extremes of the interval on
      the left; and the geometric Lorenz construction with this map as
      the quotient over the contracting invariant foliation on the
      cross-section $S$, with two corresponding periodic saddle-type
      periodic orbits $\cO(p_{\pm})$.}
  \end{figure}

  As usual in the geometric Lorenz construction, we assume that in the
  cube $I^3$ the flow is linear $\dot G_0=A\cdot G_0$ with
  $A=\diag\{\lambda_1,\lambda_2,\lambda_3\}$ and a Lorenz-like
  equilibrium at the origin $\sigma_0$ satisfying
  $\lambda_1<\lambda_3<0<-\lambda_3<\lambda_1$; see e.g. the detailed
  description in~\cite[Chap. 3, Sec. 3]{AraPac2010s}.

  The map $f$ preserves Lebesgue measure $\lambda$ on $I$ which is
  $f$-ergodic; see \cite[Sec. 5]{alves-araujo2004}. In particular, $f$
  is transitive (in fact, it is locally eventually onto, and so
  topologically mixing) and since $|log|x||$ is $\lambda$-integrable,
  then we have slow recurrence to the unique singularity at $0$.

  We thus obtain an attractor $\Lambda$ for the flow of the vector
  field $G$ depicted in the right hand side of Figure~\ref{fig:L1drep}
  which is partially hyperbolic (with $d_s=1$ and $d_{cu}=2$) and
  admits two periodic orbits $\cO(p_{\pm})$ corresponding to the
  indifferent fixed points of $f$ which are not hyperbolic. Indeed,
  the Poincar\'e first return map $R:S^*\to S$ to the cross-section
  $S=I^2\times\{1\}$, with domain
  $S^*=S\setminus\big(\{0\}\times I\times\{1\}\big)$ given by all the
  points of $S$ away from the singular line, is a skew-product map
  $R(x,y)=(f(x), g(x,y))$, where $g$ is a contraction on the second
  coordinate. The non-hyperbolicity of $\cO(p_{\pm})$ ensures that
  \emph{the attractor $\Lambda$ of the $3$-vector field $G$ is not
    singular-hyperbolic.}

  Following the standard construction described in~\cite[Chap. 7,
  Sec. 3.4]{AraPac2010s}, there exists an ergodic physical
  $R$-invariant probability measure $\nu$ on $S$ whose marginal
  $\pi_*\nu$ is $\lambda$, where $\pi:S\simeq I^2\to I$ is the natural
  projection on the first coordinate. Finally, we obtain a physical
  ergodic invariant probability measure $\mu$ for the flow of $G$ by
  considering the suspension flow with base map $R$ and roof function
  provided by the Poincar\'e first return time $\tau:S^*\to\RR^+$ to
  $S$.

  Moreover, $|f'(x)|>1$ for all $x\in I\setminus\{0,\pm1\}$ and so, if
  $(\phi_t)_{t\in\RR}$ is the flow of $G$, then since $\tau$ is
  constant on the fibers of the skew-product and $\lambda$-integrable
  \begin{align*}
    \int\log|\det D\phi_1\mid_{E^c}|\,d\mu\ge
    h_{\mu}(\phi_1)=\frac{h_\nu(R)}{\mu(\tau)}\ge
    \frac{h_\lambda(f)}{\mu(\tau)}
    =\frac1{\mu(\tau)}\int \log|f'|\,d\lambda>0,
  \end{align*}
  we conclude that $\Lambda$ is mostly asymptotically sectional
  expanding (MASE) while not being singular-hyperbolic. From
  Theorem~\ref{mthm:fABV}, these attractors admit a unique
  physical/SRB measure due to transitivity and $\dim E^{cu}=2$.
\end{example}

\begin{example}[Geometric Lorenz-like attractor with non-hyperbolic
  equilibrium]
  \label{ex:nonhypeq}
  In the recent work~\cite{bruin2023mixing} Bruin-Farias construct
  (similarly to the previous example) and study a geometric
  Lorenz-like attractor with a neutral equilibrium replacing the
  hyperbolic Lorenz-like equilibrium from the classical (geometrical)
  Lorenz attractor. This neutral equilibrium is neither Lorenz-like
  nor Rovella-like.

  The authors show that there exists a unique physical/SRB measure and
  proceed to study its mixing rate (obtaining polynomial upper
  bounds). This implies mostly asymptotic sectional expansion without
  singular-hyperbolicity.
\end{example}

\begin{remark}\label{rmk:singularidades}
  Example~\ref{ex:nonhypeq} shows, in particular, that the
  \emph{assumption of hyperbolic equilibria is not necessary for the
    existence of a physical/SRB measure} and so also not necessary to
  obtain asymptotical sectional expansion. Hence, this assumption
  should be regarded as a simplifying general assumption which is used
  in our line of proof.
\end{remark}

\subsection{Non-uniform weak expansion without slow recurrence
  nor physical measure}
\label{sec:non-uniform-section-1}

\begin{example}[Non-uniform expanding and no physical
  measure]\label{ex:nuenophysical}
  We consider the well-known vector field $X$ generating the flow
  $(\phi_t)_{t\ge0}$ of the cylinder $N:=\sS^1\times\RR$ with a
  double heteroclinic connection (the ``Bowen's eye'' flow), e.g.,
  from Takens work~\cite{Ta95} showing that \emph{Birkhoff averages
    may not exist almost everywhere}; see the left hand side of
  Figure~\ref{fig:Bowen}. 
  In this system time averages exist only for the sources $C, D$ and
  for the set of separatrixes and saddle equilibria
  $W=W_1\cup W_2 \cup W_3 \cup W_4 \cup \{A,B\}.$ Moreover the orbit
  $(\phi_t(x))_{t\ge0}$ of each $x$ not in $W$ and different from
  $C, D$ tends to $W$ as $t\nearrow\infty$.

  \begin{figure}[htpb]
    \centering
    \includegraphics[width=7cm]{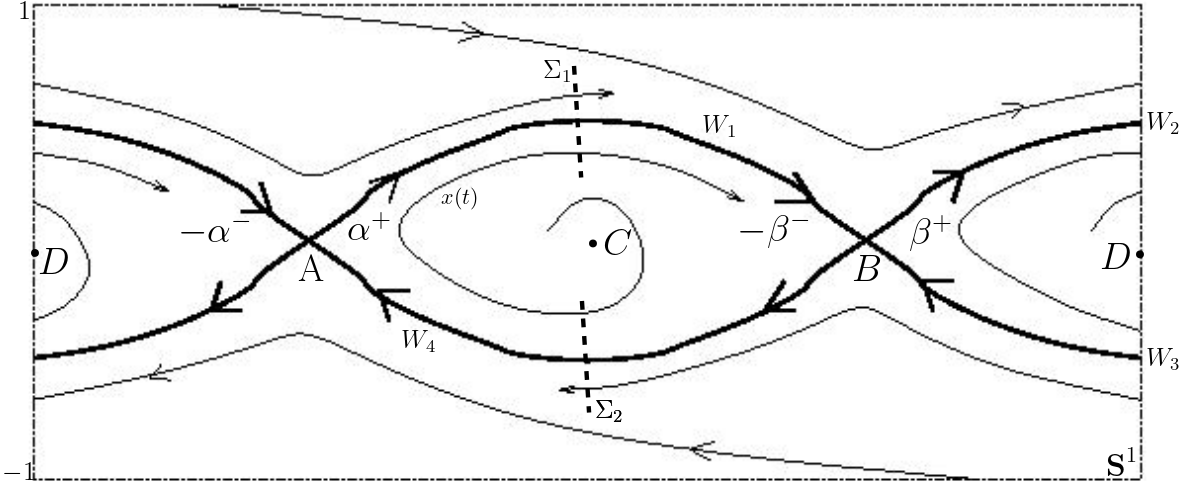}
    \quad
    \includegraphics[width=2.5cm]{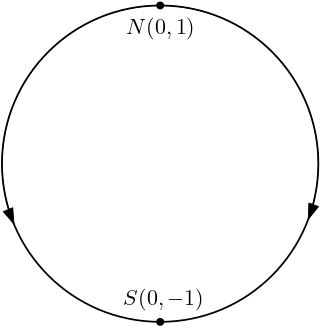}
    \caption{The double heteroclinic connection with non-existing
      time averages for a full Lebesgue measure subset; on the right
      hand side. On the left hand side, the North-South flow on the
      circle.}
    \label{fig:Bowen}
  \end{figure}
  Letting $f:=\phi_1$ denote the time $1$ map of the flow,
  we see that $W$ is a
  compact $f$-invariant attracting set, since $W=\cap_{n\ge0} f^n(U)$
  for all sufficiently small neighborhoods $U$ of $W$.
  Moreover, we may choose the saddles eigenvalues and adapted
  coordinates near $A$ and $B$ to obtain the following for every
  $x\in N\setminus\{C,D\}$
\begin{align}
  \label{eq:secontr}
  \limsup\nolimits_{T\nearrow\infty}\log|\det D\phi_T(x)|^{1/T}<0<
  \limsup\nolimits_{n\nearrow\infty}\sum\nolimits_{i=0}^{n-1}\log\|Df(f^ix)\|^{1/n}.
\end{align}
This shows that this system, although with some average asymptotic
expansion, is asymptotically sectional contracting on an open and full
Lebesgue measure subset -- which shows that these trajectories are not
Oseledets regular; see e.g.~\cite{BarPes2007} and~\cite{CCSV21}.
Since physical measures cannot exist in this system due to the
non-existence of Birkhoff time averages, then we obtain a weak
counterexample to the following conjecture.

\begin{conjecture}{\cite[Conjecture 12.37]{BDV2004}}
  \label{conj:viana}
  If a smooth map $f$ has only non-zero Lyapunov exponents at Lebesgue
  almost every point, then it admits some SRB measure. 
\end{conjecture}
The proof of~\eqref{eq:secontr} is a consequence of the
following (see also Kiriki et al.~\cite{Kiriki2022} and
  Ott-Yorke~\cite{Ott2008}).
\begin{theorem}{\cite[Theorem 1]{Ta95}}
  \label{thm:takens}
  If $g$ is a continuous function on $N$ with $g(A)>g(B)$ and the
  positive trajectory of $x$ accumulates $W$, then
  \begin{align*}
    \limsup_{T\nearrow\infty}\frac1T\int_0^Tg(\phi_tx)\,dt
    &=
      \frac{\sigma g(A) + g(B)}{1+\sigma} \qand
    \liminf_{T\nearrow\infty}\frac1T\int_0^Tg(\phi_tx)\,dt
    =
      \frac{\lambda g(B)+g(A)}{1+\lambda},
  \end{align*}
  where $ \lambda:=\alpha^-/\beta^+$ and $\sigma:=\beta^-/\alpha^+$
  from the spectra $\spec(DX(A))=\{\alpha^+, -\alpha^-\}$ and
  $\spec(DX(B))=\{\beta^+, -\beta^-\}$ with $\alpha^\pm,\beta^\pm>0$.
\end{theorem}
Indeed, to ensure that $W$ is attracting it is enough to have
$\lambda\sigma>1$ and we can set this together with
$\delta_A:=\alpha^+-\alpha^-<0$ and
$\delta_B:=\beta^+-\beta^-<0$. Since
$|\det D\phi_t(x)|=\exp\int_0^t\tr(DX(\phi_sx))\,ds$ we
set 
$g(x)=\tr(DX(x))$ to get
$\log|\det D\phi_T(x)|=\int_0^Tg(\phi_sx)\,ds$ and both
$g(A)=\delta_A$ and $g(B)=\delta_B$ strictly negative. Thus, the left
hand side inequality from~\eqref{eq:secontr} follows from
Theorem~\ref{thm:takens}.

For the right hand side inequality, we set $g(x)=\log\|Df(x)\|$ and
note that $t\mapsto g(\phi_tx)$ is $C^1$. Hence, we can write
$ \int_0^ng(\phi_tx)\,dt
=
\sum_{i=0}^{n-1}\int_0^1g(\phi_t f^ix)\,dt $ and 
$g(\phi_tf^ix)=g(f^ix)+t\cdot \partial_s(g\circ
\phi_sf^ix)\mid_{s=s(t)}$ by the Mean Value Theorem for some
$s(t)\in(0,t)$. Moreover,
\begin{align*}
  \partial_s(g\circ \phi_sf^ix)\mid_{s=s(t)}
  &=
    \nabla g(\phi_sf^ix) \cdot DX(\phi_sf^ix)X(\phi_sf^ix)
\end{align*}
is uniformly bounded from above and below, so we can find $\bar
L$ so that
\begin{align*}
  \int_0^1 g(\phi_tf^ix)\,dt
  \le
  \int_0^1 \big(g(f^ix)+t\bar L\big)\,dt
  \le
  g(f^ix) + \frac12 \bar L.
\end{align*}
This ensures that
$(1/n)\sum\nolimits_{i=0}^{n-1}g(f^ix) \ge
(1/n)\int_0^ng(\phi_tx)\,dt-\bar L/2n$ and so the right hand
inequality of~\eqref{eq:secontr} follows again from
Theorem~\ref{thm:takens}, since for our choice of $g$ we have both
$g(A)=\log \|e^{DX(A)}\|=\alpha^+$ and
$g(B)=\log \|e^{DX(B)}\|=\beta^+$ strictly positive.
\end{example}

\begin{example}[Partially hyperbolic nonuniform sectional expanding
  with no physical measure]
  \label{ex:parthypnuse}
  Continuing from the previous example, we consider the compactification
  $\sS^2$ of $N$ with a source at infinity and the direct product
  $M=\sS^2\times\sS^1$ with the ``North-South flow'' on the circle;
  see right hand side of Figure~\ref{fig:Bowen}.

  We get a flow $(\psi_t:M\to M)_{t\in\RR}$ with an
  attracting set $\cA:=\sS^2\times\{S\}$ so that
  $d(\psi_t(z),\cA)\to0$ when $t\nearrow\infty$ for all
  $z\in M\setminus \cA$, where $d$ is any Riemannian distance on
  $M$. If we let the contraction rate at the sink $S$ of the
  North-South flow to be stronger than the contracting rates of the
  saddles $A,B$ from $(\phi_t)_{t>0}$, then $\cA$ becomes a partially
  hyperbolic attracting set with splitting $T_\cA M= E^s\oplus E^c$
  given by $E^s=\{0\}\times T_S\sS^1$ and $E^c=T\sS^2\times\{0\}$.

  We note that the region between the saddle connections $W_1$ and
  $W_4$ containing $C$ has a closure $F$ which is invariant and
  $K:=F\times V_S$ becomes also a partially hyperbolic forward
  invariant set, where $V_S$ is any compact positively invariant
  neighborhood of the sink $S$ in $\sS^1$ with respect to the
  North-South flow, with the same splitting as above since we have a
  direct product.

  Moreover, because all future trajectories starting in $K$ accumulate
  $W_1\cup W_4\cup\{A,B\}$, from~\eqref{eq:secontr} we obtain
\begin{align}\label{eq:weakSNUE}
  \limsup\nolimits_{T\nearrow\infty}\log|\det D\psi_T\mid_{E^c_x}|^{1/T}
  &<
  0
  <
  \limsup\nolimits_{n\nearrow\infty}\sum\nolimits_{i=0}^{n-1}\log\|D\bar
    f\mid_{E^c_{\bar f^ix}}\|^{1/n}
\end{align}
for all $x\in K\setminus\{C\}\times V_S$ where $\bar f:=\psi_1$. Thus,
for an open and full $\leb$-measure subset of the partially hyperbolic
forward invariant set $K$ we have average asymptotic expansion along
the central bundle together with asymptotic sectional contraction, and
no physical/SRB measure.

Moreover, \emph{we do not have slow recurrence}. Indeed, for any given
$\delta,L>0$ the continuous function
$g(x):=\min\{L, -\log d_\delta(x,\{A,B,C\})\}$ is such that
$\limsup_{T\nearrow\infty}\frac1T\int_0^Tg(\phi_tx)\,dt=L$ since
$g(A)=g(B)=L$ for all $x$ whose future trajectory accumulates $W$, as
a direct consequence of Theorem~\ref{thm:takens}. Hence, for these
trajectories we arrive at
\begin{align*}
  \limsup_{T\nearrow\infty}\frac1T
  \int_0^T -\log d_\delta(\phi_tx,\{A,B,C\})\,dt = +\infty
\end{align*}
for each small $\delta>0$. Analogously, since $\|G(x)\|$ is comparable
to $d_\delta(\phi_tx,\sing(X))$ (see Lemma~\ref{le:boundsigma}) we
obtain the same results replacing the distance to the equilibria with
the norm of the vector field.
\end{example}

\begin{remark}\label{rmk:Vivas}
  The proof of the existence of a physical measure for asymptotic
  sectional hyperbolic attractors presented in \cite{smviv23} --- in
  the case when the attractor contains non-Lorenz-like equilibria
  --- is based on the assumption that \emph{the right hand side
    inequality of~\eqref{eq:weakSNUE} on a positive Lebesgue measure
    subset of points $x\in U$ implies the existence of some physical
    measure}. From Examples~\ref{ex:nuenophysical}
  and~\ref{ex:parthypnuse} we see that the proof in \cite{smviv23} is
  incomplete.
\end{remark}


\bibliographystyle{abbrv}
\bibliography{../bibliobase/bibliography}


\end{document}